\documentclass[11pt,reqno]{article}
\usepackage{graphicx} 
\usepackage[utf8]{inputenc}
\usepackage{amsmath,amssymb,amsthm,mathrsfs,dsfont,bm,amssymb}
\usepackage[a4paper, total={7in, 9.5in}]{geometry}

\usepackage{mathrsfs}

\usepackage{titlesec,hyperref}

\usepackage{color}
\usepackage{fancyhdr}
\usepackage{tikz}
\usepackage{float}
\usepackage[]{cleveref}
\usetikzlibrary{positioning} 
\usepackage{xcolor}
\usepackage{mathtools}
\usepackage[title, titletoc]{appendix}

\usepackage{abstract}
\newcommand{\beq}{\begin{equation}}
	\newcommand{\eeq}{\end{equation}}
\newcommand{\beqs}{\begin{equation*}}
	\newcommand{\eeqs}{\end{equation*}}
\newtheorem{theo}{Theorem}[section]
\newtheorem{lemm}[theo]{Lemma}

\newtheorem{prop}[theo]{Proposition}
\newtheorem{rema}[theo]{Remark}

\numberwithin{equation}{section}
\newcommand\N{{\mathbb N}}
\newcommand\R{{\mathbb R}}
\newcommand\T{{\mathbb T}}
\newcommand\Z{{\mathbb Z}}

\newcommand\DD{{\mathbb D}}
\newcommand\PP{{\mathbb P}}

\newcommand\D{{\mathcal D}}
\newcommand\F{{\mathcal F}}
\newcommand\E{{\mathcal E}}

\newcommand\W{{\mathcal W}}
\newcommand\A{\mathcal{A}}

\newcommand\HH{{\mathcal H}}

\newcommand\x{\bm{x}}

\newcommand\up{\upsilon}

\newcommand\uu{\bm{u}}
\newcommand\vv{\bm{v}}

\newcommand\U{\bm{U}}

\newcommand \tri {\coloneqq}

\let\pa=\partial
\let\na=\nabla
\let\al=\alpha

\let\var=\varepsilon

\let\f=\frac

\let\om=\omega
\let\ep=\epsilon
\def \pt {\partial_t}
\def\dive{\mathop{\rm div}\nolimits}

\def\exp{\mathop{\rm exp}\nolimits}

\def\cF{{\mathcal F}}
\def \cR {\mathcal{R}}
\newcommand{\Rey}{\mathrm{Re}}

\allowdisplaybreaks[4]

\begin{document}
	
	\title{Transition threshold for the Navier-Stokes-Coriolis system at high Reynolds numbers}
 
	\author{Minling Li, Changzhen Sun, Chao Wang, Dongyi Wei, Zhifei Zhang}
	\author{Minling Li \textsuperscript{1}\thanks{limling@pku.edu.cn}
		\quad Changzhen Sun \textsuperscript{2}\thanks{changzhen.sun@cnrs.fr}
		\quad Chao Wang \textsuperscript{1}\thanks{wangchao@math.pku.edu.cn} 
		\quad Dongyi Wei
\textsuperscript{1}\thanks{jnwdyi@pku.edu.cn} 
		\quad Zhifei Zhang \textsuperscript{1}\thanks{zfzhang@pku.edu.cn}\\[10pt]
		\footnotesize\textsuperscript{1}{School of Mathematical Sciences, Peking University}\\
		\footnotesize {Beijing 100871, P. R. China}\\
		\footnotesize\textsuperscript{2}{Laboratoire de Math\'ematiques de Besancon (UMR 6623),
			Universit\'e de Marie-et-Louis-Pasteur, CNRS}\\
		\footnotesize {F-25000 Besancon, France}\\
		[5pt]}

	\date{}

	\maketitle

\begin{abstract}
        The transition mechanism from laminar flow to turbulent flow is a central problem in hydrodynamic stability theory. To shed light on this transition mechanism, Trefethen et al.({\it \small Science 1993}) proposed the transition threshold problem, aiming to quantify the magnitude of perturbations required to trigger instability and determine their scaling with the Reynolds number. In this paper, we investigate the transition threshold of Couette flow for the three-dimensional incompressible Navier-Stokes-Coriolis system in the high Reynolds number regime ($\Rey\gg 1$). By exploiting the combined effects of  rotation (dispersion) and mixing mechanisms, we derive an improved stability threshold scaling in $\Rey.$\,\,
		Precisely, we show that if the initial perturbation satisfies
		$$\|\vv_{in}-(y, 0, 0)\|_{\tilde{H}(\T\times\DD)}\leq   \ep_0 \,\Rey^{-\alpha},$$ 
		with any $\alpha>\frac 23$ and $\tilde{H}=H^6\cap W^{3,1}$ for $\DD=\mathbb{R}^2$, and with any $\alpha\geq\frac 56$ and $\tilde{H}=H^6$ for $\DD=\mathbb{R}\times\mathbb{T}$,
		the corresponding solution of the Navier-Stokes-Coriolis system exists globally in time and remains asymptotically close to the Couette flow. The main analytical challenge arises from the anisotropic nature of the estimates for the zero modes and from the interactions between zero and non-zero modes, which we address using an anisotropic Sobolev space directly tailored to the zero modes. Additionally, we introduce a new dispersive structure for the zero modes and derive suitable Strichartz-type estimates. These tools enable us to exploit both the nonlinear structure and the improved dispersive behavior of certain good components of the zero modes, which play a crucial role in achieving the improved stability threshold.
        
        \vspace{1em} 
        
\end{abstract}
	
	\setcounter{tocdepth}{2}
	
	\maketitle

	\section{Introduction}
	
	In this paper, we consider the following Navier-Stokes-Coriolis system:
	\begin{align}\label{1-eq0}
		\left\{\begin{array}{l}
			\displaystyle
			\partial_{t}\vv + \vv\cdot \nabla_{\x_1} \vv+\beta \vec{e}_3 \times \vv-\nu \Delta_{\x_1}\vv+\nabla_{\x_1} p^N=0,
			\\[1ex]
			\displaystyle
			\dive_{\x_1} \vv=0,
		\end{array}\right.
	\end{align}
	where $t\in\R^+$, $\x_1=(x_1,y,z)\in \T\times\DD$, the domain $\DD=\R^2$ or $\R\times\T$.
	The unknown vector field $\vv=(v^1,v^2,v^3)^{T}$ describes the incompressible fluid velocity, while $p^N$ denotes the scalar pressure field.
	Our coordinate system is chosen such that the axis of rotation is  $\vec{e}_3=(0,0,1)$.
	The parameter $\beta\in\R$ characterizes the rotation rate, which reflects in the Coriolis force $\beta \vec{e}_3 \times \vv$, and $\nu=\Rey^{-1}>0$ represents the viscosity.
	The system \eqref{1-eq0} is supplemented with initial data:
	\begin{align*}
		\vv|_{t=0}=\vv_{in}.
	\end{align*}
	It is easily to verify that system \eqref{1-eq0} has the following Couette flow profile stationary solution
	\begin{align*}
		\bar{\vv}=(y,0,0)^{T},\quad \pa_{y}\bar{p}^{N}=-\beta y,
	\end{align*}
	Here, we focus on the nonlinear stability analysis around the Couette flow at high Reynolds number regime. Thus, we introduce the perturbation velocity field:
	\begin{align*}
		\uu\tri \vv-\bar{\vv},
	\end{align*}
	which leads to the following perturbation system
	\begin{align}\label{1-eq1}
		\left\{\begin{array}{l}
			\displaystyle
			\partial_{t}\uu + y\partial_{x_1}\uu+((1-\beta)u^2,\beta u^1,0)^{T}-\nu \Delta_{\x_1}\uu+\nabla_{\x_1} p^L= -\uu\cdot \nabla_{\x_1} \uu-\nabla_{\x_1} p^{NL},
			\\[1ex]
			\displaystyle
			\dive_{\x_1} \uu=0,\\
			\displaystyle \uu|_{t=0}=\uu_{in},	
		\end{array}\right.
	\end{align}
	where the linear and nonlinear pressure terms are given respectively by:
	\begin{align*}
		\Delta_{\x_1} p^L =-2\partial_{x_1} u^2 +\beta (\partial_{x_1} u^2-\partial_y u^1),\quad
		\Delta_{\x_1} p^{NL} =-\dive_{\x_1} (\uu\cdot \nabla_{\x_1} \uu).
	\end{align*}
	
	\medskip
	
	The Coriolis effect acts as a stabilizing mechanism in the fluid  systems. However, its influence can vary significantly due to changes in the parameter $\beta$. The physical quantity to quantify the effects of rotation is the so-called Bradshaw-Richardson number \cite{Bradshaw-1969,Huang-2018}:
	\begin{align*}
		B_{\beta}\tri \beta(\beta-1)\in [-\f14,+\infty),
	\end{align*}
	which is also termed the Pedley number  in some literature \cite{Leblanc-Cambon-1997,Pedley-1969}. The value of the $B_\beta$ will decide the  stability or instability of system \eqref{1-eq1}.

	\medskip

	The dynamics of solutions in system \eqref{1-eq1} differ fundamentally from those that are independent of the periodic $x_1$-variable, see Theorems \ref{theo1}-\ref{theo2} below.
	To account for this, we use the frequency decomposition, i.e., for any function $f$, 
	\begin{align*}
		\PP_0 f \tri \frac{1}{2\pi} \int_{\T}f(x_1,y,z) \,dx_1,
		\quad 
		\PP_{\not=} f\tri f-\PP_0 f,
	\end{align*}
	which we respectively refer to as the zero and non-zero frequencies of $f$.
	In the sequel, to simplify the notations, we sometimes use the $(f_0, f_{\not=})$ instead of $(\PP_{0} f,\PP_{\not=} f)$.

	Applying $\PP_0$ to system \eqref{1-eq1}, we have
	\begin{align}\label{1-eq2}
		\left\{\begin{array}{l}
			\displaystyle
			\partial_{t} u_0^1 + (1-\beta)u_0^2-\nu (\partial_y^2+\partial_z^2)u_0^1
			= -\PP_0(\uu\cdot\nabla_{\x_1} u^1),
			\\[1ex]
			\displaystyle
			\partial_{t} u_0^2 - \beta \cR_3^2	
		u_0^1-\nu (\partial_y^2+\partial_z^2)u_0^2
			= -\PP_0(\uu\cdot\nabla_{\x_1}  u^2)-\PP_0\partial_y p^{NL},
			\\[1ex]
			\displaystyle
			\partial_{t} u_0^3 + \beta \cR_2\cR_3
			u_0^1-\nu (\partial_y^2+\partial_z^2)u_0^3
			= -\PP_0(\uu\cdot\nabla_{\x_1}  u^3)-\PP_0\partial_z p^{NL},
			\\[1ex]
			\displaystyle
			\partial_y u_0^2+\partial_z u_0^3=0,
		\end{array}\right.
	\end{align}
	where $t\in \mathbb{R}^+$, $(y,z)\in\mathbb{R}^2$ or $\mathbb{R}\times \T$, the operators $\mathcal{R}_2 \tri \partial_y(-\pa_y^2-\pa_z^2)^{-\f12}$ and $\mathcal{R}_3 \tri \partial_z(-\pa_y^2-\pa_z^2)^{-\f12}$.

	\subsection{Background and known results}
	
	The stability of plane Couette flow has been studied extensively since the pioneering works of Rayleigh \cite{Rayleigh} and Kelvin \cite{Kelvin}. It is well established that plane Couette flow is linearly stable for all Reynolds numbers, as demonstrated in \cite{Drazin-Reid} and \cite{Romanov}.  However, the flow can exhibit nonlinear instability and transition to turbulence under small but finite perturbations at high Reynolds numbers—a phenomenon known as the Sommerfeld paradox. One approach to resolving this paradox involves studying the transition threshold problem, a concept first rigorously formulated by Trefethen et al. \cite{Trefethen}.  In \cite{Bedrossian-Germain-Masmoudi-2017, Bedrossian-Germain-Masmoudi-2019}, Bedrossian, Germain, and Masmoudi established a comprehensive mathematical framework for this problem, which can be stated as follows:
	
	{\it Given a norm $\|\cdot\|_{X}$, find a $\alpha_1=\alpha_1(X)$ so that
		\begin{align*}
			&\|u_0\|_{X} \leq \Rey^{-\al_1} \Rightarrow \mathrm{stability},\\
			&\|u_0\|_{X} \geq \Rey^{-\alpha_1}\Rightarrow \mathrm{instability}.
		\end{align*}
	}
	The exponent $\alpha_1$ is referred to as the transition threshold. 
	
	Over the past decade, significant progress has been made in obtaining rigorous mathematical results for this problem. 
    The threshold $\alpha_1$ in the Navier-Stokes equations exhibits a strong dependence on dimension, domain, function space of the perturbation, and boundary conditions.
    For the 2-D case on the domain $\T\times\R$, it is known that $\alpha_1\leq 0$ for the perturbations in Gevrey class \cite{Bedrossian-Masmoudi-Vicol-2016, Li-Masmoudi-Zhao} and $\alpha_1\leq \frac13$ for the perturbations in Sobolev class \cite{Bedrossian-Wang-2018, MZ, WZ-2023}. 
    On the bounded domain $\T\times [0, 1]$, the results further depend on the boundary conditions: under slip boundary conditions, $\alpha_1\leq \frac13$ for the perturbations in Sobolev class \cite{WZ-2025}, while $\alpha_1\leq 0$ for the perturbations in Gevrey class \cite{Bedrossian-He-Iyer-Wang-2024}; under no-slip boundary conditions, $\alpha_1\leq \frac12$ for the perturbations in Sobolev class \cite{CLWZ}. Additional results on the stability of other kinds of 2-D shear flow
   (such as Poiseuille flow) can be found in \cite{Coti-Elgindi-Widmayer-2020, Ding-Lin-2022, Zotto-2023, CWZ-CMP} and the references therein. 
 The 3-D case (i.e., system \eqref{1-eq1} with $\beta=0$) is considerably more challenging due to the lift-up effect. 
    For this setting, the current best-known upper bound is $\alpha_1\leq 1$, which has been established for the perturbations in Gevrey class on domain $\T\times\R\times\T$, see \cite{Bedrossian-Germain-Masmoudi-2020, Bedrossian-Germain-Masmoudi-2022}, Sobolev class on $\T\times\R\times\T$, see 
     \cite{Bedrossian-Germain-Masmoudi-2017, WZ-cpam} and $\T\times[0, 1]\times\T$ with non-slip boundary, see \cite{CWZ-MAMS}.
    Here, the results from \cite{CWZ-MAMS, WZ-cpam} confirm the transition threshold conjecture proposed by Trefethen et al. in \cite{Trefethen}.  For more results on the other model on the stability of Couette flow (for example MHD), see \cite{Liss-2020, Rao-Zhang-Zi-2025} and therein.

Determining the threshold $\alpha_1<1$ for the 3-D Navier-Stokes equations remains a challenging problem. One of the primary difficulties arises from the zero-frequency modes of the velocity field. Specifically, the linearized system for the zero modes $\uu_0$ satisfies the following heat equations:
\begin{align*}
\pa_t \uu_0-\nu\Delta \uu_0+(u_0^2, 0, 0)=0,
\end{align*}
where the term $(u_0^2, 0, 0)$ is called the lift-up term as it brings linear growth. 
To address the lift-up effect, the works of \cite{Bedrossian-Germain-Masmoudi-2017, CWZ-MAMS, WZ-cpam}  
establish the following key estimate:
\begin{align}\label{est-NS}
\|u^1_0\|_{H^2} \leq C\min\{1+t, \nu^{-1}\} \|\uu_{in}\|_{H^3},
\end{align}
which introduces either a $\nu^{-1}$ loss or a $(1+t)$ growth. This behavior significantly impacts the transition threshold $\alpha_1$, suggesting that stability for the 3-D Navier-Stokes equations requires at least $\alpha_1\geq 1$. To potentially lower the bound on $\alpha_1$, a natural starting point is the study of the equations governing the streak solutions. 

Alternatively, improving the transition threshold could involve identifying new dynamical mechanisms. A pioneering result in this direction was achieved by Coti Zelati, Del Zotto, and Widmayer \cite{Coti-Zotto, Coti-Zotto-Widmayer}, who demonstrated that the 3-D Boussinesq system 
on $\T\times\R\times\T$
exhibits a dispersive effect, capable of suppressing the lift-up effect. Their analysis establishes the following key dispersive estimate for the linear system:
\begin{align*}
\| u^2_0 \|_{L^{\infty}} +\| u_0^3-\bar{u}^3_0 \|_{L^{\infty}}+\| \theta_0-\bar{\theta}_0 \|_{L^{\infty}} \leq C (1+t)^{-\frac13} e^{-\nu t} \|(\uu_{in}, \theta_{in})\|_{W^{6,1}},
\end{align*}
where the double zero-frequency components are defined by $\bar{f}_0=\int_\T f_0\,dz$ with $f\in\{u^3, \theta\}$. 
And $f_0-\bar{f}_0$ is referred to as the simple zero-frequency components.
The additional $(1+t)^{-\frac13}$  decay rate arises from a dispersive effect, a novel phenomenon absent from the 3-D Navier-Stokes equations.
This improved decay estimate provides enhanced control over the nonlinear terms in the system for simple zero-frequency  modes. Building upon these estimates, the authors in \cite{Coti-Zotto-Widmayer} were able to establish an improved transition threshold of $\alpha_1\leq \frac{11}{12}$.


In contrast to the classical Navier-Stokes equations, the Navier-Stokes-Coriolis system \eqref{1-eq0} not only represents  a fundamental physical model,  but also exhibits stabilizing effects induced by the Coriolis force. However, the nature of these effects depends critically on the parameter $\beta$. The work of Coti Zelati, Del Zotto, and Widmayer \cite{Coti-Zotto-Widmayer-2025} provides a comprehensive analysis of the linearized system, revealing three distinct regimes: the system exhibits lift-up stability for $B_\beta=0$; exponential instabilities emerge for $B_\beta<0$; a dispersive structure develops  for $B_\beta>0$. 
Huang, Sun, and Xu \cite{Huang-Sun-Xu, Huang-Sun-Xu-1} found that the dispersive structure exhibited by the simple zero-frequency components $\uu_0$ suppresses the lift-up effect, and they established the 
enhanced decay estimate:
\begin{align*}
\| \uu_0-\bar{\uu}_0\|_{W^{2,\infty}} \leq C (1+t)^{-\frac13} e^{-\nu t} \|\uu_{in}\|_{W^{6,1}},
\end{align*}
which represents a notable improvement over the corresponding estimates for the 3-D Navier-Stokes equations (cf. \eqref{est-NS}). This dispersive behavior ensures that the zero-frequency components remain within the same order of magnitude as the initial data, thereby preventing secondary flow instabilities in system \eqref{1-eq0}. Based on this analysis, they get the transition threshold $\alpha_1\leq 1$ when $B_\beta>0$ in \cite{Huang-Sun-Xu} and $\alpha_1\leq 2$ when $\beta=1$ in \cite{Huang-Sun-Xu-1}. 
Recent advancements by Coti Zelati, Del Zotto, and Widmayer \cite{Coti-Zotto-Widmayer-2025} have further refined these thresholds: $\alpha_1\leq \frac 89$ for $B_\beta>0$ and $\alpha_1\leq \frac 56$ for $B_\beta\gtrsim\nu^{-1}$ when the initial double zero modes equal to zero.

\subsection{Main results}

For any $s\in\N^+$, we say that $f(x_1,y,z) \in \HH^s(\T\times\DD)$ whenever
\begin{align*}
	\|f\|_{\HH^s(\T\times\DD)}^2\tri \sum_{l=0}^{s}\int_{\T\times\DD}|\nabla^l f|^2\,dx_1dydz<+\infty,
\end{align*}
where $\nabla\tri (\pa_{x_1},\pa_y+t\pa_{x_1},\pa_{z})$, which refers to as the ``good derivative". We also observe that 
\begin{align}
	\PP_0\nabla^s f=\PP_0\nabla_{\x_1}^s f,\quad \|\PP_0f\|_{\HH^s(\DD)}^2=\|\PP_0f\|_{H^s(\DD)}^2.
\end{align}

Firstly, we consider the domain $\DD=\R^2$, and our main result can be stated as:
\begin{theo}[Nonlinear stability on $\T\times\R^2$]
	\label{theo1}
	Let $0<\nu\leq 1$ and $B_{\beta}>0$.
	There exists a constant $\ep_0>0$ 
	such that for any $\alpha>\f23$ and any divergence-free initial data $\uu_{in}$ satisfying
	\begin{align}\label{initial data-assume}
		\var_0\tri \|(1+\pa_z)\uu_{in}\|_{H^5(\T\times\R^2)\cap W^{2,1}(\T\times\R^2)}
		\leq \ep_0 \nu^{\alpha},
	\end{align}
	then the system \eqref{1-eq1} admits a unique global-in-time solution $\uu$ satisfying, for all $t\geq0$, 
	\begin{align*}
		\|(1+\pa_z)\uu\|_{L_t^{\infty}(\HH^4(\T\times\R^2)) }
		+	\nu^{\f12}	\|\nabla_{\x_1}(1+\pa_z)\uu\|_{L_t^2(\HH^4(\T\times\R^2))}
		+\nu^{\f16}	 \|(1+\pa_z)\uu_{\not=}\|_{L^2_t(\HH^4(\T\times\R^2))}  
		\leq C\var_0,
	\end{align*}
	where $C>0$ is a constant independent of $t$ and $\nu$.
\end{theo}

A more detailed description of the long-time dynamics of the solution can be obtained.
\begin{theo}[Asymptotic behavior on $\T\times\R^2$]
\label{theo2}
Suppose the assumptions of Theorem \ref{theo1} hold
 and that $u_{0in}\in W^{3^{+},1}.$ Then the zero-frequency component of the solution satisfies the following dispersive estimates:
	\begin{align}
    \label{Asy-u0-1}		&\|\uu_0(t)\|_{L^\infty(\R^2)}\lesssim (1+t)^{-\f12} \|\uu_{0in}\|_{W^{2^+,1}(\R^2)}+\ep_1^2\nu^{2\al-\f12-\delta_1},
        \\
    \label{Asy-u0-2}
        &\|( u_0^3,\nabla u_0^3,\pa_y\uu_0)(t)\|_{L^\infty(\R^2)}\lesssim (1+t)^{-1} \|\uu_{0in}\|_{W^{3^+,1}(\R^2)}+\ep_1^2\nu^{2\al-\delta_1},
	\end{align}
	where $\delta_1>0$ is a small constant.
	
	For the non-zero frequencies, we have the following inviscid damping and enhanced dissipation estimates:
	\begin{align}\label{Asy-unot}
		\|(u_{\not=}^1,u_{\not=}^3)(t)\|_{L^2(\T\times\R^2)}
		+(1+t)\| u_{\not=}^2\|_{L^2(\T\times\R^2)}
		\leq C \var_0 e^{-\delta_2 \nu^{\f13}t},
	\end{align}
	where $\delta_2>0$ is a small constant.
\end{theo}

\begin{rema}
	In contrast to the results established in \cite{Coti-Zotto-Widmayer-2025}, our analysis achieves comparable (or improved due to better dispersion) transition thresholds without requiring further constraints on the double zero modes, which are no longer appropriate in the current setting $\mathbb{T}\times \R^2$. 
	The main observation to avoid such assumption is that we should try to propagate the regularity of zero modes (in $x_1$) $\uu_0$ directly rather than decompose zero modes as double zero  (in both $x_1$ and $z$) and simple zero (zero in $x_1$ but non-zero in $z$) modes.
	To this end, we establish energy estimates of zero modes (in $x_1$) in an anisotropic Sobolev space, whose regularity is consistent with that of non-zero modes.
	See Section \ref{sec-strategy} for more details. 
\end{rema}

	


Our method remains applicable to the case of $\DD=\R\times \T$.
More precisely, we can obtain the following two theorems.

\begin{theo}[Nonlinear stability on $\T\times\R\times\T$]
	\label{theo3}
	Let $0<\nu\leq 1$ and $B_{\beta}>0$.
	There exists a constant $\ep_0>0$ 
	such that for any $\alpha \geq \f56$ and any divergence-free initial data $\uu_{in}$ satisfying
	\begin{align}\label{initial data-assume'}
		\var_0\tri \|(1+\pa_z)\uu_{in}\|_{H^5(\mathbb{T}\times \R\times \T)}
		\leq \ep_0 \nu^{\alpha},
	\end{align}
	then the system \eqref{1-eq1} admits a unique global-in-time solution $\uu$ satisfying, for all $t\geq0$, 
	\begin{align*}
		\|(1+\pa_z)\uu\|_{L_t^{\infty}(\HH^4(\mathbb{T}\times \R\times \T)) }
		+	\nu^{\f12}	\|\nabla_{\x_1}(1+\pa_z)\uu\|_{L_t^2(\HH^4(\mathbb{T}\times \R\times \T))}
		+\nu^{\f16}	 \|(1+\pa_z)\uu_{\not=}\|_{L^2_t(\HH^4(\mathbb{T}\times \R\times \T))}  
		\leq C\var_0,
	\end{align*}
	where $C>0$ is a constant independent of $t$ and $\nu$.
\end{theo}

\begin{theo}[Asymptotic behavior on $\T\times\R\times\T$]
	\label{theo4}
	Under the assumptions in Theorem \ref{theo3}, for the zero frequencies of the solution, we have the following dispersive estimate:
	\begin{align}
    \label{Asy-u0-1'}
		&\|(u^2_0,\nabla u^2_0,\pa_z\uu_0)(t)\|_{L^\infty(\R\times \T)}\lesssim (1+t)^{-\f13} \|\uu_{0in}\|_{W^{{3}^+,1}(\R\times \T)}+\ep_1^2\nu^{2\al-\f23}.
	\end{align}
	
	For the non-zero frequencies, we have the following inviscid damping and enhanced dissipation:
	\begin{align}\label{Asy-unot'}
		\|(u_{\not=}^1,u_{\not=}^3)(t)\|_{L^2(\mathbb{T}\times \R\times \T)}
		+(1+t)\|u_{\not=}^2(t)\|_{L^2(\mathbb{T}\times \R\times \T)}
		\leq C \var_0 e^{-\delta_2 \nu^{\f13}t},
	\end{align}
	where $\delta_2>0$ is a small constant.
\end{theo}

\subsection{Obstacles \& strategies}	
\label{sec-strategy}
In this subsection, we present the general framework, highlight the main obstacles, and outline the strategies for closing the estimates with the smallest possible thresholds. Here, we only outline the main ideas for the case $\T\times\R^2$. The case $\T \times \R \times \T$ can be addressed by employing a similar method.

\subsubsection{Non-zero modes}

We consider the non-zero modes system. One of main analytical challenge in the study of non-zero modes stems from the nonlinear interaction terms between zero and non-zero modes. 
To formulate this precisely, we introduce the moving frame
\begin{align*}
t= t,\quad x\tri x_1-ty,\quad y= y,\quad z=z
\end{align*}
and denote $\U(t,x,y,z)=\uu(t,x+ty,y,z).$
We then consider the following coupled system 
\begin{align*} 
	\left\{\begin{array}{l}
		\pa_t Q +C_{\beta} \pa_{z} |\widetilde{\nabla}| ^{-1}W-\nu \widetilde{\Delta}Q
		=-\tilde{\pa}_y \widetilde{\Delta} P^{NL}-\widetilde{\Delta}(\U\cdot\widetilde{\nabla}U^2),
		\\[1ex]
		\pa_t W+C_{\beta} \pa_{z} |\widetilde{\nabla}| ^{-1}Q
		-\pa_x\tilde{\pa}_y |\widetilde{\nabla}| ^{-2}W
		-\nu \widetilde{\Delta}W
		=\sqrt{\f{\beta}{\beta-1}} |\widetilde{\nabla}| 
		\big(
		\pa_{z}(\U\cdot\widetilde{\nabla}U^1)-\pa_{x}(\U\cdot\widetilde{\nabla}U^3)
		\big),
	\end{array}\right.
\end{align*}
where $\tilde{\pa}_y =\pa_y-t\pa_x, \, \widetilde{\nabla}=(\pa_x, \tilde{\pa}_y, \pa_z)^t, \widetilde{\Delta}=\pa_x^2+\tilde{\pa}_y^2+\pa_z^2$ and 
\begin{align*}	
Q\tri \widetilde{\Delta} U^2,	\quad	W\tri \sqrt{\f{\beta}{\beta-1}} |\widetilde{\nabla} |(\pa_{z} U^1-\pa_{x} U^3).
\end{align*}
We observe that the equation for $W$ contains a stretching term. And the presence of rotation terms results in a coupled system of equations for $Q$ and $W$.
This coupling leads to an growth in $Q$, which consequently introduces new difficulties for lowering the transition threshold.

In  the process of propagating high Sobolev regularity norm (say $H^4$) of $(Q_{\neq}, W_{\neq})$, we need to deal with the following nonlinear term:
\begin{align} 
	&\Big|\int_0^t\int_{\T\times\R^2}	
    \nabla^4\PP_{\not=}\A\tilde{\pa}_y\widetilde{\dive}(U\cdot\widetilde{\nabla} U)
    \cdot
    \nabla^4\A Q_{\not=}\,dxdydz d\tau \Big|
	\nonumber\\
	\lesssim& \int_0^t\int_{\T\times\R^2}	|\nabla^4\A(\pa_z U_{0}^2\widetilde{\nabla}U_{\not=}^3)| |\nabla^4\A\widetilde{\nabla}Q_{\not=}|\,dxdydz d\tau
    \nonumber\\
    &
    +\Big|\int_0^t\int_{\T\times\R^2} \A\tilde{\pa}_y\nabla^4(\pa_{i}U_{\not=}^2\tilde{\pa}_yU_{\not=}^i)\, \nabla^4\A Q_{\not=}\,dxdydzd\tau \Big|
    +\cdots,
    \label{sec1:est-en-non}
\end{align}
here, $\A$ is a time dependent Fourier multiplier type weights that is defined in \eqref{def-Af}. And the notation $'\cdots'$ denotes terms that can be readily handled, whereas the first two terms (which are called as the reaction terms) on the right-hand side of \eqref{sec1:est-en-non} require special attention.
We now shift our analysis to these two reaction terms.

$\bullet$ The first term in the right-hand side of \eqref{sec1:est-en-non} can be bounded by
\begin{align} 
	\int_0^t\int_{\T\times\R^2}	|\nabla^4\A(\pa_z U_{0}^2\widetilde{\nabla}U_{\not=}^3)| |\nabla^4\A\widetilde{\nabla}Q_{\not=}|\,dxdydz d\tau
	\lesssim 	\|\pa_{z}U_{0}^2 \|_{L^p_tH^5}  \|\A \widetilde{\nabla}U_{\not=}^3\|_{L^{\f{2p}{p-2}}_tH^4} \|\A\widetilde{\nabla}Q_{\not=}\|_{L^2_tH^4},
	\label{intro-1}
\end{align}
where $p=2$ or $+\infty$.  

To close the energy estimates, we need to control zero modes $\|\pa_{z}U_{0}^2 \|_{L^p_tH^5}$. Similarly, when dealing with the term like 
\begin{align*}
\int_0^t\int_{\T\times\R^2}
|\nabla^4\A\pa_{4-j}(\U_{0}\cdot\widetilde{\nabla}  U^j_{\not=})| |\nabla^4\A\widetilde{\nabla}W_{\not=}|\,dxdydzd\tau \qquad \,(j=1,3),
\end{align*}
 we need to bound $\|\pa_{z}U_{0}^{1,3} \|_{L^p_tH^5}.$ 
A natural approach is to first work on the the zero modes of $(Q, W)$ and then recover the estimates of $\U_0$. 
This relies on analyzing the boundedness of the operators
$\pa_z^{-1} |\na_{y,z}|^{-1}$ and ${\pa_y\pa_{z}^{-1}|\na_{y,z}|^{-2}}.$ 
Such a strategy works well in the setting $(y,z)\in \R\times \T,$ with additional estimates on double zero (in $x$ and $z$) frequencies of $U_{0}^{1,3}.$ Moreover, one can obtain favorable estimates on $\pa_y U_{0}^{1,3}$ 
by imposing suitable assumptions on the initial data of the double zero frequencies of $U_{0}^{1,3}.$ Nevertheless, this strategy fails in the current setting $(y,z)\in\R^2$ or the case $\R\times \T$ without any assumptions on the double zero modes, since $\pa_z |\na_{y,z}|^{-2}$ and 
${\pa_y|\nabla_y|^{-1}\pa_{z}^{-1}}$ exhibit singularities at low $z$ frequencies and it is no longer appropriate to consider separately zero and non-zero frequencies in $z.$ 
To resolve the problem, instead of working on the simple zero (zero in $x$ but non-zero in $z$) modes of $(Q, W),$ we will try to propagate the regularity of the zero modes $\U_0$ (or $\uu_0$) directly by performing the energy estimates on the system \eqref{1-eq2}. 
This is doable due to the symmetric structure in the system 
\eqref{1-eq2}.  

Coming back to \eqref{intro-1}. A natural approach is to leverage the dissipation term to control $\|\pa_{z} \uu_{0} \|_{L_t^2H^5}$, which leads to 
$\alpha\geq 1$. To relax the constraint, 
we should instead put $p=+\infty.$ 
The crucial observation that allows us to do so is
\begin{align*} 
	&\|\pa_z\tilde{\pa}_yU_{\not=}^j\|_{L^\infty_tH^4} \leq \|(Q, W)\|_{L^\infty_tH^4},~j\in\{1,3\},
	\\
	&\|\widetilde{\Delta}U_{\not=}^2\|_{L^\infty_tH^4} \leq \|Q\|_{L^\infty_tH^4}
\end{align*} 
which suggests that one can propagate the $H^5$ norm of $\pa_z \uu_0,$ given the $H^4$ regularity of $(Q,W)^t.$ 
These motivate us to introduce the anisotropic energy functional for the zero modes $\uu_0,$ see \eqref{def: e-1}.
Thus, we get that
\begin{align*}
    &\int_0^t\int_{\T\times\R^2}	|\nabla^4\A(\pa_z U_{0}^2\widetilde{\nabla}U_{\not=}^3)| |\nabla^4\A\widetilde{\nabla}Q_{\not=}|\,dxdydz d\tau
    \\
    \lesssim &~	\|\pa_{z}U_{0}^2 \|_{L^\infty_tH^5}  
    \|\A \widetilde{\nabla}U_{\not=}^3\|_{L^{2}_tH^4} \|\A\widetilde{\nabla}Q_{\not=}\|_{L^2_tH^4}
    \\
    \lesssim&~ \nu^{-\f23} |\E_{0}(t)|^{\f12} \int_0^t\D_{\not=}(\tau)\,d\tau
   \lesssim  \ep_1^3\nu^{3\al-\f23}.
\end{align*}


$\bullet$ Next, we aim to deal with the second term in the right-hand side of \eqref{sec1:est-en-non}.
A direct computation shows that it can be controlled by 
\begin{align*}
    \Big|\int_0^t\int_{\T\times\R^2} \A\tilde{\pa}_y\nabla^4(\pa_{i}U_{\not=}^2\tilde{\pa}_yU_{\not=}^i)\, \nabla^4\A Q_{\not=}\,dxdydzd\tau \Big|
\lesssim& \|\pa_i U_{\neq}^2\|_{L_t^2H^4}\|\tilde{\pa}_yU_{\not=}^i\|_{L_t^{\infty}H^4}
\|\tilde{\pa}_y \A Q_{\neq}\|_{L_t^2H^4}
\\
\lesssim&~ \nu^{-\f56} |\E_{\not=}(t)|^{\f12} \int_0^t\D_{\not=}(\tau)\,d\tau
\lesssim  \ep_1^3\nu^{3\al-\f56},
\end{align*}
which limits the stability thresholds $\alpha$ to be no less than $\frac56$.
To further reduce the threshold, we have to take fully advantage of the inviscid damping of $U_{\neq}^2.$ 
For this purpose, we split the second term in the right-hand side of \eqref{sec1:est-en-non} into three terms
\begin{align*}
	&\int_0^t\int_{\T\times\R^2}	\nabla^4\A(\tilde{\pa}_y\pa_{i}U_{\not=}^2\tilde{\pa}_yU_{\not=}^i) 
	\nabla^4\A Q_{\not=}\,dxdydzd\tau
	\nonumber
	\\
	&+
	\int_0^t\int_{\T\times\R^2}
	\A([\nabla^4,\tilde{\pa}_y^2U_{\not=}^i]\pa_{i}U_{\not=}^2) 
	\nabla^4\A Q_{\not=}\,dxdydzd\tau
	\nonumber
	\\
	&+
	\int_0^t\int_{\T\times\R^2}
	\A(\nabla^4\pa_{i}U_{\not=}^2
	\tilde{\pa}_y^2U_{\not=}^i)
	\nabla^4\A Q_{\not=}\,dxdydzd\tau 
	\nonumber
\end{align*}
The first term can be handled easily by favorable estimates for $\tilde{\pa}_y\pa_{i}U_{\not=}^2.$ For the second term, we can make use of the inviscid damping estimate for $U_{\neq}^2:$
\begin{align}\label{U2-damp-intr}
    \|U_{\neq}^2(t)\|_{H^3}\lesssim \||\tilde{\na}|^{-1}\A Q_{\neq}\|_{H^3}\lesssim (1+t)^{-1}\|\A Q_{\neq}\|_{H^4}.
\end{align}
The last term is the most involved, since a similar estimate to \eqref{U2-damp-intr} cannot be applied due to the regularity issue. The idea to overcome this difficulty is to find a way to `transfer' half derivative $|\tilde{\na}|^{-\f12},$ which arises from the relation $U_{\neq}^2(t)\approx |\tilde{\na}|^{-1}\A Q_{\neq}$
into other quantities such as $\nabla^4\A Q_{\not=},$ thereby also gaining time integrability. To this end, and motivated by \cite{Knobel-2025}, we introduce the Fourier multiplier $m_3$  (defined in \eqref{def-m0}) to construct the time-dependent weight $\varUpsilon$ (given by \eqref{def-Af}).
This weight will bring us a new damping $\|
\varUpsilon\A (Q_{\not=},W_{\not=})(t)\|_{L^2(\T\times\R^2)}$,  
which enables us to control the last term as
\begin{align*}
& 
	\int_0^t\int_{\T\times\R^2}
	\A(\nabla^4\pa_{i}U_{\not=}^2
	\tilde{\pa}_y^2U_{\not=}^i)
	\nabla^4\A Q_{\not=}\,dxdydzd\tau \\
	&\lesssim 
	\nu^{-\f23}|\ln(\nu^{-1})|^{1+a_0} \|\A(Q_{\not=},W_{\not=})\|_{L^\infty_t H^4}
	\big\|\varUpsilon\A( Q_{\not=},W_{\not=})\big\|_{L^2_t H^4}^2+\cdots
	\\
	&\lesssim \nu^{-\f23}|\ln(\nu^{-1})|^{1+a_0}| |\E_{\not=}(t)|^{\f12}\int_0^t\D_{\not=}(\tau)\,d\tau
	\lesssim  \ep_1^3\nu^{3\al-\f23}|\ln(\nu^{-1})|^{1+a_0}
\end{align*}
for some $a_0\in (0,\sqrt{2}-1).$ To summary, in order to propagate high regularity norms for the non-zero modes, we need the threshold $\alpha>\frac23.$

Considering the previously obtained transition threshold of $\alpha_1\leq \f13$ for the 2-D Navier-Stokes equations \cite{Bedrossian-Wang-2018, MZ, WZ-2023} and the presence of a stretching term in the 3D system, it is natural to conjecture that $\alpha_1\leq \f23$ in the current context.

\subsubsection{Zero modes}
We consider the zero modes system.
\medskip


\noindent\textbf{Dispersive structure and estimates.}
We work in the stable regime 
$B_{\beta}>0$ in which the system \eqref{1-eq2} exhibits an oscillatory structure that gives rise to dispersive properties.
More precisely, as in the previous works \cite{Huang-Sun-Xu, Huang-Sun-Xu-1, Coti-Zotto-Widmayer-2025}, one can
 symmetrize the coupled system satisfied by 
$(u_0^1,u_0^2)$ in system \eqref{1-eq2}, 
and introduce the quantities 
\begin{align*} 
	W^{\pm}\tri u_0^2\pm i \sqrt{\f{\beta}{\beta-1}} \tilde{\mathcal{R}}_3 u_0^1\,, \quad \text{ with } \,
	 \tilde{\mathcal{R}}_3=\sqrt{-\mathcal{R}^2_3}= 
     \sqrt{\pa_z^2(\pa_y^2+\pa_z^2)^{-1}},
\end{align*}
which satisfies that
\begin{align*}
	\pt W^{\pm}=  	\tilde{\mathcal{L}}^{\pm}W^{\pm}+\textrm{nonlinear terms},
\end{align*}
where the operators $\tilde{\mathcal{L}}^{\pm}$ are defined by
$$\tilde{\mathcal{L}}^{\pm}\tri \pm i\mathrm{sgn}(\beta)\sqrt{B_{\beta}}\tilde{\mathcal{R}}_3 +\nu(\pa_y^2+\pa_z^2). 
$$
An important feature of the linear operators $\tilde{\mathcal{L}}^{\pm}$ is that they exhibit the dispersive property: 
\begin{align}\label{dispersive-intro}
	\|e^{t\tilde{\mathcal{L}}^{\pm}}f\|_{L^{q}(\R^2)}\lesssim (1+t)^{-\f12(1-\f2{q})} \|f\|_{W^{(2^{+})(1-\f2q),q'}(\R^2)}, \qquad \forall~ 2\leq q\leq +\infty.
\end{align}
The above estimate indicates that not only is the lift-up effect suppressed when
$B_{\beta}> 0,$ but the oscillatory nature also causes the amplitude of the zero modes to decay over time. This is the primary reason that one can expect the stability threshold to be lower than in the case without rotation, i.e., $B_{\beta}=0$.
Nevertheless, by working on $W^{\pm}$ and \eqref{dispersive-intro}, there are still
 some troubles in proving the stability of Couette flow with the smallest possible threshold. Indeed, from \eqref{dispersive-intro} and the definition of $W^{\pm},$ we are only able to 
 obtain the following anisotropic estimates for $\uu_0:$	
\begin{align} \label{nau0-intro}
	\|(\nabla u_0^2,\partial_z\uu_0 ) \|_{L^{\infty}(\R^2)}
	\lesssim (1+t)^{-\f12}\|\uu_{0in}\|_{
    W^{3^{+},1}(\R^2)}
    +\text{contribution of nonlinear terms.}
\end{align}	
However, estimates for $(1+\pa_y)(u_0^1, u_0^3)$ are not available. 
Whereas on the domain $(y,z)\in \R\times\T$, such estimates can be achieved by studying the simple zero (zero in $x_1$ but non-zero in $z$) modes of $(u_0^1, u_0^3)$ and using the divergence-free condition $\pa_z u_0^3=-\pa_y u_0^2$ to control $u_0^3$ by $u_0^2$, this method fails for $(y,z)\in \R^2$ due to the singularities in the low  frequencies.
With the unfavorable dispersive estimates \eqref{nau0-intro}, we are only able to obtain a mild stability threshold $\alpha>\f45$. 

To obtain a improved threshold, we need to identify a new dispersive structure.  For such a purpose, we introduce the new unknown
\begin{align*} 
\up_0
={\f{\beta-1}{\sqrt{B_\beta}}} (-\partial_y^2-\partial_z^2)^{-\f12} (\pa_z u_0^2-\pa_y u_0^3),
\end{align*}
which, together with $u_0^1,$ solves a coupled system
\begin{align*}
\left\{\begin{array}{l}
	\pt u_0^1+\sqrt{B_\beta}\mathcal{R}_3 \up_0-\nu\Delta u_0^1
	=N^1, 
	\\[1ex]
	\pt \up_0+\sqrt{B_\beta}\mathcal{R}_3u_0^1-\nu\Delta  \up_0
	=N^2,
    \end{array}\right.
\end{align*}
where $\cR_3=\partial_z(-\partial_y^2-\partial_z^2)^{-\f12}$, and
$N^{1}, N^{2}$ are some nonlinear terms.
Then the quantity $\W^{\pm}\tri u_0^1 \pm  \up_0$ 
solves
\begin{align*}
	\pt \W^{\pm}= 
	\mathcal{L}^{\pm}
	\W^{\pm}+(N^{1}\pm N^2) 
\end{align*}
where ${\mathcal{L}}^{\pm}\tri \pm \sqrt{B_{\beta}}\mathcal{R}_3 +\nu(\pa_y^2+\pa_z^2). $ We remark that 
$e^{t\mathcal{L}^{\pm}}$ enjoys same dispersive property as $e^{t\tilde{\mathcal{L}}^{\pm}}.$ 
The advantage of introducing
 $\W^{\pm}$ is that it allows one to derive dispersive estimates for all components of
 $\uu_0$ through the following relations: 
\begin{align*}
u_0^1 = \f12\big(\W^{+}+\W^{-}\big),
\quad
\up_0 = \f12\big(\W^{+}-\W^{-}\big),
\quad u_0^2={-\f{\beta}{\sqrt {B_\beta}}} \mathcal{R}_3\up_0,
\quad
u_0^3={\f{\beta}{\sqrt {B_\beta}}} \mathcal{R}_2\up_0\,,
\end{align*}
where $\cR_2=\partial_y(-\partial_y^2-\partial_z^2)^{-\f12}.$
From these identities, we readily get that 
\begin{align}\label{intro-dispes-2}
	\|\uu_0\|_{L^{\infty}(\R^2)}
	\lesssim (1+t)^{-\f12} \|\uu_{0in}\|_{W^{2^+,1}(\R^2)}+\text{contribution of nonlinear terms.}
\end{align}
Moreover, as proved in Lemma \ref{App-dis-2}, the following improved dispersive estimate holds
\begin{align*}
	\|e^{t{\mathcal{L}}^{\pm}}\pa_y f\|_{L^{\infty}(\R^2)}\lesssim (1+t)^{-1} \|f\|_{W^{3^{+},1}(\R^2)}.  
\end{align*}
As a result, we find that the quantities $u_0^3$ and $\pa_y\uu_0$, which we refer to as ``good components", enjoy the improved dispersive estimates: 
\begin{align}\label{disp-improved}
    \|(u_0^3, \pa_y \uu_0)\|_{L^{\infty}(\R^2)}
	\lesssim (1+t)^{-1} \|\uu_{0in}\|_{W^{3^+,1}(\R^2)}+\text{contribution of nonlinear terms.}
\end{align}
\noindent\textbf{Propagation of high Sobolev regularity for zero modes.} As mentioned in the previous section, in order to close the estimates for the nonzero modes, it is necessary to propagate the zero modes in an anisotropic Sobolev norm
$\|(1+\pa_z)\uu_0\|_{L_t^{\infty}H^5(\R^2)}.$ This can be achieved by establishing the following $L_t^2L^{\infty}(\R^2)$ estimates for $\uu_0:$
\begin{align}\label{L2linfty-zero}
    \|(1+\pa_z)\uu_0\|_{L_t^2L^{\infty}(\R^2)}\lesssim \ep_1 \nu^{\f12},
\end{align}
under the \textit{a priori} assumption 
\begin{align}\label{pre-asump}
    \|(1+\pa_z)\uu_0\|_{L_t^{\infty}H^5(\R^2)}+\nu^{\f12}\|(1+\pa_z)\na\uu_0\|_{L_t^{2}H^5(\R^2)}\lesssim \ep_1 \nu^{\alpha}.
\end{align}
To show \eqref{L2linfty-zero}, we establish some useful Strichartz type estimates in Lemma \ref{lemm-str}. Notably, it holds that 
\begin{align}\label{stri-intro}
	\|(1+\pa_z)\uu_0\|_{L^2_tL^\infty(\R^2)}
	\lesssim \nu^{-\delta}\|(1+\pa_z)\uu_{0in}\|_{W^{2,1}(\R^2)}+\nu^{-\f14}\|(1+\pa_z)\PP_0(\uu\cdot\nabla_{\x_1}  \uu)\|_{L^1_tH^{1}(\R^2)},
\end{align}
where the second term in the right hand side represents the contribution of nonlinear terms. Let us remark that if we instead using directly the dispersive estimate \eqref{intro-dispes-2}, the second term is replaced by 
$\nu^{-\delta_1}\|(1+\pa_z)\PP_0(\uu\cdot\nabla_{\x_1}\uu)\|_{L^1_tW^{3,1}(\R^2)},$ which, by \eqref{pre-asump},
can only be bounded by $\ep_1^2\nu^{2\alpha-1-\delta_1}.$ 
This requires the threshold $\alpha> \f34$ to guarantee that estimate \eqref{L2linfty-zero} remains valid.
The advantage of employing an $L_t^1H^1$ type norm for the nonlinear term is that it allows us to exploit both the good structure of the transport term:
$\uu_0\cdot\nabla\uu_0= u_0^2\pa_y \uu_0+u_0^3\pa_z\uu_0$
and the improved dispersive estimate \eqref{disp-improved}. Upon using \eqref{stri-intro}, we find that 
\begin{align*}
	&\nu^{\delta_1}\|(1+\pa_z)u_0^3\|_{L^2_tW^{1,\infty}(\R^2)}
	+\nu^{\f12+\delta_1}\|(1+\pa_z)\pa_y\uu_0\|_{L^1_tW^{1,\infty}(\R^2)}
	\nonumber
	\\
	\lesssim &~
	\|(1+\pa_z)\uu_{0in}\|_{H^{2}(\R^2)}+\|(1+\pa_z)\PP_0(\uu\cdot\nabla_{\x_1}  \uu)\|_{L^1_tH^{2}(\R^2)}
    \nonumber
    \\
    \lesssim&~\var_0+\ep_1\nu^{\al-\f12}\|(1+\pa_z)\uu_0\|_{L^2_tL^\infty(\R^2)}+\ep_1^2\nu^{2\al-\f13}.
\end{align*}
Employing the estimates above, one can bound 
\begin{align*}
    & \nu^{-\f14}\|(1+\pa_z)(\uu_0\cdot\nabla\uu_0)\|_{L^1_tH^{1}(\R^2)}
    \nonumber
    \\
    \lesssim 
    &~
    \nu^{-\f14}\big(\|(1+\pa_z)u^3_0\|_{L^2_tW^{1,\infty}(\R^2)}
    \|(1+\pa_z)\pa_z \uu_0\|_{L^2_tH^{1}(\R^2)}
    \nonumber
    \\
     & +\|(1+\pa_z)u^2_0\|_{L^\infty_tH^{1}(\R^2)}
    \|(1+\pa_z)\pa_y \uu_0\|_{L^1_tW^{1,\infty}(\R^2)}\big)
    \nonumber
    \\
    \lesssim
    &~\ep_1\nu^{\al-\f34-\delta_1} (\var_0+\ep_1\nu^{\al-\f12}\|(1+\pa_z)\uu_0\|_{L^2_tL^\infty(\R^2)}+\ep_1^2\nu^{2\al-\f13}),
\end{align*}
which is less than $\ep_1\nu^{\f12}+\ep_1^2 \|(1+\pa_z)\uu_0\|_{L^2_tL^\infty}$ as long as $\ep_1$ small enough and $\alpha> \f58.$ We refer to Lemma \ref{sec4:lemm-u0} for more details.


\bigskip

\noindent\textbf{Notation.}
\begin{itemize}
\item[(1)] Throughout this paper, all constants $C$ may be different in different lines, but universal and independent of $\nu$. 
A constant with subscript(s) illustrates the dependence on the parameters. For example, $C_{a}$ is a constant which depends on $a$. 
\item[(2)]For $a\lesssim b$, we mean that there exists a universal constant $C$, such that $a\leq Cb$. While $a\gtrsim b$, we mean that there exists a universal constant $C$, such that $a\geq Cb$. And $a\sim b$ means that $a\lesssim b\lesssim a$. 
\item[(3)]  For any real number $r\geq1$, $r'$ represents the conjugate exponent of $r$. Specifically, for $r=1$, then $r'=\infty$.
\item[(4)]  The symbol $|\nabla|\tri \sqrt{-\pa_x^2-\pa_y^2-\pa_z^2}$, and $|\nabla_{y,z}|\tri \sqrt{-\pa_y^2-\pa_z^2}$.

\item[(5)] The symbol $\langle a, b\rangle \tri \sqrt{1+a^2+b^2}$, and $\langle a\rangle\tri \sqrt{1+a^2}$.

\item[(6)]  We sometimes represent for simplicity $\pa_x$, $\pa_y$ and $\pa_z$ as $\pa_1$, $\pa_2$ and $\pa_3$, respectively.

\item[(7)]  
For any function $f(t,x,y,z)$, we define its Fourier transform on the domain $\T\times\R^2$ as:
\begin{align*}
\F[f](t,k,\xi,\eta)=\hat{f}^k(t,\xi,\eta)
\tri \f{1}{(2\pi)^3} \int_{\T\times\R^2} f(t,x,y,z)e^{-ikx-i\xi y-i\eta z}\,dx dy dz,\quad \forall ~k\in\Z,~(\xi,\eta)\in\R^2.
\end{align*} 
While the Fourier transform on the domain $\T\times\R\times\T$ as:
\begin{align*}
\F[f](t,k,\xi,l)=\hat{f}^k(t,\xi,l)
\tri \f{1}{(2\pi)^3} \int_{\T\times\R\times\T} f(t,x,y,z)e^{-ikx-i\xi y-il z}\,dx dy dz,\quad \forall ~(k,l)\in\Z^2,~\xi\in\R.
\end{align*} 

\item[(8)] For any $k\in\R$, we denote $k^{+}$ as a number larger but sufficiently close to $k.$

\end{itemize}

\noindent\textbf{Organization.}
Our analysis primarily addresses the setting where the domain is $\DD=\R^2$. Since the case $\DD=\R\times\T$ can be handled analogously, its proofs are deferred to the last section of the manuscript.
The organization of this paper is as follows. We begin in Section \ref{sec2:Pre} by introducing the good unknowns and Fourier multipliers. 
Energy functionals are constructed in Section \ref{sec3:energy fun}.  
The associated energy estimates for zero modes are rigorously established in Section \ref{sec4:energy est zero}, while that for non-zero modes are derived in Section \ref{sec4:energy est nonzero}. 
Based on these energy estimates, the proofs of Theorems \ref{theo1}--\ref{theo2} for the domain $\T\times\R^2$ are detailed in Section \ref{sec5:proof-TRR}. 
The corresponding proofs for the domain $\T\times\R\times\T$ are provided in Section \ref{sec5:proof-TRT} via a parallel argument. 
For the sake of self-containment, technical proofs regarding the derivation of dispersive and Strichartz estimates are presented in Appendix \ref{App-DS-est}.

\section{Preliminaries}
\label{sec2:Pre}
\subsection{Good unknowns and change of frame }

Motivated by the standard 3-D incompressible Navier-Stokes equations, we introduce the following two auxiliary unknowns  to analyze their dynamical evolution:
\begin{align}\label{def-q-om2}
q\tri \Delta_{\x_1} u^2,
\quad
\omega^2\tri \pa_{z} u^1-\pa_{x_1} u^3.
\end{align}
It is derived from the original system \eqref{1-eq1} that $(q,\omega^2)$ satisfies the following coupled system:
\begin{align}\label{3-eq1}
\left\{\begin{array}{l}
	\displaystyle
	\pa_t q +y\pa_{x_1} q+ \beta \pa_{z}\om^2-\nu \Delta_{\x_1}q=-\pa_{y}{\Delta} p^{NL}-\Delta_{\x_1}(\uu\cdot\na_{\x_1}u^2),
	\\[1ex]
	\displaystyle
	\pa_t \om^2 +y\pa_{x_1} \om^2+ (1-\beta)\pa_z\Delta_{\x_1}^{-1}q-\nu \Delta_{\x_1}\om^2=-\pa_{z}(\uu\cdot\na_{\x_1}u^1)+\pa_{x_1}(\uu\cdot\na_{\x_1}u^3).
\end{array}\right.
\end{align}
A key distinction from the standard 3-D Navier-Stokes equations lies in the rotational coupling between $(q,\om^2)$, where their governing equations are linearly coupled through the Coriolis terms. 
This inherent coupling suggests treating $(q,\om^2)$ as a complete coupled system, whose individual components cannot be estimated separately in the energy analysis.

The inherent shear effects of Couette flow naturally motivate our introduction of moving frame coordinates:
\begin{align}\label{def-t1-x1-y1}
t= t,\quad x\tri x_1-ty,\quad y= y,\quad z=z.
\end{align}
In this frame, we set
\begin{align}\label{def-Q}
Q(t,x,y,z)\tri q(t,x+ty,y,z),
\quad 
\Omega^2(t,x,y,z)\tri \om^2(t,x+ty,y,z).
\end{align}
More generally, for any function $f(t,x_1,y,z)$ (which may be $q,\om^2,\uu,p^L,p^{NL}$), we define the corresponding quantity $F(t,x,y,z)\tri f(t,x+ty,y,z)$ (which may be $Q,\Omega^2,\U,P^L,P^{NL}$) under which the differential operators transform as	
\begin{align*}
&\tilde{\pa}_y F \tri (\pa_{y}-t\pa_{x})F,
\quad
\nabla_{\x_1}f=
\widetilde{\nabla}F \tri (\pa_{x}F,\tilde{\pa}_y F,\pa_{z}F),
\\
&\dive_{\x_1}f =
\widetilde{\dive}F \tri \pa_{x} F^1+\tilde{\pa}_y F^2+\pa_z F^3,
\quad
\Delta_{\x_1}f =
\widetilde{\Delta}F \tri \pa_{x}^2F+\tilde{\pa}_y^2F+\pa_z^2 F.
\end{align*}
And we denote 
\begin{align*}
\nabla F\tri (\pa_{x}F,\pa_{y}F,\pa_{z}F),
\quad \dive F\tri \pa_{x} F^1+\pa_{y}F^2+\pa_{z}F^3,
\quad \Delta F\tri \pa_{x}^2F+\pa_{y}^2F+\pa_{z}^2 F.
\end{align*}
We note that the operator $\nabla$ is referred  to as ``good derivative", whereas the operator $\widetilde{\nabla}$ is referred to as ``bad derivative".

In the stable regime where $B_{\beta}>0$, similar to \cite{Coti-Zotto-Widmayer-2025},  to maintain consistent regularity for $Q$, we also define the following unknown:
\begin{align}\label{def-W-nonzero}
W\tri \sqrt{\f{\beta}{\beta-1}}|\widetilde{\nabla}| \Omega^2.
\end{align}
The system \eqref{3-eq1} is then reduced to
\begin{align}\label{3-eq2}
\left\{\begin{array}{l}
	\displaystyle
	\pa_t Q +C_{\beta} \pa_{z} |\widetilde{\nabla}| ^{-1}W-\nu \widetilde{\Delta}Q
	=-\tilde{\pa}_y \widetilde{\Delta} P^{NL}-\widetilde{\Delta}(\U\cdot\widetilde{\nabla}U^2),
	\\[1ex]
	\displaystyle
	\pa_t W+C_{\beta} \pa_{z} |\widetilde{\nabla}| ^{-1}Q
	-\pa_x\tilde{\pa}_y |\widetilde{\nabla}| ^{-2}W
	-\nu \widetilde{\Delta}W
	=\sqrt{\f{\beta}{\beta-1}} |\widetilde{\nabla}| 
	\big(
	\pa_{z}(\U\cdot\widetilde{\nabla}U^1)-\pa_{x}(\U\cdot\widetilde{\nabla}U^3)
	\big),
\end{array}\right.
\end{align}
where $C_{\beta} \tri\mathrm{sgn}(\beta)\sqrt{B_\beta}$, the nonlinear pressure term may be determined explicitly,
\begin{align*}
\widetilde{\Delta} P^{NL} =-\widetilde{\dive} (\U\cdot \widetilde{\nabla} \U).
\end{align*}
The analysis for estimating the non-zero frequencies relies on the pair $(Q,W)$, hereafter denoted as the  ``good unknowns".

\subsection{The Fourier multipliers}\label{sect-F-M}
In this subsection, we present some Fourier multipliers that will be frequently used in our analysis.

We begin with the first Fourier multiplier $\varphi$, which characterizes the interplay between stretching and dissipation effects.
We focus on a fundamental model originally introduced in \cite{Bedrossian-Germain-Masmoudi-2017}, described by the following equation:
\begin{align*}
\pa_t f+\pa_x(\pa_y-t\pa_x) \widetilde{\Delta}^{-1} f-\nu \widetilde{\Delta} f=0,
\end{align*}
where the operator $\widetilde{\Delta}\tri \pa_x^2+(\pa_y-t\pa_x)^2+\pa_z^2$.
Taking the Fourier transform, this model becomes
\begin{align}\label{eq-toy}
\partial_{t} \hat{f} =\f12\f{\pa_{t}p}{p}\hat{f}-\nu p \hat{f},
\end{align}
here the symbol $p\tri k^2+(\xi-kt)^2+\eta^2$.
Notably, the stretching  term $\f12\f{\pa_{t}p}{p}\hat{f}$ in the right-hand side of equation \eqref{eq-toy} is known to induce growth effects and 
 this growth cannot be fully balanced by the dissipation $\nu p\hat{f}$. 

For $t\leq \f{\xi}{k}$, we observe that $\pa_{t}p=-2k(\xi-kt)\leq 0$, implying that the first term on the right-hand side acts as a damping term for $f$.
However, when $t> \f{\xi}{k}$, the derivative becomes positive, i.e., $\pa_{t}p> 0$, and the corresponding term drives growth.
Near the critical time $t=\f{\xi}{k}$ the dissipation satisfies $\nu p\sim\nu (k^2+\eta^2)$, which is insufficient to uniformly control the growth in $\nu$.
Nevertheless, sufficiently far from the critical layer, dissipation dominates:  for any fixed $\beta_0>2$,  the inequality
$$\f{\pa_{t}p}{p}\leq \f{1}{4}\nu p,$$
holds whenever $|t-\f{\xi}{k}|\geq \beta_0 \nu^{-\f13}$.

To enable the dissipation term $\nu p \hat{f}$ to balance the growth caused by $\f12\f{\pa_{t}p}{p}\hat{f}$, we adopt the approach of \cite{Bedrossian-Germain-Masmoudi-2017,Liss-2020}, and introduce the Fourier multiplier $\varphi(t,k,\xi,\eta)$:
\begin{align}\label{def-phi}
\left\{\begin{array}{l}
	\partial_{t} \varphi = 
	\left\{\begin{array}{l}
		0,\quad t\not\in [\f{\xi}{k},\f{\xi}{k}+\beta_0\nu^{-\f13}],\\[1ex]
		\f{\pa_{t} p}{p}\varphi	,\quad t\in [\f{\xi}{k},\f{\xi}{k}+\beta_0\nu^{-\f13}],\\[1ex]
	\end{array}\right.\\
	\varphi(0,k,\xi,\eta)=1,
\end{array}\right.
\end{align}
where $\beta_0>2$ is a constant to be determined later. 
Note that $\varphi(t,0,\xi,\eta)=1$ by construction. 
In contrast to $\hat{f},$ 
the transformed quantity $\varphi^{-\f12}\hat{f}$ exhibits no growth as it is governed by the equation 
\begin{align*}
(\pa_t +\nu p )(\varphi^{-\f12}\hat{f})=0.
\end{align*}
Consequently, the decay of the multiplier $\varphi^{-\f12}$ precisely balance the growth that $\hat{f}$.
Remarkably, $\varphi$ admits a closed-form expression as:
\begin{itemize}
\item[(1)] If $k=0$, or if $k\not=0$ and $\f{\xi}{k}\leq -\beta_0 \nu^{-\f13}$: 
$$\varphi(t,0,\xi,\eta)=1;$$


\item[(2)] If $k\not=0$, $-\beta_0 \nu^{-\f13}<\f{\xi}{k}\leq 0$:
\begin{align*}
	\varphi(t,k,\xi,\eta)=
	\begin{dcases}
		\f{k^2+(\xi-kt)^2+\eta^2}{k^2+\xi^2+\eta^2},\quad &\text{if}~~0<t\leq \f{\xi}{k}+\beta_0 \nu^{-\f13},
		\\[1ex]
		\f{k^2+(\beta_0 k\nu^{-\f13})^2+\eta^2}{k^2+\xi^2+\eta^2},\quad &\text{if}~~t>\f{\xi}{k}+\beta_0 \nu^{-\f13};
	\end{dcases}
\end{align*}

\item[(3)] If $k\not=0$, $\f{\xi}{k}> 0$:
\begin{align*}
	\varphi(t,k,\xi,\eta)=
	\begin{dcases}
		1,\quad &\text{if}~~0<t\leq \f{\xi}{k},
		\\[1ex]
		\f{k^2+(\xi-kt)^2+\eta^2}{k^2+\eta^2},\quad &\text{if}~~\f{\xi}{k}<t\leq \f{\xi}{k}+\beta_0 \nu^{-\f13},
		\\[1ex]
		\f{k^2+(\beta_0 k\nu^{-\f13})^2+\eta^2}{k^2+\eta^2},\quad &\text{if}~~t>\f{\xi}{k}+\beta_0 \nu^{-\f13}.
	\end{dcases}
\end{align*}

\end{itemize}

This leads to the following lemma, which shows that the growth is exactly balanced by dissipation.

\begin{lemm}\label{sec2:lem-0}
Let $\varphi$ be the Fourier multiplier defined in \eqref{def-phi}.
For any $t\geq0$, $k\in\Z \backslash \{0\}$ and $(\xi,\eta) \in\R^2$, we have
\begin{align}\label{est-varphi-1}
	1\leq \varphi(t,k,\xi,\eta)\leq \beta_0^2 \nu^{-\f23},\quad \f{\varphi}{p}\leq \f{1}{k^2+\eta^2}.
\end{align}
Moreover, for any $(\xi',\eta')\in\R^2$, the following product estimates hold that
\begin{align}
	&\varphi(t,k,\xi,\eta)\leq C  \varphi(t,k,\xi',\eta')\langle \xi-\xi',\eta-\eta'\rangle^2,
	\label{est-comm}
	\\
	&\langle t\rangle^{-2} \varphi(t,k,\xi,\eta) \leq 3,
	\label{est-varphi-t}
\end{align}
where $\langle \xi-\xi',\eta-\eta'\rangle\tri \sqrt{1+ (\xi-\xi')^2+(\eta-\eta')^2}$ and $C>0$ is a constant.
\end{lemm}

\begin{proof}
	The proofs of estimates \eqref{est-varphi-1}-\eqref{est-comm} can be found in \cite{Bedrossian-Germain-Masmoudi-2017} and are omitted here.
	Now, we turn to prove \eqref{est-varphi-t}.
	If $k=0$, or if $k\not=0$ and $\f{\xi}{k}\leq -\beta_0\nu^{-\f13}$, then $\varphi=1$.
	In these cases, estimate \eqref{est-varphi-t} is trivially satisfied.
	For the case that $k\not=0$ and $-\beta_0\nu^{-\f13}<\f{\xi}{k}\leq 0$, if $0<t\leq\f{\xi}{k}+\beta_0 \nu^{-\f13}$, we have
	\begin{align*}
		\langle t\rangle^{-2} \varphi 
		\leq \f{k^2+2\xi^2+2k^2t^2+\eta^2}{(k^2+\xi^2+\eta^2)\langle t\rangle^2}
		\leq \f{2\xi^2+2k^2\langle t\rangle^2+\eta^2}{(k^2+\xi^2+\eta^2)\langle t\rangle^2}
\leq 2.
	\end{align*}
	If $t>\f{\xi}{k}+\beta_0 \nu^{-\f13}$, it follows that
	\begin{align*}
		\langle t\rangle^{-2} \varphi 
		\leq \f{k^2+2\xi^2+2(\xi+\beta_0 k\nu^{-\f13})^2+\eta^2}{(k^2+\xi^2+\eta^2)\langle \f{\xi}{k}+\beta_0 \nu^{-\f13}\rangle^2}
		\leq \f{2\xi^2+2k^2\langle \f{\xi}{k}+\beta_0 \nu^{-\f13}\rangle^2+\eta^2}{(k^2+\xi^2+\eta^2)\langle \f{\xi}{k}+\beta_0 \nu^{-\f13}\rangle^2}
\leq 2.
	\end{align*}
	As for the case that $k\not=0$ and $\f{\xi}{k}> 0$, if $0<t\leq\f{\xi}{k}$, we have $\varphi=1$, from which \eqref{est-varphi-t} follows trivially.
	If $\f{\xi}{k}<t\leq \f{\xi}{k}+\beta_0 \nu^{-\f13}$, we have
	\begin{align*}
		\langle t\rangle^{-2} \varphi 
		\leq \f{k^2+k^2t^2+\eta^2}{(k^2+\eta^2)\langle t\rangle^2}
		\leq 1.
	\end{align*}
	If $t>\f{\xi}{k}+\beta_0 \nu^{-\f13}$, some direct calculations yield that
	\begin{align*}
		\langle t\rangle^{-2} \varphi 
		\leq \f{k^2\langle \beta_0 \nu^{-\f13}\rangle^2+\eta^2}{(k^2+\eta^2)\langle \f{\xi}{k}+\beta_0 \nu^{-\f13}\rangle^2}
		\leq 1.
	\end{align*}
	The estimate \eqref{est-varphi-t} is thus proved.
\end{proof}

In our energy estimates, we employ three key Fourier multipliers to capture enhanced dissipation and inviscid damping effects. 
These multipliers have been widely employed in previous works \cite{Bedrossian-Germain-Masmoudi-2017,Bedrossian-Germain-Masmoudi-2019,Bedrossian-Masmoudi-Vicol-2016,Zillinger-2017}.
The first crucial multiplier, $m_1,$ addresses the transient weakening of enhanced dissipation near critical times. 
This multiplier is particularly significant as it enables us to: (1) quantify accelerated dissipation without regularity loss; 
(2) maintain regularity unlike methods requiring infinite regularity \cite{Bedrossian-Germain-Masmoudi-2020,Bedrossian-Masmoudi-Vicol-2016}.
In particular, $m_1$ is defined by
\begin{align}\label{def-m1}
\left\{\begin{array}{l}
	\displaystyle
	\pa_{t}m_1(t,k,\xi,\eta)
	=\f{2\nu^{\f13}}{1+\nu^{\f23}(\f{\xi}{k}-t)^2}m_1(t,k,\xi,\eta),
	\\[1ex]
	\displaystyle
	m_1(0,k,\xi,\eta)=1.
\end{array}\right.
\end{align}
It is easy to check that, for any $t\geq0$ and $(\xi,\eta)\in\R^2$, the multiplier $m_1$ admits the explicit solution:
\begin{align*}
&m_1(t,k,\xi,\eta)=
\begin{dcases}
	\exp\left\{2\arctan \left(\nu^{\f13}\left(t-\f{\xi}{k}\right)\right)+2\arctan \left(\nu^{\f13}\f{\xi}{k}\right)\right\},\quad \text{for}~~k\in\Z\backslash\{0\},\\
	1,\quad\text{for}~~k=0.
\end{dcases}
\end{align*}

We then introduce another multiplier $m_2$, designed to control growth arising from both linear pressure terms and leading order nonlinear terms.
In particular, $m_2$ is defined by
\begin{align}\label{def-m2}
\left\{\begin{array}{l}
	\displaystyle
	\pa_{t}m_2(t,k,\xi,\eta)
	=A\f{k^2}{p}m_2(t,k,\xi,\eta),
	\\[1ex]
	\displaystyle
	m_2(0,k,\xi,\eta)=1,
\end{array}\right.
\end{align}
where 
$A>0$ is a large, but fixed constant, which will be determined later.

The following Lemma \ref{sec2:lem-1} states key properties of the Fourier multipliers $\varphi$, $m_i$ $(i=1,2)$.
Its proof is standard and follows from \cite{Bedrossian-Germain-Masmoudi-2017}, hence we omit it.
\begin{lemm}\label{sec2:lem-1}
Let $\varphi$, $m_i$ $(i=1,2)$ be three Fourier multipliers defined in \eqref{def-phi}, \eqref{def-m1}  and \eqref{def-m2}, respectively.
Then the multipliers $m_i$ $(i=1,2)$ are bounded from above and below uniformly with respect to $0<\nu<1$ and $(t,k,\xi,\eta)\in\R^+\times \Z\times\R^2$. 
Specifically, they are bounded by
\begin{align*}
	1\leq m_1\leq e^{2\pi}, \quad
	1\leq m_2\leq e^{A\pi}.
\end{align*}
Moreover, for any $\max\{2\beta_0^{-1}(\beta_0^2-1)^{-1},4\beta_0^{-1}\}<\delta_{\beta_0}\leq 1$ with $\beta_0>2$, we also have
\begin{align}\label{est-varphi-2}
		\delta_{\beta_0}\left(\f{\pa_{t} m_1}{m_1}+\nu p\right)+\f12\left(\f{\pa_{t}\varphi}{\varphi}-\f{\pa_{t} p}{p}\right)\geq \delta_{\beta_0} \nu^{\f13}. 
\end{align}
\end{lemm}

In order to bound the following reaction term appearing in the energy estimates of $(Q_{\neq}, W_{\neq})$
\begin{align*}
\int_0^t\int_{\T\times\R^2}
\A(\nabla^s\pa_{i}U_{\not=}^2\tilde{\pa}_yU_{\not=}^i)\nabla^s\A\widetilde{\nabla}Q_{\not=}\,dxdydzd\tau,
\end{align*}
 we introduce, motivated by \cite{Knobel-2025}, the Fourier multiplier $m_3$ as:
\begin{align}\label{def-m0}
m_{3}(t,k,\xi)\tri \exp\Big\{ \int_{-\infty}^{t}\sum_{n\in\Z\backslash\{0\}}f^{n}(\tau,k,\xi)\,d\tau \Big\}, \qquad  \forall \,t\geq0, ~~(k,\xi,\eta)\in\Z\times\R^2,
\end{align}
where the function sequences $f^{n}(\tau,k,\xi)$ are defined by
\begin{align*}
f^{n}(t,k,\xi)\tri \big\langle t-\f{\xi}{n}\big\rangle^{-1} \big[\ln\big(e+\big\langle t-\f{\xi}{n}\big\rangle\big)\big]^{-1-a_0}\langle k-n\rangle^{-1-a_0},\quad \text{for}~~0<a_0<1.
\end{align*}
The precise value of the parameter $a_0$ will be determined in subsequent sections.
By some direct calculations, we can verify that the Fourier multiplier $m_{3}(t,k,\xi)$ satisfies the following property.

\begin{lemm}\label{sec2:lem-0-0}
Let $m_{3}$ be the Fourier multiplier defined in \eqref{def-m0}. Then the multiplier $m_3$ is bounded from above and below uniformly with respect to $0<\nu<1$ and $(t,k,\xi,\eta)\in\R^+\times \Z\times\R^2$. 
Specifically, it is bounded by
\begin{align}\label{est-m0-0}
	1\leq m_3\leq e^{C a_0^{-2}}.
\end{align}
Moreover, for any $0<a_0<1$, $k\in \Z\backslash\{0\}$, $k'\in\Z$, $(\xi,\xi')\in\R^2$ and $t\geq0$, we also have 
\begin{align}
	&\big\langle t-\f{\xi}{k}\big\rangle^{-1}\big[\ln\big(e+\big\langle t-\f{\xi}{k}\big\rangle\big)\big]^{-1-a_0}
	\leq \f{\pa_tm_3}{m_3}(t,k,\xi),
	\label{est-m0-1}
	\\
	&\big\langle t-\f{\xi}{k}\big\rangle^{-1}\big[\ln\big(e+\big\langle t-\f{\xi}{k}\big\rangle\big)\big]^{-1-a_0}
	\leq Ca_0^{-1-a_0}\f{\pa_tm_3}{m_3}(t,k',\xi')\langle k-k',\xi-\xi'\rangle^{2+a_0(2+a_0)}.
	\label{est-m0-2}
\end{align}
And for any $0<a_0<1$, $(k,k')\in \Z^2$, $(\xi,\xi')\in\R^2$ and $t\geq0$, we have
\begin{align}
	\f{\pa_tm_3}{m_3}(t,k,\xi)\leq Ca_0^{-1-a_0}  \f{\pa_tm_3}{m_3}(t,k',\xi')\langle k-k',\xi-\xi'\rangle^{2+a_0(2+a_0)},
	\label{est-m0-3}
\end{align}
where $C>0$ is a constant independent of $a_0$. 
\end{lemm}
\begin{proof}
    It is easy to observe that
	\begin{align*}
		f^n(t,k,\xi)\leq \langle k-n\rangle^{-1-a_0},\quad \int_{-\infty}^{t}\sum_{n\in\Z\backslash\{0\}}f^{n}(\tau,k,\xi)\,d\tau\leq Ca_0^{-1}\langle k-n\rangle^{-1-a_0}.
	\end{align*}
    Since $0<a_0<1$, we can obtain \eqref{est-m0-0} immediately.
    Next, we turn to prove \eqref{est-m0-1}.
	Since the function sequences $f^n(t,k,\xi)$ are non-negative, it follows that for $k\not=0$,
	\begin{align*}
		\f{\pa_t m_3}{m_3}(t,k,\xi) =\sum_{n\in\Z\backslash\{0\}} f^n(t,k,\xi) \geq f^k(t,k,\xi)
		=\big\langle t-\f{\xi}{k}\big\rangle^{-1} \big[\ln\big(e+\big\langle t-\f{\xi}{k}\big\rangle\big)\big]^{-1-a_0}.
	\end{align*}
    We now focus on proving \eqref{est-m0-2}. 
    To this end, we deduce that for $k\not=0$,
	\begin{align}\label{est-pam0-1}
		\f{\pa_t m_3}{m_3}(t,k',\xi')
		\geq f^k(t,k',\xi')
		=\big\langle t-\f{\xi'}{k}\big\rangle^{-1} \big[\ln\big(e+\big\langle t-\f{\xi'}{k}\big\rangle\big)\big]^{-1-a_0}\langle k'-k\rangle^{-1-a_0}.
	\end{align}
    Some simple calculations give that for any $0<a_0<1$, $(x,y)\in\R^2$,
	\begin{align*}
		\ln(e+\langle x\rangle)\lesssim a_0^{-1} \langle x\rangle^{a_0},
		\quad
		\ln(e+\langle x+y\rangle)\leq \ln(e+\langle x\rangle)+\ln(e+\langle y\rangle).
	\end{align*}
    From which, it then follows that
	\begin{align}\label{est-ln-1}
		\ln\big(e+\big\langle t-\f{\xi'}{k}\big\rangle\big)\big[\ln\big(e+\big\langle t-\f{\xi}{k}\big\rangle\big)\big]^{-1}
		\lesssim a_0^{-1}\langle \xi-\xi'\rangle^{a_0}.
	\end{align}
    We also recall that
	\begin{align*}
		\big\langle t-\f{\xi'}{k}\big\rangle \leq \big\langle t-\f{\xi}{k}\big\rangle \langle \xi-\xi'\rangle.
	\end{align*}
    This, combined with \eqref{est-ln-1}, implies that
	\begin{align*}
		\big\langle t-\f{\xi'}{k}\big\rangle\big[\ln\big(e+\big\langle t-\f{\xi'}{k}\big\rangle\big)\big]^{1+a_0}
		\lesssim
		a_0^{-(1+a_0)}\big\langle t-\f{\xi}{k}\big\rangle 
		\big[\ln\big(e+\big\langle t-\f{\xi}{k}\big\rangle\big)\big]^{1+a_0}  \langle \xi-\xi'\rangle^{1+a_0(1+a_0)}.
	\end{align*}
	We substitute the estimate above into \eqref{est-pam0-1} to obtain that
	\begin{align*}
		\f{\pa_t m_3}{m_3}(t,k',\xi')
		\gtrsim  a_0^{1+a_0}
		\big\langle t-\f{\xi}{k}\big\rangle^{-1} 
		\big[\ln\big(e+\big\langle t-\f{\xi}{k}\big\rangle\big)\big]^{-1-a_0}  \langle \xi-\xi'\rangle^{-1-a_0(1+a_0)}\langle k'-k\rangle^{-1-a_0},
	\end{align*}
	from which the estimate \eqref{est-m0-2} follows.
    Finally, following the similar derivations of \eqref{est-m0-1} and \eqref{est-m0-2}, one obtain \eqref{est-m0-3} directly.
\end{proof}

For simplicity, the combined multiplier is then defined as:
\begin{align*}
m\tri m_1m_2m_3.
\end{align*}

\section{Construction of the energy functionals}
\label{sec3:energy fun}
Our goal in this section is to 
construct the energy functionals and establish the \textit{a priori} estimates for the local smooth solutions of system \eqref{1-eq1}.
Firstly, we introduce some symbols for any $f(t,x,y,z)$ as:
\begin{align}
&\A f \tri
\F^{-1}[\varphi^{-\f12}m^{-1}e^{\delta_2\nu^{\f13}t}\hat{f}], 
\quad 
\varUpsilon\A f=  \F^{-1}\big[\sqrt{\f{\pa_tm_3}{m_3}}\varphi^{-\f12}m^{-1}e^{\delta_2\nu^{\f13}t}\hat{f}\big],
\label{def-Af}
\\
&\A_0 f_0\tri \F^{-1}[m_3^{-1}(t,0,\xi,\eta)\hat{f}_0],
\quad
\varUpsilon_0\A_0 f_0\tri \F^{-1}\big[\sqrt{\f{\pa_tm_3}{m_3}}m_3^{-1}(t,0,\xi,\eta)\hat{f}_0\big].
\label{def-A0f0}
\end{align}
Here we note that $\varphi^{-\f12}m^{-1}(t,0,\xi,\eta)=m_3^{-1}(t,0,\xi,\eta)$.
Since the dynamics of these solutions differ from those independent of the periodic $x$-variable, we introduce distinct energy functionals accordingly.

\medskip

\noindent\textbf{$\bullet$ For the case $\DD=\R^2$:} 
For the zero modes, we introduce the energy functionals for the solution to system \eqref{1-eq2} as
\begin{align}
\label{def: e-1}
&\E_0(t) \tri  \|\uu_0(t)\|_{H^4(\R^2)}^2+\|(1+\pa_z)\nabla^5\A_0\uu_0(t)\|_{L^2(\R^2)}^2,
\\
\label{def: d-1}
&\D_0(t)\tri  \nu \|\nabla \uu_0(t)\|_{H^4(\R^2)}^2
+\nu\|(1+\pa_z)\nabla^6 \A_0\uu_0(t)\|_{L^2(\R^2)}^2
+\|(1+\pa_z)\nabla^5 \varUpsilon_0\A_0\uu_0(t)\|_{L^2(\R^2)}^2.
\end{align}
The Fourier multiplier $m_3$, corresponding to the operators $\A_0$ and $\varUpsilon_0$, is introduced in the highest energy functionals to control the reaction term arising from the quadratic interactions where both unknowns are at non-zero modes.

We now turn to the analysis of non-zero modes for which we will work on the pair $(Q,W)$ defined in \eqref{def-Q} and \eqref{def-W-nonzero}.
In addition to $m_3$, we employ the Fourier multiplier $\varphi$ to take into account the stretching term in the second equation of system \eqref{3-eq2}.
Furthermore, the Fourier multipliers $m_i$ $(i=1,2)$ are used to achieve both enhanced dissipation and inviscid damping.
We thus define the energy functionals for non-zero frequencies for the solution to system \eqref{3-eq2} as:
\begin{align}
&\E_{\not=}(t)\tri   
\|\A (Q_{\not=},W_{\not=})(t)\|_{H^4(\T\times\R^2)}^2,
\label{defDnot-1}\\
&\D_{\not=}(t)\tri 
\nu \|\widetilde{\nabla}\A (Q_{\not=},W_{\not=})(t)\|_{H^4(\T\times\R^2)}^2
+\nu^{\f13} 
\|\A (Q_{\not=},W_{\not=})(t)\|_{H^4(\T\times\R^2)}^2
\nonumber\\
&\qquad\qquad
+\|\pa_x|\widetilde{\nabla}|^{-1}\A (Q_{\not=},W_{\not=})(t)\|_{H^4(\T\times\R^2)}^2
+\|\varUpsilon\A (Q_{\not=},W_{\not=})(t)\|_{H^4(\T\times\R^2)}^2.
\label{defDnot}
\end{align}

\medskip

\noindent\textbf{$\bullet$ For the case $\DD=\R\times\T$:} 
The energy functionals for the zero modes of the solution to system \eqref{1-eq2} are defined as
\begin{align}
\label{def: ed-T}
\E_0(t) \tri  \|(1+\pa_z)\uu_0(t)\|_{H^5(\R\times\T)}^2,
\quad
\D_0(t)\tri  \nu \|(1+\pa_z)\nabla \uu_0(t)\|_{H^5(\R\times\T)}^2.
\end{align}
While the corresponding energy functionals for the non-zero modes of the solution to system \eqref{3-eq2} are defined in \eqref{defDnot-1}-\eqref{defDnot} by replacing the domain $\T\times\R^2$ by $\T\times\R\times\T$.
\begin{rema}
    On the domain $\T\times\R\times\T$, the energy functionals defined in \eqref{def: ed-T} for the zero modes are defined without the Fourier multiplier $m_3$, while those for the non-zero modes retain the definitions in \eqref{defDnot-1}-\eqref{defDnot} from $\T\times\R^2$ (including $m_3$). 
    Strictly speaking, the multiplier $m_3$ is not essential for the domain $\T\times\R\times\T$, as the stability threshold for the streak solution is limited to $\alpha\leq \f56$, 
    which removes the need for a detailed treatment of the associated reaction terms, both for zero and non-zero modes.
    Nevertheless, to maintain a unified and concise presentation in the proof of non-zero modes, we employ the same energy functional definitions for both domains $\DD=\R^2$ or $\DD=\R\times\T$.
\end{rema}

We will use the continuity argument to prove Theorems \ref{theo1}-\ref{theo2}.
Thus, we give the following proposition concerning \textit{a priori} energy estimates.
\begin{prop}\label{prop-1}
Let $T>0.$
Assume that $\uu_0$ 
and $(Q_{\not=},W_{\not=})$ 
are the solutions for systems \eqref{1-eq2} and \eqref{3-eq2}, respectively,
such that for any $\ep_1>0$ and $t\in[0,T]$,
\begin{align}\label{priori assumption}
	\sup_{\tau \in[0,t]}\big(\E_0(\tau)+\E_{\not=}(\tau)\big)
	+\int_0^t\big(\D_0(\tau)+\D_{\not=}(\tau)\big)\,d\tau
	\leq \ep_1^2\nu^{2\alpha},
\end{align}
with any $\alpha>\f23$ on the domain $\T\times\R^2$, and with any $\alpha\geq \f56$ on the domain $\T\times\R\times\T$.
Then it holds that
\begin{align}\label{sec3:priori result-1}
	\sup_{\tau\in[0,t]}\big(\E_0(\tau)+\E_{\not=}(\tau)\big)
	+\int_0^t\big(\D_0(\tau)+\D_{\not=}(\tau)\big)\,d\tau
	\leq C\var_0^2+C\ep_1^3\nu^{2\alpha}.
\end{align}
\end{prop}

To conclude this section, we establish key estimates for $\U_{\neq}$ in Lemma \ref{3lem-key}. 
These estimates are fundamental to the subsequent energy estimates in Sections \ref{sec4:energy est zero}–\ref{sec4:energy est nonzero} for proving Proposition \ref{prop-1} on $\T\times\R^2$, and in Section \ref{sec5:proof-TRT} for the case of $\T\times\R\times\T$.


\begin{lemm}\label{3lem-key}
Under the assumptions in Proposition \ref{prop-1}, there holds that
\begin{itemize}
	\item [(1)] for $U^j_{\not=}$ with $j=1,3$, it holds that
	\begin{align*}
		&\|\nabla_{x,z}^2U_{\not=}^{j}(t)\|_{H^4}^2
		+\nu^{\f23}\|\widetilde{\nabla}\nabla_{x,z}U_{\not=}^{j}(t)\|_{H^4}^2
		\lesssim \E_{\not=}(t),
		\\
		&\nu\|\widetilde{\nabla}\nabla_{x,z}^2U_{\not=}^{j}(t)\|_{H^4}^2
		+\nu^{\f53}\|\widetilde{\nabla}^2\nabla_{x,z}U_{\not=}^{j}(t)\|_{H^4}^2
		\lesssim \D_{\not=}(t),
		\\
		&\nu^{\f13}\|\nabla_{x,z}^2U_{\not=}^{j}(t)\|_{H^4}^2
		+\nu\|\widetilde{\nabla}\nabla_{x,z}U_{\not=}^{j}(t)\|_{H^4}^2
		\lesssim \D_{\not=}(t);
	\end{align*}
	\item [(2)] for $U^2_{\not=}$, it holds that
	\begin{align*}
		&\|\nabla_{x,z}^2U_{\not=}^{2}(t)\|_{H^4}^2
		+\|\widetilde{\nabla}\nabla_{x,z}U_{\not=}^{2}(t)\|_{H^4}^2
		+\nu^{\f23}\|\widetilde{\nabla}^2U_{\not=}^{2}(t)\|_{H^4}^2
		\lesssim \E_{\not=}(t),
		\\
		&
		\nu \|\widetilde{\nabla}\nabla_{x,z}^2U_{\not=}^{2}(t)\|_{H^4}^2
		+\nu\|\widetilde{\nabla}^2\nabla_{x,z}U_{\not=}^{2}(t)\|_{H^4}^2
		+\nu^{\f53}\|\widetilde{\nabla}^3U_{\not=}^{2}(t)\|_{H^4}^2
		\lesssim \D_{\not=}(t),
		\\
		&\|\nabla_{x,z}U_{\not=}^{2}(t)\|_{H^4}^2
		+
		\nu^{\f13}\|\nabla_{x,z}^2U_{\not=}^{2}(t)\|_{H^4}^2
		+\nu^{\f13}\|\widetilde{\nabla}\nabla_{x,z}U_{\not=}^{2}(t)\|_{H^4}^2
		+\nu\|\widetilde{\nabla}^2U_{\not=}^{2}(t)\|_{H^4}^2
		\lesssim \D_{\not=}(t).
	\end{align*}
\end{itemize}
\end{lemm}
\begin{proof}
By using \eqref{def-q-om2} and the divergence free condition, the non-zero frequencies of the solution $\uu$ can be explicitly expressed as:
\begin{align}
	&U^1_{\not=} =(\pa_{x}^2+\pa_{z}^2)^{-1}\Big(\sqrt{\f{\beta-1}{\beta}}\pa_z |\widetilde{\nabla}|^{-1}W_{\not=}-\pa_{x}\tilde{\pa}_{y} \widetilde{\Delta}^{-1}Q_{\not=}\Big),
	\label{def-u1not}
	\\
	&U^2_{\not=}=\widetilde{\Delta}^{-1}Q_{\not=},
	\label{def-u2not}
	\\
	&U^3_{\not=} =-(\pa_{x}^2+\pa_{z}^2)^{-1}\Big(\sqrt{\f{\beta-1}{\beta}}\pa_{x} |\widetilde{\nabla}|^{-1}W_{\not=}+\tilde{\pa}_{y}\pa_{z} \widetilde{\Delta}^{-1}Q_{\not=}\Big).
	\label{def-u3not}
\end{align}
Lemma \ref{3lem-key} follows immediately from Lemma \ref{sec2:lem-0}, the representations \eqref{def-u1not}-\eqref{def-u3not}, and the \textit{a priori} estimate \eqref{priori assumption}, we thus omit its proof.
\end{proof}

\begin{rema}
This remark clarifies the consistency in the derivative orders of the energy functionals for zero and non-zero modes. 
Using the domain $\T\times\R^2$ as an example, we observe that the energy functional $\mathcal{E}_0(t)$ defined in \eqref{def: e-1} exhibits an anisotropic derivative structure.
Combined the incompressibility condition $\partial_y^6 u_0^2 = -\partial_z\partial_y^5 u_0^3$, a key feature in $\mathcal{E}_0(t)$ is the omission of control of $\|\pa_y^6u^j\|_{L^2}$ with $=1,3$.
A consistent structure is found for the non-zero modes:
Lemma \ref{3lem-key} shows that $\E_{\not=}(t)$ similarly lacks control of
$\|\widetilde{\nabla}^2U^j_{\not=}\|_{H^4}$ with $j=1,3$.
That we are able to close the energy estimates within this anisotropic Sobolev framework stems from this consistency. This motivates the definitions in \eqref{def: e-1}--\eqref{defDnot}.
\end{rema}

\section{Energy estimates for zero modes on \texorpdfstring{$\T\times\R^2$}{T x R2}}\label{sec4:energy est zero}

From this section, we intend to give the proof of Proposition \ref{prop-1} on $\T\times\R^2$.
In view of the structural differences between the energy functionals for zero and non-zero frequencies of the solution, we perform distinct analysis for each case.
In this section, we establish the \textit{a priori} energy estimates for zero modes of the solution which is stated in the following proposition. 
\begin{prop}\label{lemm-1}
Suppose that the assumptions in Proposition \ref{prop-1} are true on the domain $\T\times\R^2$.
Then it holds that for any $t\in[0,T]$,
\begin{align}
	\label{priori result-0}
	\sup_{\tau\in[0,t]}\E_0(\tau)
	+\int_0^t \D_0(\tau)\,d\tau
	\lesssim \var_0^2+ \ep_1^3\nu^{2\al}.
\end{align}
\end{prop}

\subsection{The dispersive and Strichartz estimates for zero modes of velocity} \label{sec-zeromode-1}
Before proceeding to prove Proposition \ref{lemm-1}, we first establish the dispersive and Strichartz estimates for the zero frequencies of the solution $\uu_0$ to system \eqref{1-eq2}, a key ingredient in the proof.
This requires identifying the specific unknown functions that admit such estimates, leading us to introduce the following:
\begin{align}\label{def-cu}
\up_0
\tri \f{\beta-1}{\sqrt{B_{\beta}}} |\nabla_{y,z}|^{-1} (\pa_z u_0^2-\pa_y u_0^3),\quad B_{\beta}=\beta(\beta-1).
\end{align}
In view of system \eqref{1-eq2}, we see that $(u_{0}^1, \up_0)$ satisfies the equations: 
\begin{align*}
\left\{\begin{array}{l}
	\pt u_0^1+\sqrt{B_\beta}\mathcal{R}_3 \up_0-\nu\Delta u_0^1
	=N^1, 
	\\[1ex]
	\pt \up_0+\sqrt{B_\beta}\mathcal{R}_3u_0^1-\nu\Delta  \up_0
	=N^2,
	\\[1ex]
	(u_0^1,\up_0)|_{t=0}=(u_{0in}^1,\up_{0in})=(u_{0in}^1,\f{\beta-1}{\sqrt{B_{\beta}}} |\nabla|^{-1} (\pa_z u_{0in}^2-\pa_y u_{0in}^3)),
\end{array}\right.
\end{align*}
where the nonlinear terms $N^1$ and $N^2$ are defined by
\begin{align*}
N^1 \tri -\PP_0(\uu\cdot\nabla_{\x_1}  u^1),
\quad
N^2 \tri -\f{\beta-1}{\sqrt{B_{\beta}}}\big(\mathcal{R}_3\PP_0(\uu\cdot\nabla_{\x_1}  u^2) -\mathcal{R}_2\PP_0(\uu\cdot\nabla_{\x_1}  u^3)\big).
\end{align*}
We then define 
\begin{align}\label{def-W-1}
\W^{\pm}
\tri u_0^1 \pm  \up_0.
\end{align}
Direct computations show that $\W^{\pm}$ solve the equations
\begin{align*}
\left\{\begin{array}{l}
	\pt \W^{\pm}= 
	\mathcal{L}^{\pm}
	\W^{\pm}+N^{\pm},
	\\[1ex]
	\W^{\pm}|_{t=0}=\W^{\pm}_ {in}= u_{0in}^1 \pm \upsilon_{0in},
\end{array}\right.
\end{align*}
where the linear operators 
$$\mathcal{L}^{\pm}\tri \mp \sqrt{B_{\beta}} \mathcal{R}_3+\nu\Delta ,$$
and the nonlinear terms 
\begin{align}\label{def-Npm}
N^{\pm}\tri N^1\pm  
N^2.
\end{align}
Thus, the solution $\W^{\pm}$ can be given explicitly as
\begin{align}\label{est-Wpm}
\W^{\pm} = e^{t\mathcal{L}^{\pm}} \W^{\pm}_ {in}+\int_0^t e^{(t-\tau)\mathcal{L}^{\pm}} N^{\pm} (\tau)\, d\tau .
\end{align}
It is worth noting that various estimates for $\uu_0$ can be deduced from those of $\W^{+}$ and $\W^{-}$. Indeed, 
by the definition \eqref{def-W-1}, $u_0^1$ and $\up_0$ can be represented as:
\begin{align}\label{rep-u01-w01}
u_0^1 = \f12(\W^{+}+\W^{-}),
\quad
\up_0 = \f12(\W^{+}-\W^{-})\,.
\end{align}
Moreover, thanks to the divergence free condition,  it holds that
\begin{align}\label{rep-u023}
u_0^2=-{\f{\beta}{\sqrt{B_{\beta}}}} \mathcal{R}_3\up_0,
\quad
u_0^3=\f{\beta}{\sqrt{B_{\beta}}} \mathcal{R}_2\up_0\,.
\end{align}

The striking property for the unknowns $\W^{\pm}$ 
is that the linear evolution 
$e^{t\mathcal{L}^{\pm}}$
admits both dispersive and dissipation effects from respectively $e^{t\sqrt{B_{\beta}}\cR_3}$  and 
$e^{\nu t\Delta},$ which allow us to establish the suitable Strichartz estimates, even at endpoint case. 
We refer to
Lemmas \ref{lem-str-1} and \ref{lem-str-2} for the Strichartz estimates for $e^{t\mathcal{L}^{\pm}},$ which are based on the following dispersive estimates: 
\begin{lemm}\label{lemm-disp}
It holds that (here $\delta_1\in(0,1/2) $)
\begin{align*}	
\|e^{t\sqrt{B_{\beta}}\cR_3}f\|_{L^{\infty}(\R^2)}	\lesssim (1+\sqrt{B_{\beta}}|t|)^{-\f12} \|f\|_{W^{2^+,1}(\R^2)},\\
\|e^{t\sqrt{B_{\beta}}\cR_3}(\pa_yf,\cR_2f)\|_{L^{\infty}(\R^2)}	\lesssim (1+\sqrt{B_{\beta}}|t|)^{-1} \|f\|_{W^{3^+,1}(\R^2)},
\end{align*}
and
\begin{align*}	
\|e^{t\mathcal{L}^{\pm}}f\|_{L^{\infty}(\R^2)}	\lesssim (1+\sqrt{B_{\beta}}|t|)^{-\f12}|\nu t|^{-\delta_1} \|f\|_{W^{2,1}(\R^2)},\\
\|e^{t\mathcal{L}^{\pm}}(\pa_yf,\cR_2f)\|_{L^{\infty}(\R^2)}	\lesssim (1+\sqrt{B_{\beta}}|t|)^{-1}|\nu t|^{-\delta_1} \|f\|_{W^{3,1}(\R^2)}.
\end{align*}
\end{lemm}
\begin{proof}
The first  two  inequalities follows from Lemmas \ref{App-dis-1} and \ref{App-dis-2}, respectively. 
As $\mathcal{L}^{\pm}=\mp \sqrt{B_{\beta}} \mathcal{R}_3+\nu\Delta ,$ we apply successively the estimates 
(using \eqref{dispes-pj} and \eqref{dispes-pj-pay})
\begin{align*}
&  \|  e^{\nu t\Delta}P_k h\|_{L^{\infty}(\R^2)} \lesssim e^{-c\nu 2^{2k}t}\|h\|_{L^{\infty}(\R^2)}\lesssim |\nu 2^{2k} t|^{-\delta_1}\|h\|_{L^{\infty}(\R^2)}, \, \quad \forall~ t>0,\\
& \left\|e^{\sqrt{B_{\beta}}t \mathcal{R}_3}  {P_k} h\right\|_{L^{\infty}(\R^2)} \lesssim (1+B_{\beta}|t|)^{-\f12}  2^{2 k}\|h\|_{L^{1}(\R^2)},\\
&\left\|e^{\sqrt{B_{\beta}}t \mathcal{R}_3} \partial_y {P_k} h\right\|_{L^{\infty}(\R^2)} \lesssim (1+B_{\beta}|t|)^{-1}  2^{3 k}\|h\|_{L^{1}(\R^2)},\\
&\left\|e^{\sqrt{B_{\beta}}t \mathcal{R}_3}\mathcal{R}_2 {P_k} h\right\|_{L^{\infty}(\R^2)} \lesssim (1+B_{\beta}|t|)^{-1}  2^{2 k}\|h\|_{L^{1}(\R^2)},
\end{align*}
to find that for any $t> 0,$
\begin{align*} 
& \left\|e^{t\mathcal{L}^{\pm}}  {P_k} h\right\|_{L^{\infty}(\R^2)} \lesssim (1+B_{\beta}|t|)^{-\f12} |\nu  t|^{-\delta_1} 2^{(2-2\delta_1) k}\|h\|_{L^{1}(\R^2)},\\
&\left\|e^{t\mathcal{L}^{\pm}} \partial_y {P_k} h\right\|_{L^{\infty}(\R^2)} \lesssim (1+B_{\beta}|t|)^{-1}  |\nu  t|^{-\delta_1} 2^{(3-2\delta_1) k}\|h\|_{L^{1}(\R^2)},\\
&\left\|e^{t\mathcal{L}^{\pm}}\mathcal{R}_2 {P_k} h\right\|_{L^{\infty}(\R^2)} \lesssim (1+B_{\beta}|t|)^{-1}  |\nu  t|^{-\delta_1} 2^{(2-2\delta_1) k}\|h\|_{L^{1}(\R^2)}.
\end{align*}
Summing up in $k$, we conclude the third and fourth inequalities.
\end{proof}
Note that since $\cR_2=\pa_y|\nabla_{y,z}|^{-1},$ 
we observe from the above lemma  that $u_0^3$ and $\pa_y (u_0^1, u_0^2)$ enjoy better dispersive decay than the other quantities such as $(u_0^1, u_0^2).$
By applying 
the Strichartz estimates established in Lemmas \ref{lem-str-1} and \ref{lem-str-2}, we deduce that the following properties:
\begin{lemm}\label{lemm-str}
It holds that
\begin{align}
	&\nu^{\f14}\|(1+\pa_z)\uu_0\|_{L^2_tL^\infty(\R^2)}
	+\nu^{\f34}\|(1+\pa_z)\nabla\uu_0\|_{L^1_tL^\infty(\R^2)}
    \nonumber\\
	\lesssim&~ \nu^{\f{1}{4}-\delta_1}\|(1+\pa_z)\uu_{0in}\|_{W^{2,1}(\R^2)}+\|(1+\pa_z)\PP_0(\uu\cdot\nabla_{\x_1}  \uu)\|_{L^1_tH^{(\f12)^+}(\R^2)}.
	\label{Lipdecay-1}
\end{align}
For the ``good components" $u_0^3$ and $\pa_y\uu_0$, it holds that
\begin{align}
	&\nu^{\delta_1}\|(1+\pa_z)u_0^3\|_{L^2_tW^{1,\infty}(\R^2)}
	+\nu^{\f12+\delta_1}\|(1+\pa_z)\nabla u_0^3\|_{L^1_tW^{1,\infty}(\R^2)}
	\nonumber
	\\
	\lesssim &~
	\|(1+\pa_z)\uu_{0in}\|_{H^{2}(\R^2)}+\|(1+\pa_z)\PP_0(\uu\cdot\nabla_{\x_1}  \uu)\|_{L^1_t H^{2}(\R^2)},
	\label{Lipdecay-2}
    \end{align}
    and
    \begin{align}
   & \nu^{\f12+\delta_1}\|(1+\pa_z)\pa_y \uu_0\|_{L^1_t W^{1,\infty}(\R^2)} \lesssim 
	\|(1+\pa_z)\uu_{0in}\|_{H^{{2}}(\R^2)}+\|(1+\pa_z)\PP_0(\uu\cdot\nabla_{\x_1}  \uu)\|_{L^1_t H^{2}(\R^2)}. \label{Lipdecay-3}
\end{align}

\end{lemm}
\begin{proof}
In view of the dispersive estimate \eqref{dispes-pj}, it holds that (for $\delta_1>0 $ small enough and $b=0,1$)
\begin{align*}
	\|e^{t\mathcal{L}^{\pm}} \chi_j(D)\W^{\pm}_{in}\|_{L^{\f{2}{1+b}}_tL^\infty}&\lesssim 2^{2j}\|e^{-c2^{2j}\nu t} (1+t)^{-\f12}\|_{L^\f{2}{1+b}(\R_{+})} \|\chi_j(D)\W^{\pm}_{in}\|_{L^{1}}\\
    &\lesssim \nu^{-\delta_1-b/2}\, 2^{2j(1-\delta_1)-bj}\|\chi_j(D)\uu_{0in}\|_{L^{1}},
\end{align*}
By summing the above estimates over $j\in\Z$, we  obtain that
\begin{align}\label{est-win-2}
	\|e^{t\mathcal{L}^{\pm}} \W^{\pm}_{in}\|_{L^2_tL^\infty}+ \nu^{\f12}\|e^{t\mathcal{L}^{\pm}} \W^{\pm}_{in}\|_{L^1_tL^\infty} \lesssim    \nu^{-\delta_1}
	\|\uu_{0in}\|_{W^{2,1}}.
\end{align}
Applying Minkowski inequality, Fubini's Theorem and \eqref{str-2}, we have, for $b=0,1$
\begin{align*}
    \big\|\int_0^t  |\nabla|^b e^{(t-\tau)\mathcal{L}^{\pm}} N^{\pm} (\tau)\, d\tau \big\|_{L^{\f{2}{1+b}}_t(\R_{+};L^\infty)}
   & \lesssim 
    \int_{\R_{+}}\|\mathbb{I}_{\{t>\tau \}} e^{(t-\tau)\mathcal{L}^{\pm}} N^{\pm} (\tau)\|_{L^{\f{2}{1+b}}_t(\R_{+};L^\infty)} \, d\tau \\
&\lesssim \nu^{-(\f14+\f{b}{2})} \|N^{\pm}\|_{L^1_tH^{(\f12)^+}}.
\end{align*}
In view of \eqref{est-Wpm}, the above two estimates yield that
\begin{align*}
\|\W^{\pm}\|_{L_t^2L^{\infty}}+\nu^{\f12}\|\nabla \W^{\pm}\|_{L_t^1L^{\infty}}
    \lesssim 
    \nu^{-\delta_1}
	\|\uu_{0in}\|_{W^{2,1}}
    +\nu^{-\f14} \|N^{\pm}\|_{L^1_tH^{(\f12)^+}}.
\end{align*}
It then follows from \eqref{rep-u01-w01}, the definition of $N^{\pm}$ in \eqref{def-Npm}
that
\begin{align}\label{est-u01w01-1}
	\nu^{\f14}
	\|(u_0^1,\up_0)\|_{L^2_tL^\infty} +\nu^{\f34}\|\nabla (u_0^1,\up_0)\|_{L_t^1L^{\infty}}\lesssim  \nu^{\f14-\delta_1}\|\uu_{0in}\|_{W^{2,1}}+\|\PP_0(\uu\cdot\nabla_{\x_1}  \uu)\|_{L^1_tH^{(\f12)^+}}.
\end{align}
The estimate of $\pa_z(u_0^1,\up_0)$ can be proved along the same line, we thus finish the proof of \eqref{Lipdecay-1}. 

Thanks to the fact $u_0^3=\f{\beta}{\sqrt{B_{\beta}}} \pa_y|\nabla_{y,z}|^{-1}v_0,$ we can apply \eqref{str-4}, \eqref{str-5} and Minkowski inequality
to obtain \eqref{Lipdecay-2} and \eqref{Lipdecay-3}.


\end{proof}

\subsection{Proof of Proposition \ref{lemm-1}}
This subsection is devoted to the proof of Proposition \ref{lemm-1}. 
Building on Strichartz type estimates for the velocity established in Lemma \ref{lemm-str}, 
we first  establish the following lemma concerning the $L^2_tL^{\infty}$-norm of  ``good components" $u_{0}^3$ and $\pa_y\uu_0:$

\begin{lemm}\label{sec4:lemm-u03}
    Under the assumptions in Proposition \ref{lemm-1}, it holds that
    \begin{align}\label{est-u3-0}
        &\nu^{\delta_1}\|(1+\pa_z)u_0^3\|_{L^2_tW^{1,\infty}}
        +\nu^{\f12+\delta_1}\big(\|(1+\pa_z)\nabla u_0^3\|_{L^1_tW^{1,\infty}}
        +\|(1+\pa_z)\pa_y \uu_0\|_{L^1_tW^{1,\infty}}\big)
        \nonumber
        \\
        \lesssim &~\var_0+\ep_1\nu^{\alpha-\f12} \|(1+\pa_z)\uu_0\|_{L^2_tL^{\infty}}+\ep_1^2\nu^{2\al-\f13}.
    \end{align}
\end{lemm}

\begin{proof}
  Since only $(0,0)$ and $(\not=,\not=)$ interactions can force zero modes, 
  we can deduce from the 
    estimates \eqref{Lipdecay-2}-\eqref{Lipdecay-3} that
\begin{align}\label{est-u3-1}
	&\nu^{\delta_1}\|(1+\pa_z)u_0^3\|_{L^2_t W^{1,\infty}}
	+\nu^{\f12+\delta_1}
    \|(1+\pa_z)(\nabla u_0^3,\pa_y \uu_0)\|_{L^1_t W^{1,\infty}}
	\nonumber
	\\
	\lesssim &
	\|(1+\pa_z)\uu_{0in}\|_{H^{2}}
    +\|(1+\pa_z)(\uu_0\cdot\nabla\uu_0)\|_{L^1_tH^{2}}
    +\|(1+\pa_z)\PP_0(\uu_{\not=}\cdot\nabla_{\x_1}\uu_{\not=})\|_{L^1_tH^{2}}.
\end{align}
Thanks to the
\textit{a priori} assumption \eqref{priori assumption} and the product estimate
\begin{align*}
   \| \nabla^{\gamma_1}f\,\nabla^{\gamma_2}g\|_{L^2(\R^2)}\lesssim \|f\|_{L^{\infty}(\R^2)}\|g\|_{\dot{H}^{|\gamma_1|+|\gamma_2|}}
   +\|g\|_{L^{\infty}(\R^2)}\|f\|_{\dot{H}^{|\gamma_1|+|\gamma_2|}(\R^2)}, \quad (\gamma_1,\gamma_2\in \mathbb{N}^2),
\end{align*}
 it holds that
\begin{align}\label{est-non0-1}
    \|(1+\pa_z)(\uu_0\cdot\nabla\uu_0)\|_{L^1_tH^{2}}
    \lesssim \|(1+\pa_z)\uu_0\|_{L^2_tL^{\infty}}
     \|(1+\pa_z)\nabla\uu_0\|_{L^2_tH^{2}}
    \lesssim \ep_1\nu^{\alpha-\f12} \|(1+\pa_z)\uu_0\|_{L^2_tL^{\infty}}.
\end{align}
For the other nonlinear term, a direct calculation gives that
\begin{align}\label{est-nonq-1}
    \|(1+\pa_z)\PP_0(\uu_{\not=}\cdot\nabla\uu_{\not=})\|_{L^1_tH^{2}}
    \lesssim 
    \sum_{j=1,3}\|(1+\pa_z)(U_{\not=}^j\pa_j\U_{\not=})\|_{L^1_tH^{2}}
    +\|(1+\pa_z)(U_{\not=}^2\tilde{\pa}_y\U_{\not=})\|_{L^1_tH^{2}}.
\end{align}
By using Lemma \ref{3lem-key}, we obtain that
\begin{align}\label{est-nonq-2}
    \sum_{j=1,3}\|(1+\pa_z)(U_{\not=}^j\pa_j\U_{\not=})\|_{L^1_tH^{2}}
    &\lesssim 
    \sum_{j=1,3}\|(1+\pa_z)U_{\not=}^j\|_{L^2_tH^{2}}
    \|(1+\pa_z)\pa_j\U_{\not=}\|_{L^2_tH^{2}}
    \nonumber
    \\
    &\lesssim 
    \nu^{-\f13} \int_0^t\D_{\not=}(\tau)\,d\tau
	\lesssim  \ep_1^2\nu^{2\al-\f13}.
\end{align}
We compute that
\begin{align}\label{est-U2-1}
	|\mathcal{F}[(1+\pa_z)U^2_{\not=}]|
	\lesssim |\mathcal{F}[|\widetilde{\nabla}|^{-1}\A Q_{\not=}]|
	\lesssim \langle t\rangle^{-1}  |\mathcal{F}[\nabla\A Q_{\not=}]|.
\end{align}
This, along with 
\eqref{est-varphi-t} and Lemma \ref{3lem-key}, yields that
\begin{align}\label{est-nonq-3}
    \|(1+\pa_z)(U_{\not=}^2\tilde{\pa}_y\U_{\not=})\|_{L^1_tH^{2}}
    &\lesssim
    \int_0^t\|(1+\pa_z)U_{\not=}^2(\tau)\|_{H^{2}}
    \|(1+\pa_z)\tilde{\pa}_y\U_{\not=}(\tau)\|_{H^{2}}\,d\tau
    \nonumber
    \\
    &\lesssim 
    \int_0^t\langle \tau \rangle^{-1}\|\A Q_{\not=}(\tau)\|_{H^{3}}
    \|(Q_{\not=},W_{\not=})(\tau)\|_{H^{2}}\,d\tau
    \nonumber
    \\
    &\lesssim \|\A Q_{\not=}\|_{L^2_tH^{3}}
    \|\A(Q_{\not=},W_{\not=})\|_{L^2_tH^{2}}
    \lesssim \ep_1^2 \nu^{2\al-\f13}.
\end{align}
Substituting \eqref{est-nonq-2} and \eqref{est-nonq-3} into \eqref{est-nonq-1}, one obtains that
\begin{align}\label{est-nonq-4}
    \|(1+\pa_z)\PP_0(\uu_{\not=}\cdot\nabla\uu_{\not=})\|_{L^1_tH^{3}}
    \lesssim \ep_1^2 \nu^{2\al-\f13}.
\end{align}
Inserting \eqref{est-non0-1} and \eqref{est-nonq-4} into \eqref{est-u3-1}, we deduce \eqref{est-u3-0} immediately.
\end{proof}

Once the $L^2_tW^{1,\infty}$ estimate for the good components has been established in Lemma \ref{sec4:lemm-u03}, we establish the $L^2_tL^\infty$ estimates for the velocity $\uu_0$.
\begin{lemm}\label{sec4:lemm-u0}
    Under the assumptions in Proposition \ref{lemm-1}, it holds that
    \begin{align}\label{est-u-0}
        \|(1+\pa_z)\uu_0\|_{L^2_tL^{\infty}}
        +\nu^{\f12}\|(1+\pa_z)\nabla \uu_0\|_{L^1_tL^\infty}
        \lesssim &~\nu^{-\f16}\var_0+\ep_1^2 \nu^{\f12}.
    \end{align}
\end{lemm}

\begin{proof}
By the estimate \eqref{Lipdecay-1}, it holds that
\begin{align}\label{est-u3-2}
	&\nu^{\f14}\|(1+\pa_z)\uu_0\|_{L^2_tL^\infty}
	+\nu^{\f34}\|(1+\pa_z)\nabla\uu_0\|_{L^1_tL^\infty}
    \nonumber
    \\
	\lesssim&~ \nu^{\f{1}{4}-\delta_1}\|(1+\pa_z)\uu_{0in}\|_{W^{2,1}}
    +\|(1+\pa_z)(\uu_0\cdot\nabla\uu_0)\|_{L^1_tH^{1}}
    +\|(1+\pa_z)\PP_0(\uu_{\not=}\cdot\nabla\uu_{\not=})\|_{L^1_tH^{1}}.
\end{align}
To obtain a refined estimate, we make use 
of the exact expression of the transport terms and the improved estimates for 
``good components''.
More precisely, we can deduce from the \textit{a priori} assumption \eqref{priori assumption} and \eqref{est-u3-0} that
\begin{align*}
    &\|(1+\pa_z)(\uu_0\cdot\nabla\uu_0)\|_{L^1_tH^{1}}
    \nonumber
    \\
    \lesssim 
    &~
    \|(1+\pa_z)u^3_0\|_{L^2_tW^{1,\infty}}
    \|(1+\pa_z)\pa_z \uu_0\|_{L^2_tH^{1}}
      +\|(1+\pa_z)u^2_0\|_{L^\infty_tH^{1}}
    \|(1+\pa_z)\pa_y \uu_0\|_{L^1_tW^{1,\infty}}
    \nonumber
    \\
    \lesssim
    &~
    \ep_1\nu^{\al-\f12-\delta_1}\big(\var_0+\ep_1\nu^{\alpha-\f12} \|(1+\pa_z)\uu_0\|_{L^2_tL^{\infty}}+\ep_1^2\nu^{2\al-\f13}\big).
\end{align*}
Substituting this estimate above into \eqref{est-u3-2} and using \eqref{est-nonq-4}, we conclude that
\begin{align*}
    &\nu^{\f14}\|(1+\pa_z)\uu_0\|_{L^2_tL^\infty}
	+\nu^{\f34}\|(1+\pa_z)\nabla\uu_0\|_{L^1_tL^\infty}
    \nonumber
    \\
	\lesssim&~ \nu^{\f{1}{4}-\delta_1}\var_0
    +\ep_1\nu^{\al-\f12-\delta_1}\big(\var_0+\ep_1\nu^{\alpha-\f12} \|(1+\pa_z)\uu_0\|_{L^2_tL^{\infty}}+\ep_1^2\nu^{2\al-\f13}\big)
    +\ep_1^2 \nu^{2\al-\f13}.
\end{align*}
Since $\al>\f23$ on the domain $\T\times\R^2$, one gets, by choosing $\ep_1,\delta_1$ suitably small, that \eqref{est-u-0} holds true.

\end{proof}

With the $L^2_tL^{\infty}$ estimate for $\uu_0$ at hands, we are now in position to prove Proposition \ref{lemm-1}.

\begin{proof}
[\underline{\textbf{Proof of Proposition \ref{lemm-1}}}] 
We will divide this proof into two steps.

\medskip

\noindent\textbf{Step 1:} 
For the lower order derivatives of the velocity, we aim to show that
    \begin{align}\label{est-u0-H4}
        \|\uu_0\|_{L^\infty_tH^{4}}^2
        +\nu \|\nabla \uu_0\|_{L^2_tH^{4}}^2
        \lesssim \var_0^2 +\ep_1^3\nu^{2\al}+\ep_1^3\nu^{3\al-\f12}.
\end{align}
Applying standard energy estimate to system \eqref{1-eq2}, we find that for any integer $0\leq s\leq 4$,
\begin{align}
	&\f12\f{d}{dt}
	\Big(
	\f{\beta}{\beta-1}\|\nabla^su_0^1\|_{L^2}^2+\|\nabla^s(u_0^2,u_0^3)\|_{L^2}^2
	\Big)
	+
	\nu\Big(\f{\beta}{\beta-1}\|\nabla^{s+1}u_0^1\|_{L^2}^2+\|\nabla^{s+1}(u_0^2,u_0^3)\|_{L^2}^2
	\Big)
	\nonumber\\
	=&-\int_{\R^2}	\Big(\f{\beta}{\beta-1}\PP_0\nabla^s(\uu\cdot\nabla_{\x_1} u^1)\cdot \nabla^su_0^1+\sum_{j=2,3}\PP_0\nabla^s(\uu\cdot\nabla_{\x_1} u^j)\cdot \nabla^su_0^j 
	\Big)\,dydz
	\tri  G_s.
	\label{energy-0-1}
\end{align}
For $s=0$, by using integration by parts and the divergence free condition, we have
\begin{align*}
	|G_0|\lesssim 
    \| \PP_0(\uu_{\not=}\cdot\nabla_{\x_1}\uu_{\not=})\|_{L^1_tL^2}\|\uu_0\|_{L^\infty_tL^2}.
\end{align*}
The method used to derive \eqref{est-nonq-4} likewise yields that
\begin{align*}
	\int_0^t |G_0|(\tau)\,d\tau 
	\lesssim \nu^{-\f13} |\E_0(t)|^{\f12}\int_0^t\D_{\not=}(\tau) \,d\tau\lesssim \ep_1^3\nu^{3\alpha-\f13}.
\end{align*}
Substituting the estimate above into \eqref{energy-0-1} with $s=0$, we find that
\begin{align*}
	\|\uu_0(t)\|_{L^2}^2
    +\nu \int_0^t \|\nabla \uu_0(\tau)\|_{L^2}^2 \,d\tau
	\lesssim \var_0^2+\ep_1^3\nu^{3\alpha-\f13}.
\end{align*}
By interpolation, to prove \eqref{est-u0-H4}, it suffices to show that
\begin{align}\label{claim-1}
	\|\nabla^4\uu_0(t)\|_{L^2}^2
    +\nu\int_0^t \|\nabla^5\uu_0(\tau)\|_{L^2}^2 \,d\tau
	\lesssim  \var_0^2
    +\ep_1^3\nu^{2\alpha}
    +\ep_1^3\nu^{3\alpha-\f12}.
\end{align}
A very crude estimate gives
\begin{align}\label{est-nabla4}
	\int_0^t|G_4|(\tau)\,d\tau
	\lesssim \big(\| \PP_0\nabla^4(\uu_{0}\cdot\nabla\uu_{0})\|_{L^1_tL^2}+
    \| \PP_0\nabla^4(\uu_{\not=}\cdot\nabla_{\x_1}\uu_{\not=})\|_{L^1_tL^2}\big)\|\nabla^4\uu_0\|_{L^\infty_tL^2}.
\end{align}
It is easy to observe from the \textit{a priori} assumption \eqref{priori assumption}, and \eqref{est-u-0} that, for $\alpha>\f23,$
\begin{align}
\label{est-G4-1}    
    \| \PP_0\nabla^4(\uu_{0}\cdot\nabla\uu_{0})\|_{L^1_tL^2}
    \lesssim   \|\uu_{0}\|_{L^2_tL^\infty}\|\na \uu_{0}\|_{L^2_tH^4}
    \lesssim 
    \ep_1(\var_0+\ep_1^2\nu^{\al}),
\end{align}
due to $\alpha>\f23$.
Furthermore, for any integer $s\geq0$,
\begin{align}\label{est-inviscid}
	\|\nabla_{x,z}U^2_{\not=}\|_{H^s}
	\lesssim
	\||\widetilde{\nabla}|^{-1}\A Q_{\not=}\|_{H^s},
\end{align}
which, combined with Lemma \ref{3lem-key}, yields that
\begin{align}
\label{est-G4-2}  
    \| \PP_0\nabla^4(\uu_{\not=}\cdot\nabla\uu_{\not=})\|_{L^1_tL^2}
    \lesssim &
    \sum_{j=1,3}\|\nabla^4(U_{\not=}^j\pa_j\U_{\not=})\|_{L^1_tL^{2}}
    +\|\nabla^4(U_{\not=}^2\tilde{\pa}_y\U_{\not=})\|_{L^1_tL^{2}}
    \nonumber
    \\
    \lesssim&
    \sum_{j=1,3}
    \|U_{\not=}^j\|_{L^2_tH^{4}}\|\pa_j\U_{\not=}\|_{L^2_tH^{4}}
    +\|U_{\not=}^2\|_{L^2_tH^{4}}\|\tilde\pa_y\U_{\not=}\|_{L^2_tH^{4}}
    \nonumber
    \\
    \lesssim&~
    \nu^{-\f12}\int_0^t \D_{\not=}(\tau)\,d\tau \lesssim \ep_1^2\nu^{2\al-\f12}.
\end{align}
Substituting \eqref{est-G4-1} and \eqref{est-G4-2} into \eqref{est-nabla4}, we obtain that
\begin{align*}
	\int_0^t|G_4|(\tau)\,d\tau
	\lesssim \var_0^2+\ep_1^3\nu^{2\al}+\ep_1^3\nu^{3\al-\f12},
\end{align*}
which, in view of
\eqref{energy-0-1}, leads to \eqref{claim-1} .
\medskip

\noindent\textbf{Step 2:}
Our objective in this step is to establish the energy estimate for the $L^\infty_t\dot{H}^5$-norm of $\uu_0$. Specifically, we seek to establish the following bound: for $\alpha>\f23,$
\begin{align}\label{est-step1}
      \|\nabla^5\A_0\uu_0\|_{L^\infty_tL^{2}}^2
        +\nu \|\nabla^6\A_0\uu_0\|_{L^2_tL^{2}}^2
        +\|\nabla^5\varUpsilon_0\A_0\uu_0\|_{L^2_tL^{2}}^2
        \lesssim \var_0^2+\ep_1^3\nu^{2\al}.
\end{align}

When establishing the energy estimate in highest order derivative of the zero mode of the velocity, we introduce the Fourier multiplier $m_3$.
Recalling the definition of the operator $\A_0$, and performing the standard energy estimate, we achieve that
\begin{align}
	&\f12\f{d}{dt}
	\Big(
	\f{\beta}{\beta-1}\|\nabla^5\A_0u_0^1\|_{L^2}^2
    +\|\nabla^5\A_0(u_0^2,u_0^3)\|_{L^2}^2
	\Big)
		\nonumber\\
      &
      +\nu\Big(\f{\beta}{\beta-1}\|\nabla^{6}\A_0u_0^1\|_{L^2}^2
    +\|\nabla^{6}\A_0(u_0^2,u_0^3)\|_{L^2}^2
	\Big) + \Big(\f{\beta}{\beta-1}\|\varUpsilon_0\nabla^{5}\A_0u_0^1\|_{L^2}^2 +\|\nabla^{5}\varUpsilon_0\A_0(u_0^2,u_0^3)\|_{L^2}^2
	\Big) 	\nonumber\\
	=&-\int_{\R^2}	\f{\beta}{\beta-1}\PP_0\A_0\nabla^5(\uu\cdot\nabla_{\x_1} u^1)\cdot \nabla^5\A_0u_0^1\,dydz
    \nonumber\\  
    &
    -\sum_{j=2,3}\int_{\R^2}\PP_0\nabla^5\A_0\big(\uu\cdot\nabla_{\x_1} u^j-\pa_j \Delta_{\x_1}^{-1}\dive_{\x_1}(\uu\cdot\nabla_{\x_1} \uu)\big)\cdot \nabla^5\A_0 u_0^j \,dydz
	\tri  G_5.
	\label{energy-0-2}
\end{align}
Using \eqref{est-m0-0}, it follows directly that
\begin{align}\label{est-G5}
    \Big|\int_0^t G_5(\tau)\,d\tau\Big|
    \lesssim&~ \|\nabla^5(\uu_0\cdot\nabla\uu_0)\|_{L^1_tL^2}
    \|\nabla^5 \uu_0\|_{L^\infty_tL^2}
    \nonumber
    \\
    &
    +\sum_{j=1}^3\Big|\int_0^t \int_{\R^2}
    \nabla^5\A_0\PP_0(\U_{\not=}\cdot\widetilde{\nabla}\U_{\not=}^j)\cdot\nabla^5 \A_0 \uu_0^j \,dydzd\tau \Big|
    \nonumber
    \\
    \tri&~ G_{50}+G_{5\not=}.
\end{align}
With the help of \textit{a priori} assumption \eqref{priori assumption}, and \eqref{est-u-0}, one obtains that
\begin{align*}
    \| \nabla^5(\uu_{0}\cdot\nabla\uu_{0})\|_{L^1_tL^2}
    \lesssim 
\|\uu_{0}\|_{L^2_tL^\infty}\|\na\uu_{0}\|_{L^2_tH^5}
    \lesssim 
    \ep_1(\var_0+\ep_1^2\nu^{\al}),
\end{align*}
since $\alpha>\f23$.
This implies  that
\begin{align}\label{est-G50}
    G_{50}
    \lesssim \var_0^2+\ep_1^3\nu^{2\al}.
\end{align}
We turn to estimate $G_{5\not=}$. Some direct computations give that
\begin{align}\label{est-G5noq}
    G_{5\not=}
    \lesssim &
    \sum_{j=1,3}\sum_{k=1}^3\Big|\int_0^t \int_{\R^2}
    \nabla^5\PP_0\A_0(U^j_{\not=}\pa_j\U_{\not=}^k)\cdot\nabla^5 \A_0 \uu_0^k \,dydzd\tau \Big|
    \nonumber
    \\
    &+\sum_{k=1}^3\Big|\int_0^t \int_{\R^2}
    \nabla^4\A_0\PP_0(\widetilde{\nabla}U^2_{\not=}\tilde\pa_y\U_{\not=}^k)\cdot\nabla^5 \A_0 \uu_0^k \,dydzd\tau \Big|
    \nonumber
    \\
    &
    +\sum_{k=1}^3\Big|\int_0^t \int_{\R^2}
    \A_0\PP_0([\nabla^4,\tilde\pa_y\widetilde{\nabla}\U_{\not=}^k]U^2_{\not=})\cdot\nabla^5 \A_0 \uu_0^k \,dydzd\tau \Big|
    \nonumber
    \\
    &
    +\sum_{k=1}^3\Big|\int_0^t \int_{\R^2}
    \A_0\PP_0(\nabla^4U^2_{\not=}\tilde\pa_y\widetilde{\nabla}\U_{\not=}^k)\cdot\nabla^5 \A_0 \uu_0^k \,dydzd\tau \Big|
    \nonumber
    \\
    \tri&~G_{51\not=}+G_{52\not=}+G_{53\not=}+G_{54\not=}.
\end{align}
By integrating by parts and applying the \textit{a priori} assumption \eqref{priori assumption} and Lemma \ref{3lem-key}, we have
\begin{align}\label{est-G51noq}
    G_{51\not=}
    \lesssim &
    \sum_{j=1,3}\|\nabla^4\A_0\PP_0(U^j_{\not=}\pa_j\U_{\not=})\|_{L^2_tL^2}
    \|\nabla^6 \uu_0\|_{L^2_tL^2}
    \nonumber
    \\
    \lesssim &
    \sum_{j=1,3}\|U^j_{\not=}\|_{L^2_tH^4}
    \|\pa_j\U_{\not=}\|_{L^\infty_tH^4}
    \|\nabla^6 \uu_0\|_{L^2_tL^2}
    \nonumber
    \\
    \lesssim &~\nu^{-\f23} |\E_{\not=}(t)|^{\f12}\int_0^t(\D_{0}+\D_{\not=})(\tau)\,d\tau 
    \lesssim \ep_1^3 \nu^{3\al-\f23}.
\end{align}
Again thanks to the \textit{a priori} assumption \eqref{priori assumption} and Lemma \ref{3lem-key}, we can bound $G_{52\not=}$ as follows:
\begin{align}\label{est-G52noq}
    G_{52\not=}
    \lesssim &
    \|\nabla^4\A_0\PP_0(\widetilde{\nabla}U^2_{\not=}\tilde\pa_y\U_{\not=})\|_{L^1_tL^2}
    \|\nabla^5 \uu_0\|_{L^\infty_tL^2}
    \nonumber
    \\
    \lesssim &
    \|\widetilde{\nabla}U^2_{\not=}\|_{L^2_tH^4}
    \|\tilde\pa_y\U_{\not=}\|_{L^2_tH^4}
    \|\nabla^5 \uu_0\|_{L^\infty_tL^2}
    \nonumber
    \\
    \lesssim &~\nu^{-\f23} |\E_{0}(t)|^{\f12}\int_0^t \D_{\not=}(\tau)\,d\tau 
    \lesssim \ep_1^3 \nu^{3\al-\f23}.
\end{align}
In view of \eqref{est-varphi-t}, \eqref{est-varphi-2} and \eqref{est-U2-1}, we obtain that
\begin{align}\label{est-G53noq}
    G_{53\not=}
    \lesssim &
    \int_0^t
    \|{|\widetilde\na|}^{-1} \A Q_{\not=}\|_{H^3}
    \|\tilde\pa_y (Q_{\not=},W_{\not=})\|_{H^4}
    \|\nabla^5 \uu_0\|_{L^2}\,d\tau
    \nonumber
    \\
    \lesssim &
    \int_0^t \langle \tau \rangle^{-1} \|\A Q_{\not=}\|_{H^4}
    \|\tilde\pa_y (Q_{\not=},W_{\not=})\|_{H^4}
    \|\nabla^5 \uu_0\|_{L^2}\,d\tau
    \nonumber
    \\
    \lesssim& ~
    \|\A Q_{\not=}\|_{L^2_tH^4}
    \|\tilde\pa_y \A(Q_{\not=},W_{\not=})\|_{L^2_tH^4}
    \|\nabla^5 \uu_0\|_{L^\infty_tL^2}
    \lesssim \ep_1^3 \nu^{3\al-\f23}.
\end{align}
 Before estimating the last term $G_{54\not=}$, we compute that, for any $k'\in \Z\backslash\{0\}$, $\xi'\in\R$ and $\tau\geq 0$,
\begin{align}\label{est-tau-1}
	\langle \tau-\f{\xi'}{k'}\rangle^{-1} 
	\lesssim &~\langle \tau-\f{\xi'}{k'}\rangle^{-1} |\ln(\nu^{-1})|^{1+a_0}
	|\ln(1+\langle \tau-\f{\xi'}{k'} \rangle)|^{-1-a_0}
	+(\nu^{-1}-1)^{-1}
	\nonumber
	\\
	\lesssim &~|\ln(\nu^{-1})|^{1+a_0}\langle \tau-\f{\xi'}{k'}\rangle^{-1} |\ln(1+\langle \tau-\f{\xi'}{k'} \rangle)|^{-1-a_0}+ \nu,
\end{align}
where the implicit constant is independent of $\nu.$
It then follows from \eqref{est-inviscid}, \eqref{est-tau-1} and Lemmas \ref{sec2:lem-0-0}, \ref{3lem-key} that
\begin{align}\label{est-G54noq}
	G_{54\not=}
	\lesssim &
	\int_0^t \langle \tau \rangle  
	\sum_{k'\in\Z\backslash\{0\}}\int_{\R^2}\int_{\R^2} 
	|(k', \xi-\xi')||(\hat{Q}_{\not=},\hat{W}_{\not=})(\tau,-k',\xi-\xi',\eta-\eta')|
    \nonumber
	\\
	&\times \langle \tau-\f{\xi'}{k'}\rangle^{-1} 
	|(k',\xi',\eta')|^{4}
	|\F[\A Q_{\not=}](\tau,k',\xi',\eta')|
	\cdot|(\xi,\eta)|^{5}
	|\hat{\uu}_0(\tau,\xi,\eta)|\,d\xi' d\eta' d\xi d\eta d\tau
    \nonumber
	\\
	\lesssim& |\ln(\nu^{-1})|^{1+a_0}
	\int_0^t \langle\tau\rangle e^{-\delta_2\nu^{\f13}\tau}
	\sum_{k'\in\Z\backslash\{0\}}\int_{\R^2}\int_{\R^2} 
	|(-k',\xi-\xi')|^{2+\f{a_0(2+a_0)}{2}}
    \nonumber
	\\
	&\times
	|(\hat{Q}_{\not=},\hat{W}_{\not=})(\tau,-k',\xi-\xi',\eta-\eta')|\cdot
	|(k',\xi',\eta')|^{4}
	\big|\sqrt{\f{\pa_t m_3}{m_3}}\F[\A Q_{\not=}](\tau,k',\xi',\eta')\big|
    \nonumber
	\\
	&\times 
	|(k,\xi,\eta)|^{5}
	\big|\sqrt{\f{\pa_t m_3}{m_3}}\hat{\uu}_0(\tau,\xi,\eta)\big|\,d\xi' d\eta' d\xi d\eta d\tau
    \nonumber
	\\
	&+
	\nu \int_0^t \langle\tau\rangle e^{-\delta_2\nu^{\f13}\tau}
	\sum_{k'\in\Z\backslash\{0\}}\int_{\R^2}\int_{\R^2} 
	|(-k',\xi-\xi')||(\hat{Q}_{\not=},\hat{W}_{\not=})(\tau,-k',\xi-\xi',\eta-\eta')|
    \nonumber
	\\
	&\times 
	|(k',\xi',\eta')|^{4}
	|\F[\A Q_{\not=}](\tau,k',\xi',\eta')|
	\cdot|(\xi,\eta)|^{5}	|\hat{\uu}_0(\tau,\xi,\eta)|\,d\xi' d\eta' d\xi d\eta d\tau
    \nonumber
	\\
	\lesssim&~
	\nu^{-\f23}
    |\ln(\nu^{-1})|^{1+a_0} \|\A(Q_{\not=},W_{\not=})\|_{L^\infty_t H^4}
	\|\varUpsilon\A Q_{\not=}\|_{L^2_t H^4}
    \|\varUpsilon_0\A_0 \uu_{0}\|_{L^2_t H^5}
    \nonumber
	\\	&+\nu^{\f23}
    \|\A(Q_{\not=},W_{\not=})\|_{L^2_t H^4}
	\|\A Q_{\not=}\|_{L^2_t H^4}\|\uu_0\|_{L^\infty_tH^5}
    \nonumber
	\\
	\lesssim&~\nu^{-\f23} |\ln(\nu^{-1})|^{1+a_0}
    \Big(|\E_{0}(t)|^{\f12}+|\E_{\not=}(t)|^{\f12}\Big)
    \int_0^t(\D_{0}+\D_{\not=})(\tau)\,d\tau
	\lesssim  \ep_1^3\nu^{3\al-\f23}|\ln(\nu^{-1})|^{1+a_0},
\end{align}
by taking $0<a_0<\sqrt{2}-1$.
Substituting \eqref{est-G51noq}-\eqref{est-G53noq}, \eqref{est-G54noq} into \eqref{est-G5noq}, we obtain that
\begin{align*}
    G_{5\not=}\lesssim \ep_1^3\nu^{3\al-\f23}|\ln(\nu^{-1})|^{1+a_0}.
\end{align*}
Collecting the estimate above, \eqref{est-G50} and \eqref{est-G5}, we find that
\begin{align*}
    \Big|\int_0^t G_{5}(\tau)\,d\tau\Big| 
    \lesssim \var_0^2+\ep_1^3 \nu^{2\al}
    +\ep_1^3\nu^{3\al-\f23}|\ln(\nu^{-1})|^{1+a_0}.
\end{align*}
Substituting the estimate above into \eqref{energy-0-2}, we conclude that \eqref{est-step1} holds true.

\medskip
\noindent
\textbf{Step 3:} In this step, we aim to derive that
\begin{align}\label{est-step2}
      \|\pa_z\nabla^5\A_0\uu_0\|_{L^\infty_tL^{2}}^2
        +\nu \|\pa_z\nabla^6\A_0\uu_0\|_{L^2_tL^{2}}^2
        +\|\pa_z\nabla^5\varUpsilon_0\A_0\uu_0\|_{L^2_tL^{2}}^2
        \lesssim \var_0^2+\ep_1^3\nu^{2\al}+\ep_1^3\nu^{3\al-\f23-\delta_1},
\end{align}
where $\delta_1>0$ is a small constant.

Applying the standard estimate to system \eqref{1-eq2}, we obtain that
\begin{align}
	&\f12\f{d}{dt}
	\Big(
	\f{\beta}{\beta-1}\|\pa_z\nabla^5\A_0u_0^1\|_{L^2}^2
    +\|\pa_z\nabla^5\A_0(u_0^2,u_0^3)\|_{L^2}^2
	\Big)
	+
	\nu\Big(\f{\beta}{\beta-1}\|\pa_z\nabla^{6}\A_0u_0^1\|_{L^2}^2
    +\|\pa_z\nabla^{6}\A_0(u_0^2,u_0^3)\|_{L^2}^2
	\Big)
	\nonumber\\
	=&-\int_{\R^2}	\f{\beta}{\beta-1}\PP_0\A_0\pa_z\nabla^5(\uu\cdot\nabla_{\x_1} u^1)\cdot \pa_z\nabla^5\A_0u_0^1\,dydz
    \nonumber\\
    &
    -\sum_{j=2,3}\int_{\R^2}\PP_0\pa_z\nabla^5\A_0\big(\uu\cdot\nabla_{\x_1} u^j-\pa_j \Delta_{\x_1}^{-1}\dive_{\x_1}(\uu\cdot\nabla_{\x_1} \uu)\big)\cdot \pa_z\nabla^5\A_0 u_0^j \,dydz
	\tri  G_6.
	\label{energy-0-3}
\end{align}
Using \eqref{est-m0-0}, it follows directly that
\begin{align}\label{est-G6}
    \Big|\int_0^t G_6(\tau)\,d\tau\Big|
    \lesssim& \|\pa_z\nabla^5(\uu_0\cdot\nabla\uu_0)\|_{L^1_tL^2}
    \|\pa_z\nabla^5 \uu_0\|_{L^\infty_tL^2}
    \nonumber
    \\
    &+\sum_{j=1}^3\Big|\int_0^t \int_{\R^2}
    \pa_z\nabla^5\A_0\PP_0(\U_{\not=}\cdot\widetilde{\nabla}\U_{\not=}^j)\cdot \pa_z\nabla^5 \A_0 \uu_0^j \,dydzd\tau \Big|
    \nonumber
    \\
    \tri&~ G_{60}+G_{6\not=}.
\end{align}
The application of the \textit{a priori} assumption \eqref{priori assumption}, and \eqref{est-u-0}, yields that
\begin{align*}
    \| \pa_z\nabla^5(\uu_{0}\cdot\nabla\uu_{0})\|_{L^1_tL^2}
    \lesssim 
    \|(1+\pa_z)\uu_{0}\|_{L^2_tL^\infty}\|(1+\pa_z)\na\uu_{0}\|_{L^2_tH^5}
    \lesssim 
    \ep_1(\var_0+\ep_1^2\nu^{\al}),
\end{align*}
due to the fact that $\alpha>\f23$.
This implies directly that
\begin{align*}
    G_{60}
    \lesssim \ep_1^3\nu^{2\al}.
\end{align*}
As for $G_{6\not=}$, we use a similar derivation as \eqref{est-G5noq}-\eqref{est-G54noq} to find that
\begin{align*}
    G_{6\not=}\lesssim \ep_1^3\nu^{3\al-\f23}|\ln(\nu^{-1})|^{1+a_0}.
\end{align*}
Substituting these two estimates into \eqref{est-G6}, one gets that
\begin{align*}
    \Big|\int_0^t G_6(\tau)\,d\tau\Big| 
    \lesssim \var_0^2+
    \ep_1^3 \nu^{2\al}
    +\ep_1^3\nu^{3\al-\f23}|\ln(\nu^{-1})|^{1+a_0}.
\end{align*}
We then conclude that \eqref{est-step2} holds by substituting the above  estimate into \eqref{energy-0-3}.
\end{proof}

\section{Energy estimates for non-zero modes on \texorpdfstring{$\T\times\R^2$}{T x R2}}
\label{sec4:energy est nonzero}

In this section, we establish the \textit{a priori} energy estimates for the non-zero frequencies of the solution to system \eqref{3-eq3}. 
We thus begin by considering the Fourier-transformed version of system \eqref{3-eq2}. 
To this end, we obtain the following system:
\begin{align}\label{3-eq3}
\left\{\begin{array}{l}
	\displaystyle
	\pa_t \hat{Q}^k -C_{\beta} \f{i\eta}{p^{\f12}}\hat{W}^k+\nu p\hat{Q}^k
	=-\F[\tilde{\pa}_y\widetilde{\Delta} P^{NL}]-\F[\widetilde{\Delta}(\U\cdot\widetilde{\nabla}U^2)],
	\\[1ex]
	\displaystyle
	\pa_t \hat{W}^k-C_{\beta}  \f{i\eta}{p^{\f12}}\hat{Q}^k
	-\f12\f{\pa_{t}p}{p}\hat{W}^k
	+\nu p\hat{W}^k
	=\sqrt{\f{\beta}{\beta-1}} 
	\F\big[|\widetilde{\nabla}| 
	\big(
	\pa_{z}(\U\cdot\widetilde{\nabla}U^1)-\pa_{x}(\U \cdot\widetilde{\nabla}U^3)
	\big)\big],
\end{array}\right.
\end{align}
where $t\in\R^+$, $k\in\Z\backslash\{0\}$, and $(\xi,\eta)\in\R^2$.

Our main result is the following uniform \textit{a priori} estimate:
\begin{prop}\label{lemm-2}
Under the assumptions in Proposition \ref{prop-1},  we have for any $t\in[0,T]$,
\begin{align}\label{priori result-1}
	\sup_{\tau\in[0,t]}\E_{\not=}(\tau)
	+\int_0^t\D_{\not=}(\tau)\,d\tau
	\lesssim \var_0^2+\ep_1^3\nu^{3\alpha-\f23-\delta_1},
\end{align}
where $\delta_1>0$ is a small constant.
\end{prop}

\begin{proof}
Recalling the definitions of $\A$ and $\varUpsilon$ in \eqref{def-Af}, performing the standard energy estimate to system \eqref{3-eq3} in $H^4$-framework, and then applying Lemma \ref{sec2:lem-1}, we finally obtain that
\begin{align}	
	\label{energy-not=-1}
	&\f{d}{dt}\|\A(Q_{\not=},W_{\not=})\|_{H^4}^2
	+\nu \|\widetilde{\nabla}\A(Q_{\not=},W_{\not=})\|_{H^4}^2
	+\nu^{\f13} \|\A(Q_{\not=},W_{\not=})\|_{H^4}^2
	\nonumber
	\\
	&+\|\pa_{x}|\widetilde{\nabla}|^{-1}\A(Q_{\not=},W_{\not=})\|_{H^4}^2
	+\|\varUpsilon \A( Q_{\not=}, W_{\not=})\|_{H^4}^2
	\nonumber
	\\
	\lesssim&~
    \delta_1\nu^{\f13}\|\A(Q_{\not=},W_{\not=})\|_{H^4}^2+
	\Big|\sum_{s=0}^{4}
	\int_{\T\times\R^2} \nabla^s\PP_{\not=}\A \tilde{\pa}_y\widetilde{\Delta} P^{NL}
	\cdot
	\nabla^s\A Q_{\not=} \,dxdydz\Big|
	\nonumber
	\\
	&
	+
	\Big|\sum_{s=0}^{4}
	\int_{\T\times\R^2} \nabla^s\PP_{\not=}\A\widetilde{\Delta}(\U\cdot\widetilde{\nabla}U^2)
	\cdot
	\nabla^s\A Q_{\not=}
	\,dxdydz\Big|
	\nonumber
	\\
	&
	+\Big|\sum_{s=0}^{4}
	\int_{\T\times\R^2}
	\nabla^s\PP_{\not=}\A|\widetilde{\nabla}|\big(
	\pa_{z}(\U\cdot\widetilde{\nabla}U^1)-\pa_{x}(\U\cdot\widetilde{\nabla}U^3)
	\big)
	\cdot
	\nabla^s\A W_{\not=}
	\,dxdydz\Big|.
\end{align}
The nonlinear terms in the right-hand side of \eqref{energy-not=-1} will be
handled in the following three lemmas.
\begin{lemm}\label{lemm-2-1}
	Under the assumptions in Lemma \ref{lemm-2}, it holds that for any small constant $\delta_1>0$,
	\begin{align}\label{4est-1}
		\Big|\sum_{s=0}^{4}
		\int_0^t\int_{\T\times\R^2}\nabla^s\PP_{\not=}\A \tilde{\pa}_y\widetilde{\Delta} P^{NL}
		\cdot
		\nabla^s\A Q_{\not=} \,dxdydzd\tau \big|
		\lesssim \ep_1^3\nu^{3\al-\f23-\delta_1}.
	\end{align}
\end{lemm}
\begin{lemm}\label{lemm-2-2}
	Under the assumptions in Lemma \ref{lemm-2}, it holds that for any small constant $\delta_1>0$,
	\begin{align}\label{4est-2}
		\big|\sum_{s=0}^{4}
		\int_0^t
		\int_{\T\times\R^2} \nabla^s\PP_{\not=}\A\widetilde{\Delta}(\U\cdot\widetilde{\nabla}U^2)
		\cdot
		\nabla^s\A Q_{\not=} \,dxdydzd\tau \Big|
		\lesssim \ep_1^3\nu^{3\al-\f23-\delta_1}.
	\end{align}
\end{lemm}
\begin{lemm}\label{lemm-2-3}
	Under the assumptions in Lemma \ref{lemm-2}, it holds that for any small constant $\delta_1>0$,
	\begin{align}\label{4est-3}
		\Big| \sum_{s=0}^{4}
		\int_0^t
		\int_{\T\times\R^2}
		\nabla^s\PP_{\not=}\A|\widetilde{\nabla}|\big(
		\pa_{z}(\U\cdot\widetilde{\nabla}U^1)-\pa_{x}(\U\cdot\widetilde{\nabla}U^3)
		\big)
		\cdot
		\nabla^s\A W_{\not=}
		\,dxdydzd\tau \Big|
		\lesssim \ep_1^3\nu^{3\al-\f23-\delta_1}.
	\end{align}
\end{lemm}

We proceed under the temporary assumption of Lemmas \ref{lemm-2-1}-\ref{lemm-2-3}, deferring their proofs to the end of this subsection, and continue our proof of Lemma \ref{lemm-2}. 
Thus, integrating inequality \eqref{energy-not=-1} over $[0,t]$, and then substituting estimates \eqref{4est-1}-\eqref{4est-3} into the resulting inequality, we obtain the desired estimate \eqref{priori result-1}.
This completes the proof of Proposition \ref{lemm-2}.
\end{proof}

We now turn to the proofs of Lemmas \ref{lemm-2-1}–\ref{lemm-2-3}, thereby completing the proof of Proposition \ref{lemm-2}.

\begin{proof}[\underline{\textbf{Proof of Lemma \ref{lemm-2-1}}}]
We divide the nonlinear pressure term $P^{NL}$ as:
\begin{align*}
	\nabla^s\PP_{\not=}\A \widetilde{\Delta} P^{NL}
	\sim \nabla^s\PP_{\not=}\A (\pa_{i}U^j\pa_{j}U^i)
	+ \nabla^s\PP_{\not=}\A (\pa_{i}U^2\tilde{\pa}_yU^i)
	+\nabla^s\PP_{\not=}\A (\tilde{\pa}_yU^2)^2,
	\quad \text{for}~~i,j\in\{1,3\}.
\end{align*}
As the interactions $(0,0)$ cannot force non-zero modes, the relevant contributions arise from the $(0,\not=)$ and $(\not=\,\not=)$ interactions, which we proceed to bound.
We then claim that for $i,j\in\{1,3\}$,
\begin{align}
	&H_1=\sum_{s=0}^{4}
	\int_0^t\int_{\T\times\R^2}
	\big( |\nabla^s\A( \pa_{i}U^j_{\not=}\pa_{j}U^i_{\not=})| 
	+|\nabla^s\A (\tilde{\pa}_yU^2_{\not=})^2| 
	\big)
	|\nabla^s\A\tilde{\partial}_yQ_{\not=}|\,dxdydzd\tau
	\lesssim \ep_1^3\nu^{3\al-\f23};
	\label{est-4.5-1}
	\\
	&H_2=\Big|\sum_{s=0}^{4}
	\int_0^t\int_{\T\times\R^2}
	\nabla^s\A\tilde{\pa}_y(\pa_{i}U_{\not=}^2\tilde{\pa}_yU_{\not=}^i)
	\nabla^s\A Q_{\not=}\,dxdydzd\tau \Big|
	\lesssim \ep_1^3\nu^{3\al-\f23-\delta_1};
	\label{est-4.5-2}
	\\
	&H_3=\sum_{s=0}^{4}
	\int_0^t\int_{\T\times\R^2}
	|\nabla^s\A (\pa_yU^i_{0}\pa_{i}U_{\not=}^2)| 
|\nabla^s\A\tilde{\pa}_yQ_{\not=}|\,dxdydzd\tau
	\lesssim \ep_1^3\nu^{3\al-\f12};
	\label{est-4.5-3}
	\\
	&H_4=\sum_{s=0}^{4}
	\int_0^t\int_{\T\times\R^2}
	\big( |\nabla^s\A( \pa_{z}U^j_{0}\pa_{j}U^3_{\not=})| 
	+|\nabla^s\A (\pa_yU^2_{0}\tilde{\pa}_yU_{\not=}^2)| 
	\big)	|\nabla^s\A\widetilde{\nabla}Q_{\not=}|\,dxdydzd\tau
	\lesssim \ep_1^3\nu^{3\al-\f23};
	\label{est-4.5-4}
	\\
	&H_5=\sum_{s=0}^{4}
	\int_0^t\int_{\T\times\R^2}
	|\nabla^s\A( \pa_{z}U^2_{0}\tilde{\pa}_yU^3_{\not=})| 
	|\nabla^s\A\widetilde{\nabla}Q_{\not=}|\,dxdydzd\tau
	\lesssim \ep_1^3\nu^{3\al-\f23}.
	\label{est-4.5-5}
\end{align}
Combining \eqref{est-4.5-1}-\eqref{est-4.5-5}, we can obtain \eqref{4est-1} directly.

It remains to prove \eqref{est-4.5-1}-\eqref{est-4.5-5} in turn.
Firstly, we use the fact that the multiplier $\mathcal{A}$ is uniformly bounded and 
Lemma \ref{3lem-key} to estimate $H_1$ as
\begin{align*}
	H_1
	\lesssim &
	\big(\|\pa_{j}U^i_{\not=}\|_{L^\infty_tH^4}
	\|\pa_{i}U^j_{\not=}\|_{L^2_tH^4}
	+\|\tilde{\pa}_yU^2_{\not=}\|_{L^\infty_tH^4}\|\tilde{\pa}_yU^2_{\not=}\|_{L^2_tH^4}\big)
	\|\A\widetilde{\nabla}Q_{\not=}\|_{L^2_tH^4}
	\\
	\lesssim &~\nu^{-\f23} |\E_{\not=}(t)|^{\f12}\int_0^t\D_{\not=}(\tau)\,d\tau
	\lesssim
	\ep_1^3\nu^{3\al-\f23}.
\end{align*}
The estimate of $H_2$ is  more involved, and we bound it by the following three terms:
\begin{align}\label{est-H2-1}
	H_2
	\leq &\sum_{s=0}^{4}
	\int_0^t\int_{\T\times\R^2}	|\nabla^s\A(\tilde{\pa}_y\pa_{i}U_{\not=}^2\tilde{\pa}_yU_{\not=}^i)| 
	|\nabla^s\A Q_{\not=}|\,dxdydzd\tau
	\nonumber
	\\
	&+\sum_{s=0}^{4}
	\int_0^t\int_{\T\times\R^2}
	|\A([\nabla^s,\tilde{\pa}_y^2U_{\not=}^i]\pa_{i}U_{\not=}^2)| 
	|\nabla^s\A Q_{\not=}|\,dxdydzd\tau
	\nonumber
	\\
	&+\Big|\sum_{s=0}^{4}
	\int_0^t\int_{\T\times\R^2}
	\A(\nabla^s\pa_{i}U_{\not=}^2
	\tilde{\pa}_y^2U_{\not=}^i)
	\nabla^s\A Q_{\not=}\,dxdydzd\tau \Big|
	\nonumber
	\\
	\tri&~ H_{21}+H_{22}+H_{23}.
\end{align}
Applying Lemma \ref{3lem-key}, we find that
\begin{align*}
	H_{21}
	\lesssim 
	\|\tilde{\pa}_{y}U^i_{\not=}\|_{L^2_tH^4}
	\|\tilde{\pa}_{y}\pa_{i}U^2_{\not=}\|_{L^2_tH^4}
	\|\A Q_{\not=}\|_{L^{\infty}_tH^4}
	\lesssim \nu^{-\f23} |\E_{\not=}(t)|^{\f12} \int_0^t\D_{\not=}(\tau)\,d\tau
	\lesssim  \ep_1^3\nu^{3\al-\f23}.
\end{align*}
As for $H_{22}$, we use the fact 
\eqref{est-varphi-t}, the
estimate 
\begin{align}\label{est-u2not=-1}
	|\mathcal{F}[\pa_{i}U^2_{\not=}]|
	\lesssim |\mathcal{F}[|\widetilde{\nabla}|^{-1}\A Q_{\not=}]|
	\lesssim \langle t\rangle^{-1}  |\mathcal{F}[\nabla\A Q_{\not=}]|,
\end{align}
and Lemma \ref{3lem-key} to find
\begin{align*}
	H_{22}
	\lesssim& 
	\int_0^t\|\tilde{\pa}_{y} (Q_{\not=},W_{\not=})\|_{H^4}
	\||\widetilde{\nabla}|^{-1} \A Q_{\not=}\|_{H^3}
	\|\A Q_{\not=}\|_{H^4}\,d\tau
	\\
	\lesssim&~
	\|\tilde{\pa}_{y} \A(Q_{\not=},W_{\not=})\|_{L^2_tH^4}
	\|\A Q_{\not=}\|_{L^\infty_tH^4}
	\|\A Q_{\not=}\|_{L^2_tH^4}
	\\
	\lesssim&~ \nu^{-\f23} |\E_{\not=}(t)|^{\f12} \int_0^t\D_{\not=}(\tau)\,d\tau
	\lesssim  \ep_1^3\nu^{3\al-\f23}.
\end{align*}
Now, we deal with $H_{23}$, it follows from \eqref{est-inviscid}, \eqref{est-tau-1}, and Lemmas \ref{sec2:lem-0-0}, \ref{3lem-key} that
\begin{align*}
	H_{23}
	\lesssim &\sum_{s=0}^{4}
	\int_0^t \langle\tau\rangle
	\sum_{k,k'\in\Z\backslash\{0\}}\int_{\R^2}\int_{\R^2} 
	|(k-k',\xi-\xi')||(\hat{Q}_{\not=},\hat{W}_{\not=})(\tau,k-k',\xi-\xi',\eta-\eta')|
	\\
	&\times \langle \tau-\f{\xi'}{k'}\rangle^{-1} 
	|(k',\xi',\eta')|^{s}
	|\F[\A Q_{\not=}](\tau,k',\xi',\eta')|
	\cdot|(k,\xi,\eta)|^{s}
	|\F[\A Q_{\not=}](\tau,k,\xi,\eta)|\,d\xi' d\eta' d\xi d\eta d\tau
	\\
	\lesssim&~ |\ln(\nu^{-1})|^{1+a_0}
	\sum_{s=0}^{4}
	\int_0^t \langle\tau\rangle
	\sum_{k,k'\in\Z\backslash\{0\}}\int_{\R^2}\int_{\R^2} 
	|(k-k',\xi-\xi')|^{2+\f{a_0(2+a_0)}{2}}
	\\
	&\times
	|(\hat{Q}_{\not=},\hat{W}_{\not=})(\tau,k-k',\xi-\xi',\eta-\eta')|
	|(k',\xi',\eta')|^{s}
	\bigg|\sqrt{\f{\pa_t m_3}{m_3}}\F[\A Q_{\not=}](\tau,k',\xi',\eta')\bigg|
	\\
	&\times 
	|(k,\xi,\eta)|^{s}
	\bigg|\sqrt{\f{\pa_t m_3}{m_3}}\F[\A Q_{\not=}](\tau,k,\xi,\eta)\bigg|\,d\xi' d\eta' d\xi d\eta d\tau
	\\
	&+
	\nu \sum_{s=0}^{4} \int_0^t \langle\tau\rangle
	\sum_{k,k'\in\Z\backslash\{0\}}\int_{\R^2}\int_{\R^2} 
	|(k-k',\xi-\xi')||(\hat{Q}_{\not=},\hat{W}_{\not=})(\tau,k-k',\xi-\xi',\eta-\eta')|
	\\
	&\times 
	|(k',\xi',\eta')|^{s}
	|\F[\A Q_{\not=}](\tau,k',\xi',\eta')|
	\cdot|(k,\xi,\eta)|^{s}
	|\F[\A Q_{\not=}](\tau,k,\xi,\eta)|\,d\xi' d\eta' d\xi d\eta d\tau
	\\
	\lesssim&~
	\nu^{-\f23}|\ln(\nu^{-1})|^{1+a_0} \|\A(Q_{\not=},W_{\not=})\|_{L^\infty_t H^4}
	\|\varUpsilon\A( Q_{\not=},W_{\not=})\|_{L^2_t H^4}^2
	\\	&+\nu^{\f23}\|\A(Q_{\not=},W_{\not=})\|_{L^\infty_t H^4}
	\|\A Q_{\not=}\|_{L^2_t H^4}^2
	\\
	\lesssim&~\nu^{-\f23}|\ln(\nu^{-1})|^{1+a_0} |\E_{\not=}(t)|^{\f12}\int_0^t\D_{\not=}(\tau)\,d\tau
	\lesssim  \ep_1^3\nu^{3\al-\f23}|\ln(\nu^{-1})|^{1+a_0},
\end{align*}
by taking $0<a_0<\sqrt{2}-1$. Note that in the second last inequality, we have used the definition of $\A$ so that 
$\|(Q_{\neq},W_{\neq})(t)\|_{H^4}\lesssim e^{-\delta_2\nu^{\f13}t}\nu^{-\f13}\|\A(Q_{\neq},W_{\neq})(t)\|_{H^4}.$
Substituting the estimates of $H_{2i}$ $(i=1,2,3)$ into \eqref{est-H2-1}, we obtain that for any small constant $\delta_1>0$,
\begin{align}\label{est-H2-2}
	H_2\lesssim \ep_1^3\nu^{3\al-\f23-\delta_1}.
\end{align}
As for $H_3$, we can deduce from \eqref{est-inviscid} that
\begin{align*}
	H_3
	\lesssim 
	\|\pa_{y}U^i_{0}\|_{L^\infty_tH^4}
	\|\pa_{i}U^2_{\not=}\|_{L^2_tH^4}
	\|\A\widetilde{\nabla}Q_{\not=}\|_{L^2_tH^4}
	\lesssim \nu^{-\f12} |\E_{0}(t)|^{\f12}\int_0^t\D_{\not=}(\tau)\,d\tau
	\lesssim  \ep_1^3\nu^{3\al-\f12}.
\end{align*}
Again thanks to Lemma \ref{3lem-key}, one gets that
\begin{align*}
	H_4
	\lesssim 
	\|\nabla \U_{0}\|_{L^\infty_tH^4}
	\big(
	\|\tilde{\pa}_{y}U^2_{\not=}\|_{L^2_tH^4}
	+\|\pa_jU^3_{\not=}\|_{L^2_tH^4}\big)
	\|\A\widetilde{\nabla}Q_{\not=}\|_{L^2_tH^4}
	\lesssim \nu^{-\f23} |\E_{0}(t)|^{\f12} \int_0^t\D_{\not=}(\tau)\,d\tau
	\lesssim  \ep_1^3\nu^{3\al-\f23}.
\end{align*}
In view of \eqref{est-comm}, the definition of $\D_{\neq}$ in \eqref{defDnot}, and \eqref{def-u3not}, we find that
\begin{align*}
	H_5
	\lesssim 
	\|\pa_z U_{0}^2\|_{L^\infty_tH^5}
	\|\A \tilde{\pa}_{y} U^3_{\not=}\|_{L^2_tH^4}
	\|\A\widetilde{\nabla}Q_{\not=}\|_{L^2_tH^4}
	\lesssim \nu^{-\f23} |\E_{0}(t)|^{\f12}\int_0^t\D_{\not=}(\tau)\,d\tau
	\lesssim  \ep_1^3\nu^{3\al-\f23}.
\end{align*}
The proof of this lemma is complete.
\end{proof}

\begin{proof}[\underline{\textbf{Proof of Lemma \ref{lemm-2-2}}}]
Since the interaction $(0,\not=)$ and $(\not=,\not=)$ can force non-zero modes, the transport term $\PP_{\not=}\A\widetilde{\nabla}(\U\cdot\widetilde{\nabla}U^2)$ can be rewritten as: for any $j\in\{1,3\}$,
\begin{align*}
	\PP_{\not=}\A\widetilde{\nabla}(\U\cdot\widetilde{\nabla}U^2)
	=&~	\PP_{\not=}\A\big(U_{\not=}^j\pa_{j}\widetilde{\nabla}U_{\not=}^2
	+\widetilde{\nabla}U_{\not=}^2\widetilde{\pa}_y U_{\not=}^2\big)
	+\PP_{\not=}\A\big(\widetilde{\nabla}U_{\not=}^j\pa_{j}U_{\not=}^2
	+U_{\not=}^2\widetilde{\pa}_y \widetilde{\nabla}U_{\not=}^2\big)
	\\
	&+\A\big(\nabla \U_0\cdot \widetilde{\nabla}U_{\not=}^2\big)
	+\A\big(\U_0\cdot \widetilde{\nabla}^2U_{\not=}^2\big)
	+\A\big(U_{\not=}^2\pa_y \nabla U_{0}^2\big)
	\\
	&+\A\big(\widetilde{\nabla}U_{\not=}^2\pa_y U_{0}^2\big)
	+\A\big(U_{\not=}^3\pa_z \nabla U_{0}^2\big)
	+\A\big(\widetilde{\nabla}U_{\not=}^3\pa_z U_{0}^2\big).
\end{align*}
We claim that for $j\in\{1,3\}$,
\begin{align}
	&I_1=\sum_{s=0}^{4}
	\int_0^t\int_{\T\times\R^2}
	\big( |\nabla^s\A(U_{\not=}^j\pa_{j}\widetilde{\nabla}U_{\not=}^2)| 
	+|\nabla^s\A(\widetilde{\nabla}U_{\not=}^2\widetilde{\pa}_y U_{\not=}^2)| 
	\big)
	|\nabla^s\A\widetilde{\nabla}Q_{\not=}|\,dxdydzd\tau
	\lesssim \ep_1^3\nu^{3\al-\f23};
	\label{est-4.6-1}
	\\
	&I_2=\Big|\sum_{s=0}^{4}
	\int_0^t\int_{\T\times\R^2}
	\nabla^s\A\widetilde{\dive}(\widetilde{\nabla}U_{\not=}^j\pa_{j}U_{\not=}^2) 
	\nabla^s\A Q_{\not=} \,dxdydzd\tau \Big|
	\lesssim \ep_1^3\nu^{3\al-\f23-\delta_1};
	\label{est-4.6-2}
	\\
	&I_3=\Big|\sum_{s=0}^{4}
	\int_0^t\int_{\T\times\R^2}
	\nabla^s\PP_{\not=}\A\widetilde{\dive}( U_{\not=}^2\widetilde{\pa}_y\widetilde{\nabla}U_{\not=}^2)
	\nabla^s\A Q_{\not=}\,dxdydzd\tau \Big|
	\lesssim \ep_1^3\nu^{3\al-\f23-\delta_1};
	\label{est-4.6-2-1}
	\\
	&I_4=\sum_{s=0}^{4}
	\int_0^t\int_{\T\times\R^2}
	\big( |\nabla^s\A(\nabla \U_0\cdot \widetilde{\nabla}U_{\not=}^2)| 
	+|\nabla^s\A(\widetilde{\nabla}U_{\not=}^2\pa_y U_{0}^2)| 
	\big)
	|\nabla^s\A\widetilde{\nabla}Q_{\not=}|\,dxdydzd\tau
	\lesssim \ep_1^3\nu^{3\al-\f23};
	\label{est-4.6-3}
	\\
	&I_5=\sum_{s=0}^{4}
	\int_0^t\int_{\T\times\R^2}
	|\nabla^s\A(U_{\not=}^3\pa_z \nabla U_{0}^2)| 
	|\nabla^s\A\widetilde{\nabla}Q_{\not=}|\,dxdydzd\tau
	\lesssim \ep_1^3\nu^{3\al-\f23};
	\label{est-4.6-4}
	\\
	&I_6=\sum_{s=0}^{4}
	\int_0^t\int_{\T\times\R^2}
	\big( |\nabla^s\A(\U_0\cdot \widetilde{\nabla}^2U_{\not=}^2)| 
	+|\nabla^s\A(\widetilde{\nabla}U_{\not=}^3\pa_z U_{0}^2)| 
	\big)
	|\nabla^s\A\widetilde{\nabla}Q_{\not=}|\,dxdydzd\tau
	\lesssim \ep_1^3\nu^{3\al-\f23};
	\label{est-4.6-5}
	\\
	&I_7=\sum_{s=0}^{4}
	\int_0^t\int_{\T\times\R^2}
	|\nabla^s\A(U_{\not=}^2\pa_{y}\nabla U_{0}^2)| 
	|\nabla^s\A\widetilde{\nabla}Q_{\not=}|\,dxdydzd\tau
	\lesssim \ep_1^3\nu^{3\al-\f12}.
	\label{est-4.6-6}
\end{align}	
The combination of \eqref{est-4.6-1}-\eqref{est-4.6-6} yields \eqref{4est-2} immediately.

To complete the proof of this lemma, we intend to prove \eqref{est-4.6-1}-\eqref{est-4.6-6}.
We employ Lemma \ref{3lem-key} to discover that
\begin{align*}
	I_1
	\lesssim&
	\big(\|U_{\not=}^j\|_{L^\infty_tH^4}\|\nabla_{x,z}\widetilde{\nabla}U_{\not=}^2\|_{L^2_tH^4}
	+
	\|\widetilde{\nabla}U_{\not=}^2\|_{L^\infty_tH^4}\|\widetilde{\nabla}U_{\not=}^2\|_{L^2_tH^4}
	\big)
	\|\A\widetilde{\nabla}Q_{\not=}\|_{L^2_tH^4}
	\\
	\lesssim& ~\nu^{-\f23} |\E_{\not=}(t)|^{\f12}\int_0^t\D_{\not=}(\tau)\,d\tau
	\lesssim  \ep_1^3\nu^{3\al-\f23}.
\end{align*}
While for $I_2$, by an argument analogous to that used for $H_2$, we find that
\begin{align*}
	I_{2} \lesssim \ep_1^3\nu^{3\al-\f23-\delta_1}.
\end{align*}
We now estimate $I_3$ via the following decomposition:
\begin{align*}
	I_3
	\lesssim &
	\sum_{s=0}^{4}
	\int_0^t\int_{\T\times\R^2}
	|\nabla^s\PP_{\not=}\A( \widetilde{\nabla}U_{\not=}^2  \widetilde{\partial}_y\widetilde{\nabla}U_{\not=}^2)|
	|\nabla^s\A Q_{\not=}|\,dxdydzd\tau
	\\
	&+\sum_{s=0}^{4}
	\int_0^t\int_{\T\times\R^2}
	|\PP_{\not=}\A( [\nabla^s,\widetilde{\pa}_yQ_{\not=}] U_{\not=}^2)|
	|\nabla^s\A Q_{\not=}|\,dxdydzd\tau
	\\
	&\sum_{s=0}^{4}
	\Big|\int_0^t\int_{\T\times\R^2}
	\PP_{\not=}\A( \nabla^sU_{\not=}^2\widetilde{\pa}_yQ_{\not=})\nabla^s\A Q_{\not=}\,dxdydzd\tau \Big|
	\\
	\tri &~I_{31}+I_{32}+I_{33}.
\end{align*}
For $(i=1,2,3)$, the term $I_{3i}$ can be estimated similarly to $H_{2i}$.
Accordingly, we have
\begin{align*}
	I_{3} \lesssim \ep_1^3\nu^{3\al-\f23-\delta_1}.
\end{align*}
In view of Lemma \ref{3lem-key}, we obtain that
\begin{align*}
	I_4
	&\lesssim
	\|\nabla \U_{0}\|_{L^\infty_tH^4}
	\|\widetilde{\nabla}U_{\not=}^2\|_{L^2_tH^4}
	\|\A\widetilde{\nabla}Q_{\not=}\|_{L^2_tH^4}
	\lesssim \nu^{-\f23} |\E_{0}(t)|^{\f12} \int_0^t\D_{\not=}(\tau)\,d\tau
	\lesssim  \ep_1^3\nu^{3\al-\f23},
	\\
	I_5
	&\lesssim
	\|\pa_{z}U_{0}^2\|_{L^\infty_tH^5}\|U_{\not=}^3\|_{L^2_tH^4}
	\|\A\widetilde{\nabla}Q_{\not=}\|_{L^2_tH^4}
	\lesssim \nu^{-\f23} |\E_{0}(t)|^{\f12} \int_0^t\D_{\not=}(\tau)\,d\tau
	\lesssim  \ep_1^3\nu^{3\al-\f23}.
\end{align*}
As for $I_6$, recalling \eqref{est-comm}, the definition of $\D_{\neq}$ in \eqref{defDnot}, \eqref{def-u2not} and \eqref{def-u3not}, we find that
\begin{align*}
	I_6
	\lesssim&~
	\|(1+\pa_{z})\U_{0}\|_{L^\infty_tH^5}
	\big(\|\A \widetilde{\nabla}^2U_{\not=}^2\|_{L^2_tH^4}
	+\|\A \widetilde{\nabla}U_{\not=}^3\|_{L^2_tH^4}\big)
	\|\A\widetilde{\nabla}Q_{\not=}\|_{L^2_tH^4}
	\\
	\lesssim&~ \nu^{-\f23} |\E_{0}(t)|^{\f12} \int_0^t\D_{\not=}(\tau)\,d\tau
	\lesssim  \ep_1^3\nu^{3\al-\f23}.
\end{align*}
For the last term, we use the 
divergence free condition, the inequality \eqref{est-inviscid} and Lemma \ref{3lem-key} to get that
\begin{align*}
	I_7
	\lesssim
	\|\pa_{z}\nabla U_{0}^3\|_{L^\infty_tH^4}
	\| U_{\not=}^2\|_{L^2_tH^4}
	\|\A\widetilde{\nabla}Q_{\not=}\|_{L^2_tH^4}
	\lesssim \nu^{-\f12} |\E_{0}(t)|^{\f12} \int_0^t\D_{\not=}(\tau)\,d\tau
	\lesssim  \ep_1^3\nu^{3\al-\f12}.
\end{align*}
This finishes the proof of Lemma \ref{lemm-2-2}.

\end{proof}

\begin{proof}[\underline{\textbf{Proof of Lemma \ref{lemm-2-3}}}]
Recalling that only the interactions $(0,\not=)$ and $(\not=,\not=)$ are capable of generating non-zero modes, we proceed to rewrite
$\pa_{z}(\U\cdot\widetilde{\nabla}U^1)-\pa_{x}(\U\cdot\widetilde{\nabla}U^3)$ in the form:
\begin{align*}
	&|\PP_{\not=}\A(\pa_{z}(\U\cdot\widetilde{\nabla}U^1)-\pa_{x}(\U\cdot\widetilde{\nabla}U^3))|
	\\
	\lesssim &~
	|\PP_{\not=}\A\pa_{4-i}(U^j_{\not=}\pa_{j}U^i_{\not=})|
	+|\PP_{\not=}\A\pa_{4-i}(U^2_{\not=}\tilde{\pa}_{y}U^i_{\not=})|
	+|\A\pa_{4-i}(\U_{0}\cdot\widetilde{\na}U^{i}_{\not=})|
	\\
	&
    +|\A\pa_{4-i}(\U_{\not=}\cdot\nabla U^{i}_{0})|,
	\quad\text{for}~~i,j\in\{1,3\}.
\end{align*}
We claim that for $i,j\in\{1,3\}$,
\begin{align}
	&J_1=\sum_{s=0}^{4}
	\int_0^t\int_{\T\times\R^2}
	|\nabla^s\PP_{\not=}\A\pa_{4-i}(U^j_{\not=}\pa_{j}U^i_{\not=})|
	|\nabla^s\A\widetilde{\nabla}W_{\not=}|\,dxdydzd\tau
	\lesssim \ep_1^3\nu^{3\al-\f23};
	\label{est-4.7-1}
	\\
	&J_2=\Big| \sum_{s=0}^{4}
	\int_0^t\int_{\T\times\R^2}
	\nabla^s\PP_{\not=}\A|\widetilde{\nabla}|\pa_{4-i}(U^2_{\not=}\tilde{\pa}_{y}U^i_{\not=})
	\nabla^s\A W_{\not=} \,dxdydzd\tau \Big|
	\lesssim \ep_1^3\nu^{3\al-\f23-\delta_1};
	\label{est-4.7-2}
	\\
	&J_3=\sum_{s=0}^{4}
	\int_0^t\int_{\T\times\R^2}
	|\nabla^s \A\pa_{4-i}(\U_{0}\cdot\widetilde{\na}U^{i}_{\not=})|
|\nabla^s\A\widetilde{\nabla}W_{\not=}|\,dxdydzd\tau
	\lesssim \ep_1^3\nu^{3\al-\f23};
	\label{est-4.7-3}
	\\
	&J_4=\sum_{s=0}^{4}
	\int_0^t\int_{\T\times\R^2}
	|\nabla^s\A\pa_{4-i}(\U_{\not=}\cdot\nabla U^{i}_{0})|
	|\nabla^s\A\widetilde{\nabla}W_{\not=}|\,dxdydzd\tau
	\lesssim \ep_1^3\nu^{3\al-\f23}.
	\label{est-4.7-4}
\end{align}
Collecting \eqref{est-4.7-1}-\eqref{est-4.7-4}, we can obtain \eqref{4est-3} directly.

In order to complete the proof of this lemma, we shall focus on the proofs of
\eqref{est-4.7-1}-\eqref{est-4.7-4}.
By using Lemma \ref{3lem-key}, it follows that
\begin{align*}
	J_1
	\lesssim& 
	\big(\|\nabla_{x,z}U^j_{\not=}\|_{L^\infty_tH^4}\|\nabla_{x,z}U^i_{\not=}\|_{L^2_tH^4}
	+
	\|U^j_{\not=}\|_{L^\infty_tH^4}\|\nabla_{x,z}^2U^i_{\not=}\|_{L^2_tH^4}\big)
	\|\A\widetilde{\nabla}W_{\not=}\|_{L^2_tH^4}
	\\
	\lesssim& ~\nu^{-\f23}|\E_{\not=}(t)|^{\f12}\int_0^t\D_{\not=}(\tau)\,d\tau
	\lesssim \ep_1^3\nu^{3\al-\f23}.
\end{align*}
It is easy to observe that
\begin{align*}
    J_2=\Big| \sum_{s=0}^{4}
	\int_0^t\int_{\T\times\R^2}
	\nabla^s\PP_{\not=}\A\widetilde{\nabla}\pa_{4-i}(U^2_{\not=}\tilde{\pa}_{y}U^i_{\not=})\cdot
	\widetilde{\nabla}|\widetilde{\nabla}|^{-1}\nabla^s\A W_{\not=} \,dxdydzd\tau \Big|.
\end{align*}
Hence, we can bound $J_2$ as :
\begin{align*}
	J_2\lesssim &\sum_{s=0}^{4}
	\int_0^t\int_{\T\times\R^2}
	|\nabla^s\PP_{\not=}\A\pa_{4-i}(\widetilde{\nabla}U^2_{\not=}\tilde{\pa}_{y}U^i_{\not=})|
	|\nabla^s\A W_{\not=}|\,dxdydzd\tau
	\\
	&+\sum_{s=0}^{4}
	\int_0^t\int_{\T\times\R^2}
	|\PP_{\not=}\A\pa_{4-i}([\nabla^s,\tilde{\pa}_{y}\widetilde{\nabla}U^i_{\not=}]U^2_{\not=})|
	|\nabla^s\A W_{\not=}|\,dxdydzd\tau\\&+\sum_{s=0}^{4}\Big|
	\int_0^t\int_{\T\times\R^2}
	\PP_{\not=}\A\pa_{4-i}(\nabla^sU^2_{\not=}\tilde{\pa}_{y}\widetilde{\nabla}U^i_{\not=})
    \cdot
	\widetilde{\nabla}|\widetilde{\nabla}|^{-1}\nabla^s\A W_{\not=}\,dxdydzd\tau\Big|
\end{align*}
It can be estimated following the same method as for $H_2$ that
\begin{align*}
	J_2 \lesssim \ep_1^3\nu^{3\al-\f23-\delta_1}.
\end{align*}
In view of \eqref{est-comm}, the definition of $\D_{\neq}$ in \eqref{defDnot}, \eqref{def-u1not}, and \eqref{def-u3not}, we get that
\begin{align*}
	J_3
	\lesssim& ~
	\big(\|(1+\pa_z)U^2_{0}\|_{L^\infty_tH^5}\|\A\nabla_{x,z}\tilde{\pa}_yU^i_{\not=}\|_{L^2_tH^4}
	\\
	&
	+
	\|(1+\pa_z)(U_0^1,U^3_{0})\|_{L^\infty_tH^5}\|\A\nabla_{x,z}^2U^i_{\not=}\|_{L^2_tH^4}\big)
	\|\A\widetilde{\nabla}W_{\not=}\|_{L^2_tH^4}
	\\
	\lesssim&~ \nu^{-\f23}|\E_{0}(t)|^{\f12}\int_0^t\D_{\not=}(\tau)\,d\tau
	\lesssim \ep_1^3\nu^{3\al-\f23}.
\end{align*}
Applying \eqref{est-inviscid} and Lemma \ref{3lem-key} again, one can deduce that
\begin{align*}
	J_4
	\lesssim 
	\|\nabla_{x,z}\U_{\not=}\|_{L^2_tH^4}\|(1+\pa_{z}) \U_0\|_{L^\infty_tH^5}\|\A\widetilde{\nabla}W_{\not=}\|_{L^2_tH^4}
	\lesssim \nu^{-\f23}|\E_{0}(t)|^{\f12}\int_0^t\D_{\not=}(\tau)\,d\tau
	\lesssim \ep_1^3\nu^{3\al-\f23}.
\end{align*}
This concludes the proof of Lemma \ref{lemm-2-3}.

\end{proof}

\section{Proof of  Theorems \ref{theo1} and Theorem \ref{theo2}}\label{sec5:proof-TRR}



We now prove Theorem \ref{theo1} following the standard three-step approach for global existence theory for nonlinear PDEs:
\begin{itemize}
\item[(1)] constructing approximate solutions;
\item[(2)] establishing the uniform  \textit{a priori} energy estimates;
\item[(3)] passing to the limit.
\end{itemize}
While steps (1) and (3) can be obtained through standard arguments, the essential difficulty lies in step (2), namely, establishing \textit{a priori} energy estimates for smooth solutions of system \eqref{1-eq1}.
The energy estimates rigorously derived in Proposition \ref{prop-1}, combined with standard continuity arguments, lead to Theorem \ref{theo1}.
Accordingly, we focus on the proof of Proposition \ref{prop-1}.

\begin{proof}[\textbf{\underline{Proof of Proposition \ref{prop-1} on $\T\times\R^2$}}]
Combining Propositions \ref{lemm-1}, \ref{lemm-2}, we obtain that for any $\alpha>\f23$,
\begin{align*}
	\sup_{\tau\in[0,t]}\big(\E_0(\tau)+\E_{\not=}(\tau)\big)
	+\int_0^t\big(\D_0(\tau)+\D_{\not=}(\tau)\big)\,d\tau
	\leq  C\var_0^2+C\ep_1^3\nu^{2\al}+C\ep_1^3\nu^{3\alpha-\f23-\delta_1}
    \leq C\var_0^2+C\ep_1^3\nu^{2\al}.
\end{align*}
These two estimates yield \eqref{sec3:priori result-1} immediately.
We thereby finish the proof of Proposition \ref{prop-1} on $\T\times\R^2$.
\end{proof}

At the end of this section, we give the proof of Theorem \ref{theo2}.
\begin{proof}[\textbf{\underline{Proof of Theorem \ref{theo2}}}]
It then follows from Lemma \ref{lemm-disp} and \eqref{sec3:priori result-1} that
\begin{align}\label{est-u01-decay}
    &\|(u_0^1,\cR_2\up_0,\cR_3\up_0)(t)\|_{L^\infty}
    \nonumber\\
    \lesssim & ~(1+t)^{-\f12} \|\uu_{0in}\|_{W^{2^+,1}}+\nu^{-\delta_1}\|\mathbb{P}_0(\uu\cdot\na_{\x_1}\uu)\|_{L^2_tW^{2,1}}
    \nonumber\\
    \lesssim & ~(1+t)^{-\f12} \|\uu_{0in}\|_{W^{2^+,1}}+\ep_1^2\nu^{2\al-\f12-\delta_1},
\end{align}
and
    \begin{align}\label{est-payu01-decay}
&\|\pa_y(u_0^1,\cR_2\up_0,\cR_3\up_0,|\nabla|^{-1}\up_0)(t)\|_{L^\infty}
\nonumber\\
   \lesssim & ~(1+t)^{-1} \|\uu_{0in}\|_{W^{3^+,1}}+\nu^{-\delta_1}\|\mathbb{P}_0(\uu\cdot\na_{\x_1}\uu)\|_{L_t^{\infty}W^{3,1}}
   \nonumber\\
    \lesssim & ~(1+t)^{-1} \|\uu_{0in}\|_{W^{3^+,1}}+\ep_1^2\nu^{2\al-\delta_1}.
\end{align}
These two estimates, along with \eqref{rep-u023}, imply
\eqref{Asy-u0-1} and \eqref{Asy-u0-2}. Here in \eqref{est-u01-decay}, we used that
\begin{align*}
    \|\mathbb{P}_0(\uu\cdot\na_{\x_1}\uu)\|_{L^2_tW^{2,1}}
    \leq &\|\mathbb{P}_0(\uu_0\cdot\na_{\x_1}\uu_0)\|_{L^2_tW^{2,1}}+\|\mathbb{P}_0(\uu_{\neq}\cdot\na_{\x_1}\uu_{\neq})\|_{L^2_tW^{2,1}}\\
    =&\|\mathbb{P}_0(\uu_0\cdot\na\uu_0)\|_{L^2_tW^{2,1}}+\|\mathbb{P}_0(U_{\neq}\cdot\widetilde{\na}U_{\neq})\|_{L^2_tW^{2,1}}\\
    \leq&\|\uu_0\|_{L_t^{\infty}H^2}\|\na\uu_0\|_{L^2_{t}H^2}+\|U_{\neq}\|_{L_t^{\infty}H^2}\|\widetilde{\na}U_{\neq}\|_{L^2_{t}H^2}
    \lesssim\ep_1^2\nu^{2\al-\f12}.
\end{align*}
And in the derivation of \eqref{est-payu01-decay}, we have used the fact that
\begin{align*}
    &\|\mathbb{P}_0(\uu\cdot\na_{\x_1}\uu)\|_{L^{\infty}_tW^{3,1}}\leq \|\mathbb{P}_0(\uu_0\cdot\na\uu_0)\|_{L^{\infty}_tW^{3,1}}+\|\mathbb{P}_0(\uu_{\neq}\cdot\na_{\x_1}\uu_{\neq})\|_{L^{\infty}_tW^{3,1}},\\
    &\|\mathbb{P}_0(\uu_0\cdot\na\uu_0)\|_{L^{\infty}_tW^{3,1}}\lesssim\|\uu_0\|_{L^{\infty}_tH^3}\|\uu_0\|_{L^{\infty}_tH^4}
    \lesssim\ep_1^2\nu^{2\al},
    \end{align*}
and
    \begin{align*}
    \|\mathbb{P}_0(\uu_{\neq}\cdot\na_{\x_1}\uu_{\neq})\|_{L^{\infty}_tW^{3,1}}=&
    \|\mathbb{P}_0(U_{\neq}\cdot\widetilde{\na}U_{\neq})\|_{L^{\infty}_tW^{3,1}}\\ 
    \lesssim&\sum_{j=1,3}
    \|U_{\neq}^j\|_{L^{\infty}_{t}H^3}\|\partial_jU_{\neq}\|_{L^{\infty}_tH^{3}}+
    \|U_{\neq}^2\|_{L^{\infty}_{t}H^3}\|\tilde{\partial}_yU_{\neq}\|_{L^{\infty}_tH^{3}}
  \\ \lesssim&~ \ep_1^2\nu^{2\al}+\||\widetilde{\na}|^{-1} \A Q_{\not=}\|_{L^{\infty}_{t}H^3}
    \| (Q_{\not=},W_{\not=})\|_{L^{\infty}_{t}H^3}\\ \lesssim& ~\ep_1^2\nu^{2\al}+\langle t\rangle^{-1}\| \A Q_{\not=}\|_{L^{\infty}_{t}H^4}
    \| (Q_{\not=},W_{\not=})\|_{L^{\infty}_{t}H^3}\\ \lesssim& ~\ep_1^2\nu^{2\al}+\| \A Q_{\not=}\|_{L^{\infty}_{t}H^4}
    \| \A(Q_{\not=},W_{\not=})\|_{L^{\infty}_{t}H^3}\lesssim \ep_1^2\nu^{2\al}.
\end{align*}
We now focus on proving \eqref{Asy-unot}.
It follows from Proposition \ref{prop-1} and continuity arguments that there is $\delta_2>0$ such that
\begin{align}\label{est-en-diss}
\sup_{t\geq 0}\|\A(Q_{\not=},W_{\not=})\|_{H^4}^2+\delta_2\nu^{\f13}\|\A(Q_{\not=},W_{\not=})\|_{L_t^2H^4}^2\lesssim \varepsilon_0^2\,.
\end{align}
By applying \eqref{est-varphi-1}, \eqref{def-u2not} and \eqref{est-en-diss}, we find that
\begin{align}\label{est-u2}
	\|u^2_{\not=}\|_{L^2}
	=\|U^2_{\not=}\|_{L^2}
	\lesssim e^{-\delta_2\nu^{\f13}t}\||\widetilde{\nabla}|^{-1}\A Q_{\not=}\|_{L^2}
	\lesssim (1+t)^{-1}e^{-\delta_2\nu^{\f13}t} \|\A Q_{\not=}\|_{H^1}
	\lesssim (1+t)^{-1} e^{-\delta_2\nu^{\f13}t}\var_0.
\end{align}
As for $u^1_{\not=}$ and $u^3_{\not=}$, we can use \eqref{def-u1not}, \eqref{def-u3not} and \eqref{est-en-diss} to derive that
\begin{align}\label{est-u1u3}
	\|(u^1_{\not=},u^3_{\not=})\|_{L^2}
	=&\|(U^1_{\not=},U^3_{\not=})\|_{L^2}
	\lesssim  e^{-\delta_2\nu^{\f13}t} \|\nabla_{x,z}(\pa_x^2+\pa_z^2)^{-1}\A (Q_{\not=}, W_{\not=})\|_{L^2}\lesssim e^{-\delta_2\nu^{\f13}t}\var_0\,. 
\end{align}
Combining \eqref{est-u2} and \eqref{est-u1u3}, we conclude that \eqref{Asy-unot} holds true.
\end{proof}


\section{Proof of  Theorems \ref{theo3} and Theorem \ref{theo4}}  \label{sec5:proof-TRT}


The objective of this section is to prove Theorems \ref{theo3} and \ref{theo4} for the domain $\T\times\R\times \T$. As with the case of $\T\times\R^2$, Theorem \ref{theo3} can be established via a continuity argument combined with the \textit{a priori} estimates given in Proposition \ref{prop-1}. For Theorem \ref{theo4}, \eqref{Asy-u0-1'} follows from Lemma \ref{lem-disp-T} (and a similar argument as in Theorem \ref{theo2}), and the proof of \eqref{Asy-unot'} is similar to Theorem \ref{theo2}.
We therefore present only those parts of the proof that differ from the previous case.

Before proceeding to establish the energy estimates for the zero modes, we first give the dispersive and Strichartz estimates. For any $f(y,z)$, we denote 
    \begin{align*}
      f=\sum_{l\in\Z} f^l(y) e^{ilz},\quad \text{with}~~
       f^l(y)\tri\f{1}{2\pi}\int_{\T}f(y,z)e^{-ilz}\,dz.
    \end{align*}
    We also define the corresponding operators: 
    \begin{align*}
        \mathcal{R}_3^l\tri -il(l^2-\pa_y^2)^{-\f12},
        \quad
        \mathcal{L}^{\pm,l}\tri \pm\sqrt{B_{\beta}}\mathcal{R}_3^l+\nu(\pa_y^2-l^2).
    \end{align*}
    Recall the definition of $\W^{\pm}$
   in \eqref{def-W-1}. 
   Taking the Fourier transform  in $z,$ it is found that 
     $\W^{\pm,l}$ is given by
    \begin{align}\label{est-Wpm-T}
       \W^{\pm,l} = e^{t\mathcal{L}^{\pm,l}} \W^{\pm,l}_{in}+\int_0^t e^{(t-\tau)\mathcal{L}^{\pm,l}} N^{\pm,l} (\tau)\, d\tau.
    \end{align}
    We first derive the following lemma from Lemma \ref{lem:disp-T} and the Minkowski inequality, the proof of which is omitted.
\begin{lemm}\label{lem-disp-T}
    For any $l\not=0$, it holds that
    \begin{align*}
        \|(u_0^{1,l},\up_0^l)(t)\|_{L^\infty_y(\R)} 
        \lesssim &
        (1+t)^{-\f13}
        \big(|l|\|\uu_{0in}^l\|_{L^1_y(\R)}+|l|^{-\f12}\|\uu_{0in}^l\|_{W^{(\f32)^+,1}_y(\R)}\big)
        \\
        &+\int_0^t(1+\tau)^{-\f13}e^{-\nu l^2\tau}
        \big(|l|\|\PP_0(\uu\cdot\nabla_{\x_1}  \uu)^l\|_{L^1_y(\R)}+|l|^{-\f12}\|\PP_0(\uu\cdot\nabla_{\x_1}  \uu)^l\|_{W^{(\f32)^+,1}_y(\R)}\big)\,d\tau,\\
        \|\partial_z(u_0,\mathcal{R}_2\up_0,\mathcal{R}_3\up_0&,|\nabla|^{-1}\up_0)(t)\|_{L^\infty(\R\times\T)}\lesssim 
        (1+t)^{-\f13}
       \|\uu_{0in}\|_{W^{3^+,1}(\R\times\T)}+\nu^{-\f16}\|\PP_0(\uu\cdot\nabla_{\x_1}  \uu)\|_{L^2_tW^{3,1}(\R\times\T)}.
    \end{align*}
\end{lemm}
Based on the dispersive estimate shown in Lemma \ref{lem-disp-T}, we can deduce from Lemma \ref{lem:disp-T} and the similar derivation of Lemma \ref{lemm-str} that
\begin{lemm}\label{lemm-str-T}
It holds that
\begin{align}\label{est-str-u02-T}
&	\nu^{\f13}\|(1+\pa_z)u_0^2\|_{L^2_tL^\infty_yL^2_z}
	+\nu^{\f56}\|(1+\pa_z)\pa_z\uu_0\|_{L^1_t L^\infty_yL^2_z}\nonumber\\
	\lesssim& \|(1+\pa_z)\uu_{0in}\|_{L^{2}}
    +\|(1+\pa_z)\PP_0(\uu\cdot\nabla_{\x_1}  \uu)\|_{L^1_t{L^2}}.
\end{align}

\end{lemm}

\begin{proof}
    
     Taking $ \gamma=\sqrt{B_{\beta}}$ in \eqref{disp-RT}, we find that for $l\neq 0,$
    \begin{align*}
        \| e^{t\mathcal{L}^{\pm,l}  } \cR_3^l\W^{\pm,l}_{in} \|_{L^2_t L^{\infty}_y}
        \lesssim \nu^{-\f13}|l|^{-\f16}\|\W^{\pm,l}_{in}\|_{L^2_y}
    \end{align*}
 We thus obtain that
    \begin{align*}
        \| e^{t \mathcal{L}^{\pm}} \mathcal{R}_3\W^{\pm}_{in} \|_{L^2_t L^{\infty}_yL^2_z}
        \lesssim  \Big(\sum_{l\in  \Z \backslash \{0\}}
        \| e^{t \mathcal{L}^{\pm,l}} \mathcal{R}_3^l\W_{in}^{\pm,l} \|_{L^2_t L^{\infty}_y}^2\Big)^{\f12}\lesssim 
        \nu^{-\f13}\|\W^{\pm,l}_{in}\|_{L^2}.
    \end{align*}
  In view of \eqref{est-Wpm-T}, we apply the Minkowski inequality in time to derive that
    \begin{align}\label{est-R3-RT1}
        \| e^{t \mathcal{L}^{\pm}} \mathcal{R}_3\W^{\pm} \|_{L^2_t L^{\infty}_yL^2_z}
        \lesssim
        \nu^{-\f13}\|\W^{\pm}_{in}\|_{L^2}+\nu^{-\f13}\|N^{\pm}\|_{L^1_tL^2}.
    \end{align}
    Following the similar derivation as above, we have
    \begin{align*}
       |l| \| e^{t \mathcal{L}^{\pm,l}} (1,\cR_2^l,\cR_3^l)\W^{\pm,l} \|_{L^1_t L^{\infty}_y}
        \lesssim &~
        \nu^{-\f56}\|\W^{\pm,l}_{in}\|_{L^2_{y}}+\nu^{-\f56}\|N^{\pm,l}\|_{L^1_tL^2_{y}},
    \end{align*}
 and thus
    \begin{align}\label{est-R3-RT2}
        \| e^{t \mathcal{L}^{\pm}} \pa_z(1,\cR_2,\cR_3)\W^{\pm} \|_{L^1_t L^{\infty}_yL^2_z}     
        \lesssim &~
        \nu^{-\f56}\|\W^{\pm}_{in}\|_{L^2}+\nu^{-\f56}\|N^{\pm}\|_{L^1_tL^2}.
    \end{align}
    Recalling \eqref{rep-u01-w01} and \eqref{rep-u023}, combined with \eqref{est-R3-RT1} and \eqref{est-R3-RT2}, it holds that
         \begin{align*}
        \nu^{\f13}\|u_0^2\|_{L^2_t L^{\infty}_yL^2_z}
        +\nu^{\f56}\|\pa_z\uu_0\|_{L^1_t L^{\infty}_yL^2_z}
        \lesssim    \|\uu_{0in}\|_{L^2}+\|\PP_0(\uu\cdot\nabla_{\x_1}  \uu)\|_{L^1_tL^2}.
    \end{align*}
    By a similar derivation, we can find that
    \begin{align*}
        \nu^{\f13}\|\pa_zu_0^2\|_{L^2_t L^{\infty}_yL^2_z}
        +\nu^{\f56}\|\pa_z^2\uu_0\|_{L^1_t L^{\infty}_yL^2_z}
        \lesssim 
        \|\pa_z\uu_{0in}\|_{L^2}+\|\pa_z\PP_0(\uu\cdot\nabla_{\x_1}  \uu)\|_{L^1_tL^2}.
    \end{align*}
    The proof is therefore complete.
\end{proof}

With Lemma \ref{lemm-str-T} at hand, it therefore suffices to prove Proposition \ref{prop-1}.
\begin{proof}[\underline{\textbf{Proof of Proposition \ref{prop-1} on $\T\times\R\times \T$}}]
    For the non-zero modes, the estimates analogous to the domain $\T\times\R^2$ in Section \ref{sec4:energy est zero} yield the following bound:
    \begin{align}\label{sec3:priori result-1-T}
	   \sup_{\tau\in[0,t]}\E_{\not=}(\tau)
	   +\int_0^t\D_{\not=}(\tau)\,d\tau
	   \lesssim \var_0^2+\ep_1^3\nu^{3\alpha-\f23-\delta_1}
       \lesssim \var_0^2+\ep_1^3\nu^{2\alpha},
    \end{align}
    where we have used the fact that $\alpha\geq \f56$.
    
    \noindent While for the zero modes, we divide the proof into two steps.

    \medskip

    \noindent\textbf{Step 1:} we aim to show that
    \begin{align}\label{est-energy-1-T}
        \sum_{s=0}^5\big(\nu^{\f13}\|\nabla^s(1+\pa_z)u_0^2\|_{L^2_tL^\infty_yL^2_z}
	    +\nu^{\f56}\|\nabla^s(1+\pa_z)\pa_z\uu_0\|_{L^1_t L^\infty_yL^2_z}\big)
        \lesssim \var_0+ \ep_1^2 \nu^{\al}.
    \end{align}
    The direct application of \eqref{est-str-u02-T} yields that
    \begin{align}\label{est-energy-T-1}
	&\sum_{s=0}^5\big(\nu^{\f13}\|\nabla^s(1+\pa_z)u_0^2\|_{L^2_tL^\infty_yL^2_z}
	+\nu^{\f56}\|\nabla^s(1+\pa_z)\pa_z\uu_0\|_{L^1_t L^\infty_yL^2_z}\big)
    \nonumber
    \\
	\lesssim &\|(1+\pa_z)\uu_{0in}\|_{H^5}
    +\|(1+\pa_z)\PP_0(\uu\cdot\nabla_{\x_1}\uu)\|_{L^1_t{H^5}}.
\end{align}
   We now deal with the nonlinear term on the right-hand side of \eqref{est-energy-T-1}:
   \begin{align}\label{est-F-T}
       \|(1+\pa_z)\PP_0(\uu\cdot\nabla_{\x_1}  \uu)\|_{L^1_t{H^5}}
       \lesssim 
       \|(1+\pa_z)(\uu_0\cdot\nabla \uu_0)\|_{L^1_t{H^5}}
       +\|(1+\pa_z)\PP_0(\uu_{\not=}\cdot\nabla_{\x_1} \uu_{\not=})\|_{L^1_t{H^5}}.
    \end{align}
    Applying the Sobolev embedding inequality
    $H^1(\T)\hookrightarrow L^{\infty}(\T)$
    and the \textit{a priori} assumption \eqref{priori assumption}, we find that
   \begin{align}\label{est-F1-T}  
       &\|(1+\pa_z)(\uu_0\cdot\nabla \uu_0)\|_{L^1_t{H^5}}
       \nonumber
       \\
       \lesssim & ~
       \|(1+\pa_z)u_0^2\|_{L^2_tW^{4,\infty}} \|(1+\pa_z)\pa_y\uu_0\|_{L^2_tH^5}
       +\|(1+\pa_z)\nabla^5 u_0^2\|_{L^2_tL^\infty_{y}L^2_z} \|(1+\pa_z)\pa_y\uu_0\|_{L^2_tL^2_yL^\infty_{z}}
       \nonumber
       \\
       &+\|(1+\pa_z)u_0^3\|_{L^\infty_tL^2_yL^\infty_{z}} \|(1+\pa_z)\na^5\pa_z\uu_0\|_{L^1_tL^\infty_yL^2_{z}}
       +\|(1+\pa_z) u_0^3\|_{L^\infty_tH^5} \|(1+\pa_z)\pa_z\uu_0\|_{L^1_tW^{4,\infty}}
       \nonumber
       \\  
       \lesssim&~ \ep_1 \nu^{\al-\f56}\sum_{s=0}^5\big(\nu^{\f13}\|\nabla^s(1+\pa_z)u_0^2\|_{L^2_tL^\infty_yL^2_z}
+\nu^{\f56}\|\nabla^s(1+\pa_z)\pa_z\uu_0\|_{L^1_t L^\infty_yL^2_z}\big).
   \end{align}
   We employ \eqref{est-inviscid} and Lemma \ref{3lem-key}, to obtain that
   \begin{align}\label{est-F2-T}
       &\|(1+\pa_z)\PP_0(\uu_{\not=}\cdot\nabla_{\x_1} \uu_{\not=})\|_{L^1_t{H^5_{y,z}}}
       \nonumber
       \\
       \lesssim &~
       \|(1+\pa_z)\PP_0(1,\nabla_{\x_1})(\uu_{\not=}\cdot\nabla_{\x_1} \uu_{\not=})\|_{L^1_t{H^4_{y,z}}}
       \nonumber
       \\
       \lesssim &
       \sum_{j=1,3}
       \|(1+\pa_z)(1,\widetilde{\nabla})(U_{\not=}^j\pa_j \U_{\not=})\|_{L^1_t{H^4}}
       +\|(1+\pa_z)((1,\widetilde{\nabla})U_{\not=}^2\tilde{\pa}_y \U_{\not=})\|_{L^2_t{H^4}}
       +\|(1+\pa_z)(U_{\not=}^2\tilde{\pa}_y \widetilde{\nabla}\U_{\not=})\|_{L^2_t{H^4}}
       \nonumber
       \\
       \lesssim&
       \sum_{j=1,3}
       \big(\|(1+\pa_z)(1,\widetilde{\nabla})U_{\not=}^j\|_{L^2_tH^4}\|(1+\pa_z)\pa_j \U_{\not=}\|_{L^2_t{H^4}}
       +\|(1+\pa_z)U_{\not=}^j\|_{L^2_tH^4}\|(1+\pa_z)\widetilde{\nabla}\pa_j \U_{\not=}\|_{L^2_t{H^4}}\big)
       \nonumber
       \\
       &+\|(1+\pa_z)(1,\widetilde{\nabla})U_{\not=}^2\|_{L^2_t{H^4}}
       \|(1+\pa_z)\tilde{\pa}_y \U_{\not=}\|_{L^2_t{H^4}}
       +\|(1+\pa_z)U_{\not=}^2\|_{L^2_t{H^4}} \|(1+\pa_z)\tilde{\pa}_y \widetilde{\nabla}\U_{\not=}\|_{L^2_t{H^4}}
       \nonumber
       \\
       \lesssim& ~\nu^{-\f56}\int_0^t\D_{\not=}(\tau)\,d\tau
       \lesssim
       \ep_1^2\nu^{2\al-\f56}.
   \end{align}
   Since $\al\geq \f56$, we choose $\ep_1$ suitably small and deduce from \eqref{est-energy-T-1}, \eqref{est-F1-T} and \eqref{est-F2-T} that \eqref{est-energy-1-T} holds.

    \medskip

    \noindent\textbf{Step 2:} we aim to show that
    \begin{align*}
        \|(1+\pa_z)\uu_0\|_{L^\infty_t H^5}^2
	    +\nu\|(1+\pa_z)\nabla\uu_0\|_{L^2_t H^5}^2
        \lesssim \var_0^2+ \ep_1^3 \nu^{2\al}.
    \end{align*}
    Applying standard energy estimates to system \eqref{1-eq2} as \eqref{energy-0-1}, and using the estimates \eqref{est-energy-1-T}, \eqref{est-F1-T} and \eqref{est-F2-T}, we can obtain that
    \begin{align*}
        &\|(1+\pa_z)\uu_0\|_{L^\infty_t H^5}^2
	    +\nu\|(1+\pa_z)\nabla\uu_0\|_{L^2_t H^5}^2
        \nonumber
        \\
        \lesssim& \|(1+\pa_z)\uu_{0in}\|_{H^5}^2+\big(\|(1+\pa_z)(\uu_0\cdot\nabla \uu_0)\|_{L^1_tH^5}
      +  \|(1+\pa_z)\PP_0(\uu_{\not=}\cdot\nabla_{\x_1} \uu_{\not=})\|_{L^1_t{H^5}}\big)
        \|(1+\pa_z)\uu_0\|_{L^\infty H^5}
        \nonumber
        \\
        \lesssim &~ \var_0^2+\ep_1^2\nu^{2\al-\f56}\big(
        \var_0+\ep_1^2\nu^\al\big)+\ep_1^3\nu^{3\al-\f56}
        \nonumber
        \\
        \lesssim&~\var_0^2+\ep_1^3\nu^{2\al},
    \end{align*}
    where we have used the fact that $\al\geq \f56$.
    We thus finish the proof of Proposition \ref{prop-1} on $\T\times\R\times \T$.
    \end{proof}



\begin{appendices}


\section{Dispersive and Strichartz estimates}
\label{App-DS-est}

This appendix provides the proof of the dispersive and Strichartz estimates the zero modes. Here, $k^{+}$ is a number greater than but arbitrarily close to $k$.


\subsection{Dispersive estimates on \texorpdfstring{$\R^2$}{R2}}
\begin{lemm}\label{App-dis-1}
	Let $\mathcal{R}_3=\pa_z|\nabla_{y,z}|^{-1}$.
	There exists $C>0$ such that for any $\gamma>0, f\in \mathcal{S}(\R^2),$
	\beq\label{dispersive}
	\left\| e^{\gamma t\mathcal{R}_3(D)}f \right\|_{L^\infty(\mathbb{R}^2)} \leq C (1+|\gamma t|)^{-\frac{1}{2}} \|f\|_{W^{2^{+},1}(\mathbb{R}^2)}\,.
	\eeq
\end{lemm}

\begin{proof}
	It suffices to prove the case of $\gamma=1.$
	In fact, to show \eqref{dispersive},  we will prove that
	\beq\label{dispersive-1}
	\left\| e^{t\mathcal{R}_3}f \right\|_{L^\infty(\mathbb{R}^2)} \leq |t|^{-\frac{1}{2}} \|f\|_{W^{2^{+},1}(\mathbb{R}^2)}, \quad \forall~ |t|\geq 1,
	\eeq
	which, combined with the Sobolev embeddings, gives \eqref{dispersive}.
	
	Let $\chi_j=\chi_j(\xi, \eta)=\chi_0(2^{-j}(\xi, \eta))\, (j\in \mathbb{Z})$ be radial, dyadic cut-off functions that are supported on $\{2^{j-1}\leq |(\xi,\eta)|\leq 2^{j+1}\}$ and satisfies 
	\begin{align*}
		\sum_{j\in \mathbb{Z}}\chi_j(\xi,\eta)=1, \quad \text{for}~(\xi,\eta)\neq 0\,.
	\end{align*}
	
	Let us first claim that it suffices to show that 
	\begin{align}\label{disp-show}
		\left\| e^{t\mathcal{R}_3}\chi_0(D)f \right\|_{L^\infty(\mathbb{R}^2)} \lesssim |t|^{-\frac{1}{2}} \|\chi_0(D)f\|_{L^1(\mathbb{R}^2)}.
	\end{align}
	Indeed, it holds that 
	\begin{align*}
		\cF\big[e^{t\mathcal{R}_3}\chi_j(D)f\big](\xi,\eta)&=e^{t \mathcal{R}_3(\xi,\eta)}\chi_0\big(2^{-j} (\xi,\eta)\big) \cF[f](\xi,\eta) \\
		&= e^{t \mathcal{R}_3(\tilde{\xi},\tilde{\eta})}\chi_0(\tilde{\xi},\tilde{\eta}) \cF[f](2^{j}(\tilde{\xi},\tilde{\eta}))
        \\
        &=e^{t \mathcal{R}_3(\tilde{\xi},\tilde{\eta})}\chi_0(\tilde{\xi},\tilde{\eta}) \cF[g](\tilde{\xi},\tilde{\eta}),
	\end{align*}
	where we denoted $(\tilde{\xi}, \tilde{\eta})=2^{-j}(\xi,\eta)$ and $\cF[g](\tilde{\xi},\tilde{\eta})\tri \cF[f](2^{j}(\tilde{\xi},\tilde{\eta}))$.
	Consequently, it follows from \eqref{disp-show} that
	\begin{align}\label{dispes-pj}
		\left\| e^{t\mathcal{R}_3}\chi_j(D)f \right\|_{L^\infty(\mathbb{R}^2)} 
        &=2^{2j}\left\| e^{t\mathcal{R}_3}\chi_0(D)g \right\|_{L^\infty(\mathbb{R}^2)} 
        \lesssim 2^{2j}|t|^{-\frac{1}{2}} \left\| \chi_0(D)g \right\|_{L^1(\mathbb{R}^2)} 
        \nonumber
        \\
        &=|t|^{-\frac{1}{2}} \left\| \chi_j(D)f(2^{-j}(x,y))\right\|_{L^1(\mathbb{R}^2)} 
        =
        2^{2j}|t|^{-\frac{1}{2}} \|\chi_j(D)f\|_{L^1(\mathbb{R}^2)}, \quad \forall~ j\in \mathbb{Z}\, ,
	\end{align}
    where we have used the fact that
    \begin{align*}
        (\chi_0(D)g) (x,y)=2^{-2j}(\chi_j(D) f)(2^{-j}(x,y)).
    \end{align*}
	Summarizing $j$ over $\mathbb{Z},$ we obtain 
	\eqref{dispersive-1}.
	
	We now focus on proving \eqref{disp-show}. By Inverse Fourier transform, it suffices to show that
	\[
	K(t, y,z) = \int_{\mathbb{R}^2} e^{i  (y\,\xi + z\,\eta) + t \mathcal{R}_3(\xi, \eta)} \tilde{\chi}_0(\xi, \eta) \, d\xi \, d\eta
	\]
	satisfies (here $ \tilde{\chi}_0={\chi}_0+{\chi}_1+{\chi}_{-1}$)
	\[
	\| K(t,y,z) \|_{L^\infty_{y,z}} \lesssim t^{-\frac{1}{2}}, \quad \forall~ t\geq 1.
	\]
	It is convenient to write $K$ in the polar coordinate
	\[
	\xi = r \cos \theta, \quad \eta = r \sin \theta \Rightarrow \mathcal{R}_3(\xi, \eta) = i\,\sin \theta\,\,.
	\]
	Let us also set \( y = \rho \cos \varphi, \, z = \rho \sin \varphi \), then $K$ can be written
	\begin{align*}
		K(t,y,z)&=K(t,\rho,\varphi) 
		= \int_0^\infty \int_{-\pi}^{\pi} e^{i t \sin \theta} \, e^{i r  
			\rho \cos (\theta-\varphi)} A(r) \, r \, dr \, d\theta\\
		&=\int_0^\infty \int_{-\pi}^{0} e^{i t \sin \theta} \, e^{i r  \rho \cos (\theta-\varphi)} A(r) \, r \, dr \, d\theta+\int_0^\infty \int_{0}^{\pi} e^{i t \sin \theta} \, e^{i r  \rho \cos (\theta-\varphi)} A(r) \, r \, dr \, d\theta\\
		&\tri  K_{1}(t, \rho, \varphi)+K_{2}(t,\rho, \varphi),
	\end{align*}
	where 
	\beqs
	\tilde{\chi}_0(\xi,\eta)= A(r) \in  C_c^{\infty}(\mathbb{R}).
	\eeqs
	Define
	\[
	B(z) \tri \int_0^\infty e^{i r  z} A(r) r \, dr\,=\mathcal{F}^{-1}(A(r) r)(z).
	\]
	We will prove that, 
    for any $(\rho,\varphi)\in (0,+\infty)\times [0,2\pi]$
	\beq\label{es-K2}
	|K_2(t, \rho,\varphi)|=\bigg|\int_{0}^{\pi} e^{it \sin\theta}B(\rho \cos (\theta-\varphi))\, d\theta\bigg| \lesssim |t|^{-1/2},
	\eeq
	the proof of $K_1$ is similar. Note that $B=\mathcal{F}^{-1}(A(r) r)$ is a Schwartz function, in particular $B\in W^{1,1}(\R)$,
\begin{align*}
		\int_{-\pi}^{\pi} |\partial_{\theta}B(\rho \cos (\theta-\varphi))|\, d\theta=\int_{-\pi}^{\pi} |\partial_{\theta}B(\rho \cos \theta)|=2\int_{0}^{\pi} |\partial_{\theta}B(\rho \cos \theta)|\, d\theta=2\int_{-\rho}^{\rho}|B'(z)|dz\leq C, 
	\end{align*}
	We can apply the 1-d stationary phase arguments to show \eqref{es-K2}.
	Indeed, $(\sin\theta)'=\cos\theta$ vanishes at 
	$\theta_0=\pi/2,$ but $(\sin\theta)''=-\sin\theta$ does not vanish near $\theta=\theta_0.$ 
	We thus split $K_2=K_{21}+K_{22},$ where 
	\begin{align*}
		K_{21}(t,\rho,\varphi)=\int 
		_{I_{\delta}}\, e^{it \sin\theta} B(\rho \cos (\theta-\varphi)) \, d\theta,
        ~~\text{and}
		~~  
        K_{22}(t,\rho,\varphi)=
		\int_{[0,\pi]\backslash I_{\delta}}\, e^{it \sin\theta} B(\rho \cos (\theta-\varphi)) \,d\theta, 
	\end{align*}
	where $I_{\delta}=[\frac{\pi}{2}-\delta,\, \frac{\pi}{2}+\delta]$ with
	$0<\delta<\pi/2$ being small and to be chosen. 
	On one hand, thanks to $B\in L^{\infty}$, 
	$K_{21}$ can be controlled simply by 
	\begin{align}\label{es-K21}
		|K_{21}(t,\cdot,\cdot)| \lesssim \delta, 
	\end{align}
	On the other hand, it holds that on the interval $[0,\pi]\backslash I_{\delta},$ 
	\beqs 
	|\cos\theta|\geq \sin\delta\geq \frac{\delta}{2}.
	\eeqs
	Consequently, using the identity $e^{it\sin\theta}=(\frac{\partial_{\theta}}{i t \cos\theta})e^{it\sin\theta}$ and integrating by parts in $\theta,$ we obtain that 
	\beq\label{K22}
	\sup_{\varphi\in  [0,2\pi]}  |K_{21}(t,\rho,\varphi)| \lesssim  \frac{1}{\delta t} \sup_{\varphi\in  [0,2\pi]}  
	\big(\|B(\rho \cos (\theta-\varphi))\|_{L_{\theta}^{\infty}}+\|\pa_{\theta}B(\rho \cos (\theta-\varphi))\|_{L_{\theta}^{1}}\big)
	\lesssim  \frac{1}{\delta t}.
	\eeq
	The estimate \eqref{es-K2} follows from 
	\eqref{es-K21} and \eqref{K22} by choosing $\delta=Ct^{-\frac12}$. 
\end{proof}

\begin{rema}
    By density arguments, it is easy to see that \eqref{dispersive} holds as long as 
    $f$ is such that the right-hand side is bounded.
\end{rema}
\begin{lemm}\label{App-dis-2}
	Let $\mathcal{R}_3=\pa_z|\nabla_{y,z}|^{-1}$.
  It holds that for  $f\in \mathcal{S}(\R^2)$ and
    for $a=0,1,$
	\begin{align}\label{dispersive-improved}
		\left\| e^{\gamma t\mathcal{R}_3}\pa_y |\nabla_{y,z}|^{-a} f \right\|_{L^\infty(\mathbb{R}^2)} \leq C (1+|\gamma t|)^{-1} \|f\|_{W^{(3-a)^{+},1}(\mathbb{R}^2)}\,.
	\end{align}
\end{lemm}
\begin{proof}
	The proof follows from similar and easier arguments as in the previous lemma. 
	It suffices to show 
    \begin{align}\label{dispes-pj-pay}
		\left\| e^{t\mathcal{R}_3}\pa_y |\nabla_{y,z}|^{-a} \chi_j(D)f \right\|_{L^\infty(\mathbb{R}^2)} \leq 2^{(3-a)j}|t|^{-1} \|\chi_j(D)f\|_{L^1(\mathbb{R}^2)}, \quad \forall~ j\in \mathbb{Z}\, ,
	\end{align}
which relies on the boundedness of 
	\begin{align*}
		M(t,\rho,\varphi)=  \int_{-\pi}^{\pi} e^{i t \sin \theta}\cos (\theta) \tilde{B}({\rho}\cos (\theta-\varphi)) \, d\theta 
	\end{align*}
	where 
	\beqs
	\tilde{B}(z)\tri \int_0^\infty e^{i r 
		z} A(r) \, r^{2-a} \, dr\,=\mathcal{F}^{-1}(A(r) r^{2-a})(z)\in W^{1,1}(\R).
	\eeqs
	
	
	Thanks to the identity $e^{i t \sin \theta}\cos (\theta) =\frac{1}{it}\pa_{\theta} (e^{i t \sin \theta}),$ we integrate by parts in $\theta$ to get
	\begin{align*}
		|M(t,\rho,\varphi)|&=\left|t^{-1}\int_{-\pi}^{\pi} e^{i t \sin \theta}\pa_{\theta}\tilde{B}({\rho}\cos (\theta-\varphi)) \, d\theta\right|\leq t^{-1}\int_{-\pi}^{\pi} |\pa_{\theta}\tilde{B}({\rho}\cos (\theta-\varphi))| \, d\theta\\&= t^{-1}\int_{-\pi}^{\pi} |\pa_{\theta}\tilde{B}({\rho}\cos \theta)| \, d\theta=2t^{-1}\int_{0}^{\pi} |\pa_{\theta}\tilde{B}({\rho}\cos \theta)| \, d\theta=2t^{-1}\int_{-{\rho}}^{{\rho}} |\tilde{B}'(z)| \, dz\lesssim t^{-1}.
	\end{align*}
	
\end{proof}

\subsection{Strichartz estimates on \texorpdfstring{$\R^2$}{R2}}
Based on Lemmas \ref{App-dis-1}-\ref{App-dis-2}, our goal in this appendix is to prove the Strichartz estimates.

\begin{lemm}\label{lem-str-1}
	Let $\mathcal{R}_3=\pa_z|\nabla_{y,z}|^{-1}$.
	It holds that for any $f\in \mathcal{S}(\R^2),$
    any $\gamma>0, r\in[2,+\infty]$
	\begin{align}\label{str-1}
		\left\| e^{\gamma t\mathcal{R}_3}f \right\|_{L^p_tL^r(\mathbb{R}^2)} \leq C \|f\|_{H^{(1-\f2r)^+}(\mathbb{R}^2)}\,,~~\text{with}~~\f1p+\f{1}{2r}= \f14.
	\end{align}
    Moreover, for $b=0,1$
	\begin{align}\label{str-2}
		\left\| |\nabla_{y,z}|^b e^{\gamma t\mathcal{R}_3+\nu t\Delta}f \right\|_{L^{\f{2}{1+b}}_t L^\infty(\mathbb{R}^2)} \leq C\nu^{-(\f14+\f{b}{2})}\|f\|_{H^{(\f12)^+}(\mathbb{R}^2)}.
	\end{align}
\end{lemm}
\begin{proof}
 Let $P_j, \tilde{P}_j$ be the Fourier multipliers with symbol supported on $\{2^{j-1}\leq |(\xi,\eta)|\leq 2^{j+1}\}$ and $\{2^{j-2}\leq |(\xi,\eta)|\leq 2^{j+2}\}.$
 To prove \eqref{str-1}, it suffices to show that
 \begin{align}\label{str-mode}
     \left\| e^{\gamma t\mathcal{R}_3}P_j f \right\|_{L^p_t L^r(\mathbb{R}^2)} \lesssim  2^{j(1-\f2r)}\|\tilde{P}_j f\|_{L^2} \,,~~\text{if}~~\f1p+\f{1}{2r}= \f14.
 \end{align}
 Such an estimate follows from \eqref{dispes-pj} and the standard  
$TT^\ast$ argument, we omit its proof for simplicity. 
We now turn to prove \eqref{str-2}. Since the  Fourier transform of $ e^{\gamma t\mathcal{R}_3}P_j f$ is supported of the annulus 
 $\{2^{j-1}\leq |(\xi,\eta)|\leq 2^{j+1}\},$ we can apply Lemma 2.4 of \cite{Bahouri-Chemin-Danchin} to obtain that 
  \begin{align*}
     \left\| |\nabla_{y,z}|^b e^{\nu t\Delta}e^{\gamma t\mathcal{R}_3}P_j f \right\|_{L^{\infty}(\mathbb{R}^2)} \lesssim e^{-c 2^{2j}\nu t} \, 2^{j(1+b)}\|\tilde{P}_j f\|_{L^2},
 \end{align*}
where the constant $c>0$ is independent of $\nu.$ 
Taking the $L_t^{\f{2}{1+b}}$ norm, application of 
Hölder inequality and 
\eqref{str-mode} with $(p,r)=(4,+\infty)$ gives that 
\begin{align}\label{str-2-modes}
     \left\| |\nabla_{y,z}|^b e^{\nu t\Delta} e^{\gamma t\mathcal{R}_3}P_j f \right\|_{L^{\f{2}{1+b}}_t L^{\infty}(\mathbb{R}^2)} \lesssim \nu^{-(\f14+\f{b}{2})} \,2^{\f{j}{2}}\|\tilde{P}_j f\|_{L^2}. 
    \end{align}
The estimate \eqref{str-2} then follows by summing over $j\in\mathbb{Z}.$

\end{proof}

	

\begin{lemm}\label{lem-str-2}
	Let $\mathcal{R}_3=\pa_z|\nabla|^{-1}$.
	It holds that for any  $f\in \mathcal{S}(\R^2),$ any $\gamma>0,$ 
	\begin{align}\label{str-3}
		\left\| e^{\gamma t\mathcal{R}_3}\pa_yf \right\|_{L^p_t L^r(\mathbb{R}^2)} \leq C \|f\|_{H^{(2-\f2r)^+}(\mathbb{R}^2)}\,,~~\text{with}~~\f1p+\f{1}{r}= \f12 ~~\text{and}~~ (p,r)\not=(2,\infty).
	\end{align}
		\noindent For $(p,r)=(2,\infty)$, and $a=0,1,$
    it holds that for any small constant $\delta_1>0$
	\begin{align}\label{str-4}	\left\| e^{\gamma t\mathcal{R}_3+\nu t\Delta}\pa_y|\nabla_{y,z}|^{-a}f \right\|_{L^2_t L^\infty(\mathbb{R}^2)} \leq C\nu^{-\delta_1} \|f\|_{H^{(2-a)}(\mathbb{R}^2)}.\end{align}
    Moreover, 
    \begin{align}\label{str-5}	
    \left\| e^{\gamma t\mathcal{R}_3+\nu t\Delta}\pa_yf \right\|_{L^1_t L^\infty(\mathbb{R}^2)} \leq C\nu^{-(\f12+\delta_1)} \|f\|_{H^{1}(\mathbb{R}^2)}.
    \end{align}
\end{lemm}
\begin{proof}
The proof of \eqref{str-3} follows from 
the dispersive estimate \eqref{dispes-pj-pay}
and the standard $TT^{*}$ arguments. The estimate 
\eqref{str-5} follows from \eqref{str-4} 
and arguments similar to those used in the proof of \eqref{str-2}; hence, we omit the details.
We now focus on the critical case $(p,r)=(2,\infty)$ and give the proof of \eqref{str-4} 
for $a=0,$ the case $a=1$ follows along the same lines. 
Define the operator $\tilde{T}_k$ from $L^2(\R^2)$ to $L_t^2(\R_{+}; L^{\infty}(\R^2)):$
$$\tilde{T}_k:=e^{\gamma t \mathcal{R}_3} e^{\nu t\Delta}\partial_y P_k,$$
where $P_k$ is the Fourier multiplier with symbol supported on $\{2^{j-1}\leq |(\xi,\eta)|\leq 2^{j+1}\}.$
The corresponding adjoint operator $\tilde{T}_k^{*}:  L_t^{2}(\R_{+}; L^{1}(\R^2))\rightarrow L^2 (\R^2)\ $ is defined as
\begin{align*} 
\tilde{T}_k^{*}= -\int_{\mathbb{R}_{+}} e^{-\gamma s \mathcal{R}_3 }e^{\nu s\Delta}\partial_y P_k \, d s
\end{align*}
and thus $\tilde{T}_k\tilde{T}_k^{*}: \,  L_t^{2}(\R_{+}; L^{1}(\R^2))\rightarrow   L_t^{2}(\R_{+}; L^{\infty}(\R^2))$ reads
\begin{align*} 
\tilde{T}_k\tilde{T}_k^{*} = - \mathbb{I}_{\{t\geq 0\}} \int_{\mathbb{R}} \mathbb{I}_{\{s\geq 0\}} \,e^{\gamma(t- s) \mathcal{R}_3 }e^{\nu (t+s)\Delta}\partial_y^2 P_k^2 \, d s
\end{align*}
We apply successively the estimates 
\begin{align*}
&  \| \pa_y e^{\nu(t+s)\Delta}P_k h\|_{L^{\infty}(\R^2)} \lesssim e^{-c\nu 2^{2k}(t+s)}2^k\|h\|_{L^{\infty}(\R^2)}\lesssim |\nu 2^{2k}(t+s)|^{-2\delta}2^k\|h\|_{L^{\infty}(\R^2)}, \, \quad \forall~ \delta>0,~ t+s>0,\\
& \left\|e^{\gamma( t-s) \mathcal{R}_3} \partial_y {P_k} h\right\|_{L^{\infty}(\R^2)} \lesssim (1+|t-s|)^{-1}  2^{3 k}\|h\|_{L^{1}(\R^2)},
\end{align*}
to find that for any $t\geq 0,$
\begin{align*} 
\|\tilde{T}_k\tilde{T}_k^{*} g(t)\|_{L^{\infty}(\R^2)}\lesssim 2^{4k}  \int_{\mathbb{R}} \mathbb{I}_{\{s\geq 0\}} |\nu 2^{2k}(t+s)|^{-2\delta}(1+|t-s|)^{-1} \|g(s)\|_{L^1(\R^2)}\, d s .
\end{align*}
Since for any $t,s\geq 0$, $t+s\geq |t-s|,$ we use the Young's inequality (after extending $h(s)$ to $s<0$ by zero extension) to find that
 \begin{align*} 
\|\tilde{T}_k\tilde{T}_k^{*} g(t)\|_{L_t^2(\R_{+};L^{\infty}(\R^2))}\lesssim \nu^{-2\delta} 2^{(4-4\delta)k} \|g\|_{L_t^2(\R_{+};L^{1}(\R^2))},
\end{align*}
from which \eqref{str-4} easily follows.
\end{proof}

\subsection{Dispersive and Strichartz estimates on \texorpdfstring{$\R\times \T$}{R x T}}
In this short appendix, we state some useful dispersive and Strichartz estimates when $\R^2$ is replaced by $\R\times \T.$
For any $h(y,z)$, we denote
\begin{align*}
   h=\sum_{l\in\Z} h^l(y) e^{ilz},\quad \text{for}~~
    h^l(y)\tri\f{1}{2\pi}\int_{T}h(y,z)e^{-ilz}\,dz.
\end{align*}
We also define the corresponding operator $\mathcal{R}_3^l\tri -il(l^2-\pa_y^2)^{-\f12}$.
The semigroup generated by $\mathcal{R}^l_3$ exhibits the following decay estimate:
\begin{lemm}\label{lem:disp-T}
    For any $l\not=0$, it holds that 
    \begin{align}\label{disp-RT-0}
        &\|e^{\gamma t \mathcal{R}_3^l} h^l\|_{L^\infty_y(\R)} 
        \leq C 
        (\gamma t)^{-\f13}
        |l|\|h^l\|_{L^1_y(\R)}+|l|^{-\f12}(\gamma t)^{-\f12}\|h^l\|_{W^{(\f32)^+,1}_y(\R)},\\ \label{disp-RT}
        &\|e^{\gamma t \mathcal{R}_3^l+\nu t(\partial_y^2-l^2)}h^l\|_{L_t^2(L_y^{\infty}(\R))}+\nu^{\f12}|l|\|e^{\gamma t \mathcal{R}_3^l+\nu t(\partial_y^2-l^2)}h^l\|_{L_t^1(L_y^{\infty}(\R))}\leq C \nu^{-\f13} |l|^{-\f16}\|h^l\|_{L^{2}(\R)}.
    \end{align}
\end{lemm}
\begin{proof}
    We will use the notations and some results in  the proof of [Prop. 3.1] in \cite{Coti-Zotto-Widmayer}. The first inequality was proved in \cite{Coti-Zotto-Widmayer}. Let $P_jh^l(y)$ be the projections
 associated to a standard Littlewood-Paley decomposition, $h^l=\sum_{j\in\Z}P_jh^l$, $\mathcal{F}(P_jh^l)(\eta)=\varphi_j(\eta)\hat{h^l}(\eta)$, $\varphi_j(\eta)=\varphi(2^{-j}\eta)$. 
 Then by Young's convolution inequality, we have
\begin{align}\label{h1}
        \|e^{\gamma t \mathcal{R}_3^l}P_jh^l\|_{L^{\infty}}=\|\mathcal{F}^{-1}(e^{-it\gamma l/|\eta,l|}\varphi(2^{-j}\eta))*h^l\|_{L^{\infty}}\leq\|\mathcal{F}^{-1}(e^{-it\gamma l/|\eta,l|}\varphi(2^{-j}\eta))\|_{L^{\infty}}\|h^l\|_{L^{1}}.
    \end{align}  
Without loss of generality, we assume $l>0$, the case $l<0$ can be proved by taking conjugation. 
By the change of variables $ \eta=l\xi$, it follows that
\begin{align*}
        &\mathcal{F}^{-1}(e^{-it\gamma l/|\eta,l|}\varphi(2^{-j}\eta))(y)=\int_{\R}e^{iy\eta-it\gamma l/|\eta,l|}\varphi(2^{-j}\eta)d\eta
        =|l|\int_{\R}e^{iyl\xi-it\gamma /|\xi,l|}\varphi(2^{-j}l\xi)d\xi\eqqcolon|l|I(t\gamma,l,j),
    \end{align*}
    It was proved in \cite{Coti-Zotto-Widmayer} that (here $j_0:=\log_2l-5/2$)
   \begin{align*}
        &|I(t\gamma,l,j)|\lesssim\left\{\begin{array}{ll}
                                          \min\{2^j|l|^{-1},(\gamma t)^{-\f12}\}&j<j_0, \\
                                          (\gamma t)^{-\f13}&j_0\leq j\leq j_0+4, \\
                                          (\gamma t)^{-\f12}2^{\f32j}|l|^{-\f32}&j>j_0+4.
                                        \end{array}
        \right.
    \end{align*} 
    Inserting this into \eqref{h1}, we find that
    \begin{align*}
        &\|e^{\gamma t \mathcal{R}_3^l}P_jh^l\|_{L^{\infty}}\lesssim |l|\min\{2^j|l|^{-1},(\gamma t)^{-\f12}\}\|h^l\|_{L^{1}}\lesssim |l|(2^j|l|^{-1})^{\f13}(\gamma t)^{-\f13}\|h^l\|_{L^{1}},\quad j<j_0,\\
        &\|e^{\gamma t \mathcal{R}_3^l}P_jh^l\|_{L^{\infty}} \lesssim |l|(\gamma t)^{-\f13} \lesssim |l|(2^j|l|^{-1})^{\f13}(\gamma t)^{-\f13}\|h^l\|_{L^{1}},\quad j_0\leq j\leq j_0+4,\\
        &\|e^{\gamma t \mathcal{R}_3^l}P_jh^l\|_{L^{\infty}} \lesssim |l|(\gamma t)^{-\f12}2^{\f32j}|l|^{-\f32}\|h^l\|_{L^{1}},\quad j >j_0+4.
    \end{align*}
   Then by the standard $TT^\ast$ method, one can obtain that
   \begin{align*}
        &\|e^{\gamma t \mathcal{R}_3^l}P_jh^l\|_{L_t^6L_y^{\infty}} 
        \lesssim |l|^{\f12}(2^j|l|^{-1})^{\f16}\|h^l\|_{L^{2}(\R)},\quad j\leq j_0+4,\\
        &\|e^{\gamma t \mathcal{R}_3^l}P_jh^l\|_{L_t^4L_y^{\infty}}
        \lesssim
        2^{\f34j}|l|^{-\f14}\|h^l\|_{L^{2}(\R)},\quad j >j_0+4.
    \end{align*}
    These estimates lead to that\begin{align*}
        &\|e^{\gamma t \mathcal{R}_3^l+\nu t(\partial_y^2-l^2)}P_jh^l\|_{L_t^2L_y^{\infty}}\leq \|e^{-\nu tl^2}\|_{L_t^3}\|e^{\gamma t \mathcal{R}_3^l}P_jh^l\|_{L_t^6L_y^{\infty}}\lesssim |\nu l^2|^{-\f13}|l|^{\f12}(2^j|l|^{-1})^{\f16}\|h^l\|_{L^{2}(\R)},\quad j\leq j_0+4,\\
        &\|e^{\gamma t \mathcal{R}_3^l+\nu t(\partial_y^2-l^2)}P_jh^l\|_{L_t^2L_y^{\infty}}\lesssim\|e^{-c\nu t2^{2j}}\|_{L_t^4}\|e^{\gamma t \mathcal{R}_3^l}P_jh^l\|_{L_t^4L_y^{\infty}}\lesssim|\nu 2^{2j}|^{-\f14}2^{\f34j}|l|^{-\f14}\|h^l\|_{L^{2}(\R)},\quad j >j_0+4.
    \end{align*}
    On the other hand, we have
    \begin{align*}
        &\|e^{\gamma t \mathcal{R}_3^l+\nu t(\partial_y^2-l^2)}P_jh^l\|_{L_y^{\infty}}
        \lesssim
        2^{j/2}\|e^{\gamma t \mathcal{R}_3^l+\nu t(\partial_y^2-l^2)}P_jh^l\|_{L_y^{2}}
        \lesssim 
        2^{j/2}e^{-c\nu t2^{2j}}\|h^l\|_{L^{2}(\R)},\\
        &\|e^{\gamma t \mathcal{R}_3^l+\nu t(\partial_y^2-l^2)}P_jh^l\|_{L_t^2L_y^{\infty}}
        \lesssim
        \|2^{j/2}e^{-c\nu t2^{2j}}\|_{L_t^2}\|h^l\|_{L^{2}(\R)}
        \lesssim
        \nu^{-\f12}2^{-\f j2}\|h^l\|_{L^{2}(\R)}.
    \end{align*}
    In summary, we achieve that
    \begin{align*}
        &\|e^{\gamma t \mathcal{R}_3^l+\nu t(\partial_y^2-l^2)}P_jh^l\|_{L_t^2L_y^{\infty}}
        \lesssim
        \nu^{-\f13} |l|^{-\f16}(2^j|l|^{-1})^{\f16}\|h^l\|_{L^{2}(\R)},\quad j\leq j_0+4,\\
        &\|e^{\gamma t \mathcal{R}_3^l+\nu t(\partial_y^2-l^2)}P_jh^l\|_{L_t^2L_y^{\infty}}
        \lesssim
        \min\{\nu^{-\f14}
        2^{\f j4}|l|^{-\f14},\nu^{-\f12}2^{-\f j2}\}\|h^l\|_{L^{2}(\R)},\quad j >j_0+4.
    \end{align*}
    These two estimates above yield that
    \begin{align*}
        &\|e^{\gamma t \mathcal{R}_3^l+\nu t(\partial_y^2-l^2)}h^l\|_{L_t^2L_y^{\infty}}\leq  \sum_{j\in\Z}\|e^{\gamma t \mathcal{R}_3^l+\nu t(\partial_y^2-l^2)}P_jh^l\|_{L_t^2L_y^{\infty}}\\ 
        \lesssim& \sum_{j\leq j_0+4}\nu^{-\f13} |l|^{-\f16}(2^j|l|^{-1})^{\f16}\|h^l\|_{L^{2}(\R)}
        +\sum_{j\in\Z}\min\{\nu^{-\f14}
        2^{\f j4}|l|^{-\f14},\nu^{-\f12}2^{-\f j2}\}\|h^l\|_{L^{2}(\R)}\\ 
        \lesssim&~ \nu^{-\f13} |l|^{-\f16}\|h^l\|_{L^{2}(\R)}
        +\sum_{2^j\leq|l/\nu|^{1/3}}\nu^{-\f14}
        2^{\f j4}|l|^{-\f14}\|h^l\|_{L^{2}(\R)}
        +\sum_{2^j\geq|l/\nu|^{1/3}}
      \nu^{-\f12}2^{-\f j2}\|h^l\|_{L^{2}(\R)}\\ 
      \lesssim& ~\nu^{-\f13} |l|^{-\f16}\|h^l\|_{L^{2}(\R)}.
    \end{align*}
    Using the above inequality with $\nu$ replaced by $\nu/2$, it is easy to observe that
    \begin{align*}
        &\|e^{\gamma t \mathcal{R}_3^l+\nu t(\partial_y^2-l^2)}h^l\|_{L_t^1L_y^{\infty}}
        \leq 
        \|e^{-\nu t l^2/2}\|_{L_t^2}\|e^{\gamma t \mathcal{R}_3^l+\nu t(\partial_y^2-l^2)/2}h^l\|_{L_t^2L_y^{\infty}} 
        \lesssim
        |\nu l^2|^{-\f12}\nu^{-\f13} |l|^{-\f16}\|h^l\|_{L^{2}(\R)},
    \end{align*} 
 which completes the proof.
 \end{proof}

\end{appendices}

\section*{Acknowledgments}
\normalfont
M. Li is supported by Postdoctoral Fellowship Program of CPSF under Grant No. GZB20240024 and China Postdoctoral Science Foundation under Grant No. 2024M760057 and No. 2025T180840. 
C. Wang is supported by NSF of China under Grant No. 12471189. 
Z. Zhang is supported by NSF of China under Grant No. 12288101.	C. Sun is supported by the ANR project ANR-24-CE40-3260, he would like to thank Peking University for its warm hospitality during his visit.

\end{document}